\documentclass[a4paper,10pt]{article}

\usepackage{geometry}
\usepackage{epsfig}
\usepackage{amsmath}
\usepackage{amssymb}
\usepackage{amsthm}
\usepackage{mathrsfs}
\usepackage{color}
\usepackage{amscd}
\usepackage{euscript}
\usepackage{stmaryrd}
\usepackage{hyperref}

\usepackage{amsthm}
\newtheorem{theorem}{Theorem}[section]
\newtheorem{definition}[theorem]{Definition}
\newtheorem{lemma}[theorem]{Lemma}
\newtheorem{proposition}[theorem]{Proposition}
\newtheorem{corollary}[theorem]{Corollary}
\newtheorem{remark}[theorem]{Remark}
\newtheorem{example}[theorem]{Example}
\newtheorem{examples}[theorem]{Examples}
\newtheorem{assumption}[theorem]{Assumption}

\usepackage{amsfonts}

\newcommand{\oo}{{\mathbb{O}}}
\newcommand{\hh}{{\mathbb{H}}}
\newcommand{\cc}{{\mathbb{C}}}
\newcommand{\rr}{{\mathbb{R}}}
\newcommand{\zz}{{\mathbb{Z}}}
\newcommand{\nn}{{\mathbb{N}}}
\newcommand{\D}{\mathbb{D}}

\newcommand{\s}{{\mathbb{S}}}

\newcommand{\sr}{\mathcal{SR}}
\newcommand{\I}{\mathcal{I}}
\renewcommand{\L}{\mathcal{L}}
\newcommand{\R}{\mathcal{R}}
\renewcommand{\P}{\mathcal{P}}
\newcommand{\T}{\mathcal{T}}
\newcommand{\B}{\mathcal{B}}

\newcommand{\torus}{\mathbb{T}}
\newcommand{\mon}{\mathrm{Mon}}
\newcommand{\reg}{\operatorname{Reg}}
\newcommand{\slice}{\mathcal{S}}
\renewcommand{\ker}{\operatorname{Ker}}

\newcommand{\h}{{\bf h}}
\renewcommand{\k}{{\bf k}}
\newcommand{\F}{\mathscr{F}}

\newcommand{\debar}{\overline{\partial}}

\newcommand\re{\operatorname{Re}}
\newcommand\im{\operatorname{Im}}
\newcommand{\Span}{\operatorname{Span}}

\title{{\bf  A unified theory of regular functions \\ of a hypercomplex variable}}

\author{Riccardo Ghiloni\\
\small Dipartimento di Matematica, Universit\`a di Trento\\ 
\small Via Sommarive 14, I-38123 Povo Trento, Italy\\
\small riccardo.ghiloni@unitn.it\\
\and
Caterina Stoppato
\\ 
\small Dipartimento di Matematica e Informatica ``U. Dini'', Universit\`a di Firenze \\
\small Viale Morgagni 67/A, I-50134 Firenze, Italy\\
\small caterina.stoppato@unifi.it}

\date{  }

\begin{document}

\maketitle


\begin{abstract}
This work proposes a unified theory of regularity in one hypercomplex variable: the theory of \emph{$T$-regular} functions. In the special case of quaternion-valued functions of one quaternionic variable, this unified theory comprises Fueter-regular functions, slice-regular functions and a recently-discovered function class. In the special case of Clifford-valued functions of one paravector variable, it encompasses monogenic functions, slice-monogenic functions, generalized partial-slice monogenic functions, and a variety of function classes not yet considered in literature. For $T$-regular functions over an associative $*$-algebra, this work provides integral formulas, series expansions, an Identity Principle, a Maximum Modulus Principle and a Representation Formula. It also proves some foundational results about $T$-regular functions over an alternative but nonassociative $*$-algebra, such as the real algebra of octonions.
\end{abstract}


{\footnotesize
\section*{Contents}
\contentsline {section}{\numberline {1}Introduction}{2}{section.1}%
\contentsline {section}{\numberline {2}Hypercomplex subspaces and monogenic functions}{3}{section.2}%
\contentsline {subsection}{\numberline {2.1}Alternative real $*$-algebras}{4}{subsection.2.1}%
\contentsline {subsection}{\numberline {2.2}Hypercomplex subspaces}{9}{subsection.2.2}%
\contentsline {subsection}{\numberline {2.3}Monogenic functions on hypercomplex subspaces}{12}{subsection.2.3}%
\contentsline {section}{\numberline {3}Properties of monogenic functions on hypercomplex subspaces}{13}{section.3}%
\contentsline {subsection}{\numberline {3.1}Monogenic polynomial maps on a hypercomplex subspace}{13}{subsection.3.1}%
\contentsline {subsection}{\numberline {3.2}Integral representation of functions on a hypercomplex subspace}{15}{subsection.3.2}%
\contentsline {subsection}{\numberline {3.3}Properties of the reproducing kernel and harmonicity}{17}{subsection.3.3}%
\contentsline {subsection}{\numberline {3.4}Series expansions of monogenic functions on a hypercomplex subspace}{18}{subsection.3.4}%
\contentsline {section}{\numberline {4}Regularity in hypercomplex subspaces}{19}{section.4}%
\contentsline {subsection}{\numberline {4.1}$T$-fans}{20}{subsection.4.1}%
\contentsline {subsection}{\numberline {4.2}$T$-regularity}{21}{subsection.4.2}%
\contentsline {subsection}{\numberline {4.3}Integral representation}{25}{subsection.4.3}%
\contentsline {section}{\numberline {5}Polynomial regular functions}{25}{section.5}%
\contentsline {subsection}{\numberline {5.1}The polynomial functions $\mathcal {T}_{\bf k}$}{26}{subsection.5.1}%
\contentsline {subsection}{\numberline {5.2}Adapted partial derivatives}{28}{subsection.5.2}%
\contentsline {subsection}{\numberline {5.3}Properties of polynomial functions $\mathcal {T}_{\bf k}$}{32}{subsection.5.3}%
\contentsline {section}{\numberline {6}$T$-functions and strongly $T$-regular functions}{37}{section.6}%
\contentsline {subsection}{\numberline {6.1}$T$-stem functions}{37}{subsection.6.1}%
\contentsline {subsection}{\numberline {6.2}$T$-functions and strongly $T$-regular functions}{38}{subsection.6.2}%
\contentsline {subsection}{\numberline {6.3}Mirror $T$-stem functions}{40}{subsection.6.3}%
\contentsline {section}{\numberline {7}Series expansion and representation formula}{42}{section.7}%
\contentsline {subsection}{\numberline {7.1}Series expansion}{43}{subsection.7.1}%
\contentsline {subsection}{\numberline {7.2}Representation formula on $T$-symmetric $T$-slice domains}{45}{subsection.7.2}%
\contentsline {section}{\numberline {8}Foundations for the nonassociative theory}{49}{section.8}%
\contentsline {section}{Appendix}{53}{page.53}%
\contentsline {section}{Acknowledgements}{74}{page.74}%
\contentsline {section}{References}{74}{page.74}%

}

\section{Introduction}

For several centuries now, complex analysis has been an attractive field of research. Its richness, already in one variable, has pushed scholars to seek analogs in higher-dimensional settings. Besides complex functions of several variables, functions of one variable over real algebras have been extensively studied: this gave rise to a new area of research, known as hypercomplex analysis. Some examples are: Fueter's theory of quaternionic regular functions~\cite{fueter1,fueter2,sudbery}; the celebrated theory of monogenic functions over Clifford algebras~\cite{librosommen,librocnops,librogurlebeck2}; Gentili and Struppa's theory of quaternionic slice-regular functions~\cite{librospringer2,cras,advances}; and Colombo, Sabadini and Struppa's theory of slice-monogenic functions over Clifford algebras~\cite{israel,librodaniele2}. Hypercomplex analysis is not limited to the associative setting: for instance, octonionic function theories have been introduced in~\cite{sce,rocky} and studied in subsequent works. It is also not limited to one variable, but well-developed also in several variables, see~\cite{librodaniele,gpseveral,perticiseveral} and references therein. While each of the aforementioned theories has been successfully developed and usefully applied in other areas of mathematics and physics, hypercomplex analysis ended up fragmenting while it grew in scope and importance. Even theories based on the same differential operator, such as Fueter's theory and monogenic function theory, traditionally required separate, algebra-specific, presentations.

The purpose of this work is to reduce the fragmentation of hypercomplex analysis, offering a unified approach at two different levels. At a first level, we unify the treatment of function classes defined as kernels of Cauchy-Riemann operators, such as Fueter's theory and monogenic function theory. At a second, deeper, level, we develop a unified theory of regularity in one hypercomplex variable encompassing Fueter-regularity, slice-regularity, monogenicity, slice-monogenicity and even examples not yet considered in literature. Surprisingly, an example of this kind is already available in an associative lower-dimensional algebra such as the real algebra of quaternions: it is presented and studied in some detail in~\cite{unifiednotion}. The same article announced the general definitions of the concepts of \emph{monogenic function on a hypercomplex subspace} and of \emph{$T$-regular function}, upon which the present work is based. Independently, the work~\cite{xsgeneralizedpartialslice} (see also~\cite{xsannouncement}) developed the notion of \emph{generalized partial-slice monogenic function}, which is a sub-case of the notion of Clifford-valued $T$-regular function on the paravector subspace. For generalized partial-slice monogenic functions, the same authors proved a version of the Fueter-Sce Theorem in~\cite{xsfuetersce}.

We provide here, for $T$-regular functions over a general associative $*$-algebra $A$, integral and series representations, an Identity Principle and a Representation Formula valid under specific hypotheses on the domains. Some important tools to prove these results are the properties that monogenic functions on a hypercomplex subspace of $A$ share with classical Clifford monogenic functions. However, the proofs of the properties of $T$-regular functions require original ideas that do not follow the lines of any previously-known function theory. As a final addition, we provide foundations for a theory of $T$-regular functions in the alternative but nonassociative setting. The paper is organized as follows.

Section~\ref{sec:monogenic} is devoted to preliminaries. Alternative real $*$-algebras are covered in Subsection~\ref{subsec:*-algebras}. Subsection~\ref{subsec:hypercomplexsubspaces} covers the concept of hypercomplex subspace, defined in~\cite{perotticr}. The notions of Cauchy-Riemann operator on a hypercomplex subspace (from~\cite{perotticr}) and of monogenic function on a hypercomplex subspace (from~\cite{unifiednotion}) are recalled in Subsection~\ref{subsec:monogenic}.

Section~\ref{sec:monogenicproperties} studies monogenic functions on hypercomplex subspaces, in the associative case. Subsection~\ref{subsec:monogenicpolynomial} presents a family of polynomial examples, which turn out to generate all polynomial monogenic functions. Integral representations and a reproducing kernel are provided in Subsection~\ref{subsec:monogenicintegral}. Subsection~\ref{subsec:reproducingkernel} studies the reproducing kernel and establishes that monogenic functions on hypercomplex subspaces are harmonic functions. Subsection~\ref{subsec:seriesexpansionmonogenic} provides series expansions. These results resonate with classical Clifford monogenic function theory and their proofs are postponed to the final Appendix.

Section~\ref{sec:hypercomplexregularity} concerns $T$-regular functions. For each list of steps $T$, Subsection~\ref{subsec:Tfans} recalls the concepts of $T$-fan, mirror and $T$-torus given in~\cite{unifiednotion}. The definition of $T$-regular function, as well as the concepts of $T$-slice domain and $T$-symmetric set, are recalled in Subsection~\ref{subsec:Tregularity}. In the associative case, Subsection~\ref{subsec:integralrepresentation} provides integral representations of $T$-regular functions: a Cauchy Formula and a Mean Value Formula.

Section~\ref{sec:polynomials} is entirely devoted to $T$-regular functions over an associative algebra. For each $k\in\nn$, a finite set $\F_k$ of polynomial $T$-regular functions is constructed in Subsection~\ref{subsec:T_k}. The elements of $\F_k$ are related to the polynomial functions constructed in Subsection~\ref{subsec:monogenicpolynomial} in a highly nontrivial fashion. After the construction of adapted partial derivatives in Subsection~\ref{subsec:partialderivatives}, $\F_k$ is proven to generate all $k$-homogenous polynomial $T$-regular functions in Subsection~\ref{subsec:propertiespolynomial}. In the same subsection, $\F_1$ is used to determine for which other lists of steps $\widehat{T}$ the class of $\widehat{T}$-regular functions coincides with the class of $T$-regular functions.

Section~\ref{sec:Tfunctions} concerns possible symmetries in $T$-regular functions. The concept of $T$-stem function defined in~\cite{unifiednotion} is recalled in Subsection~\ref{subsec:Tstem}. The definitions of $T$-function and strongly $T$-regular function are recalled and studied, in the associative case, in Subsection~\ref{subsec:Tfunctions}. Proofs of most results therein are postponed to Section~\ref{sec:nonassociative}. Subsection~\ref{subsec:mirrorTstem} defines and studies, still in the associative case, the concept of mirror $T$-stem function, which is a technical preparation for the subsequent section.

Section~\ref{sec:seriesexpansion} studies $T$-regular functions over an associative algebra more in depth. Subsection~\ref{subsec:seriesexpansion} provides series expansions for any $T$-regular function on an open ball centered at a point of the mirror. As a consequence, an Identity Principle for $T$-regular functions on $T$-slice domains is established and used, in turn, to prove a Maximum Modulus Principle. In Subsection~\ref{subsec:representationformula}, $T$-regular functions on a $T$-symmetric $T$-slice domain are proven to be automatically strongly $T$-regular, with the so-called Representation Formula.

Section~\ref{sec:nonassociative} studies the notions of $T$-function and of strongly $T$-regular function in full generality, without assuming the algebra considered to be associative.

The final Appendix comprises the proofs of all properties of monogenic functions on hypercomplex subspaces stated in Section~\ref{sec:monogenicproperties}.


\section{Hypercomplex subspaces and monogenic functions}\label{sec:monogenic}

Complex and hypercomplex analysis traditionally study classes of \emph{monogenic} functions from $\cc,\hh$ or $\oo$ to itself, or functions from the space of paravectors $\rr^{m+1}$ to the Clifford algebra $C\ell(0,m)$. These setups have mostly been treated separately in literature because of the different natures of the algebras considered. In this section, we present an approach that allows to treat all these cases, and further cases, at once.


\subsection{Alternative real $*$-algebras}\label{subsec:*-algebras}

\begin{assumption}
We assume $(A,+,\cdot,^c)$ to be an alternative real $*$-algebra of finite dimension. Additionally, we endow $A$ (whence all its real vector subspaces) with the natural topology and differential structure as a real vector space.
\end{assumption}

We recall that a real $*$-algebra of finite dimension is a finite-dimensional $\rr$-vector space endowed with an $\rr$-bilinear multiplicative operation and with an involutive $\rr$-linear antiautomorphism $x\mapsto x^c$ (called $*$-involution). We recall that $A$ is alternative if, and only if, $x(xy)=x^2y,(xy)y=xy^2$ for all $x,y\in A$. This is automatically true if $A$ is associative. More details about nonassociative algebras can be found in~\cite{schafer}.

Function theory over $A$ has been extensively studied, especially in the following special cases.

\begin{examples}[Division algebras]
The $*$-algebras of complex numbers $\cc$, quaternions $\hh$ and octonions $\oo$ can be built from the real field $\rr$ by means of the so-called Cayley-Dickson construction:
\begin{itemize}
\item $\cc=\rr+i\rr$, $(\alpha+i\beta)(\gamma+i\delta)=\alpha\gamma-\beta\delta+i(\alpha\delta+\beta\gamma)$, $(\alpha+i\beta)^c=\alpha-i\beta\quad\forall\,\alpha,\beta,\gamma,\delta\in\rr$.
\item $\hh=\cc+j\cc$, $(\alpha+j\beta)(\gamma+j\delta)=\alpha\gamma-\delta\beta^c+j(\alpha^c\delta+\gamma\beta)$, $(\alpha+j\beta)^c=\alpha^c-j\beta\quad\forall\,\alpha,\beta,\gamma,\delta\in\cc$.
\item $\oo=\hh+\ell\hh$, $(\alpha+\ell \beta)(\gamma+\ell \delta)=\alpha\gamma-\delta\beta^c+\ell(\alpha^c\delta+\gamma\beta)$, $(\alpha+\ell \beta)^c=\alpha^c-\ell \beta\quad\forall\,\alpha,\beta,\gamma,\delta\in\hh$.
\end{itemize}
\end{examples}

Set $\nn^*:=\nn\setminus\{0\}$. For any $m\in\nn^*$, let $\mathscr{P}(m)$ denote the power set of $\{1,\ldots,m\}$. Furthermore: for all $K\in\mathscr{P}(m)$, let $|K|$ denote the cardinality of $K$.

\begin{examples}[Clifford algebras]
The Clifford algebra $C\ell(p,q)$ is the associative $*$-algebra constructed by taking the real vector space $\rr^{2^{m}}$ with $m=p+q$ and with the following conventions:
\begin{itemize}
\item $(e_K)_{K\in\mathscr{P}(n)}$ denotes the standard basis of $\rr^{2^{m}}$; if $K=\{k_1,\ldots,k_s\}$ with $1\leq k_1<\ldots<k_s\leq m$, then the element $e_K$ is also denoted as $e_{k_1\ldots k_s}$;
\item $e_\emptyset$ is defined to be the neutral element and also denoted as $1$;
\item $e_k^2:=1$ for all $k\in\{1,\ldots,p\}$ and $e_k^2:=-1$ for all $k\in\{p+1,\ldots,m\}$;
\item $1\leq k_1<\ldots<k_s\leq m$, then the product $e_{k_1}\cdots e_{k_s}$ is defined to be $e_{k_1\ldots k_s}$;
\item $e_he_k = -e_ke_h$ for all distinct $h,k\in\{1,\ldots,m\}$;
\item $e_K^c:=e_K$ if $|K|\equiv0,3 \mod 4$ and $e_K^c:=-e_K$ if $|K|\equiv1,2 \mod 4$.
\end{itemize}
The $*$-involution $x\mapsto x^c$ is called Clifford conjugation.
\end{examples}

For more details on these two examples and their history, we refer the reader to~\cite{ebbinghaus,librogurlebeck2}. Another example follows.

\begin{example}[Dual quaternions]
The associative $*$-algebra $\D\hh$ of dual quaternions can be defined as $\hh+\epsilon\hh$, where $(\alpha+\epsilon\beta)(\gamma+\epsilon\delta)=\alpha\gamma+\epsilon(\alpha\delta+\beta\gamma)$ and $(\alpha+\epsilon\beta)^c=\alpha^c+\epsilon \beta^c$ for all $\alpha,\beta\in\hh$. In particular, $\epsilon$ commutes with every element of $\D\hh$ and $\epsilon^2=0$.
\end{example}

On our alternative $*$-algebra $A$, we will use the notations $t(x):=x+x^c$ and $n(x):=xx^c$ for all $x\in A$ and call the elements of
\[\s_A:=\{x\in A: t(x)=0,n(x)=1\}\]
the \emph{imaginary units} of $A$.

\begin{assumption}
We assume $\s_A\neq\emptyset$.
\end{assumption}

Before proceeding to set up the domains of the $A$-valued functions we are going to study, it is useful to understand conjugation, $t$ and $n$ a bit more in detail.

\begin{definition}
A \emph{fitted basis} of $A$ is an ordered basis $(w_0,w_1,\ldots,w_d)$ of $A$ such that $w_s^c=\pm w_s$ for all $s\in\{0,\ldots,d\}$.
\end{definition}

\begin{remark}
Every element $w_s$ of a fitted basis is an eigenvector for $t$, with eigenvalue $2$ or $0$. Moreover, $n(w_s)=\pm w_s^2$.
\end{remark}

For $\nu,p\in\nn$, let $I_\nu$ denote the real $\nu\times\nu$ identity matrix and $0_{\nu,p}$ denote the real $\nu\times p$ zero matrix.

\begin{lemma}\label{lem:fitted}
Fix $m\geq1$. If $v_1,\ldots,v_m\in\s_A$ are linearly independent, then $(1,v_1,\ldots,v_m)$ can be completed to a fitted basis of $A$.
\end{lemma}

\begin{proof}
Since $1^c=1$, the element $1\in A$ is an eigenvector with eigenvalue $1$ for the $*$-involution $A\to A\,\ a\mapsto a^c$. For $s\in\{1,\ldots,m\}$, 
the hypothesis $v_s\in\s_A$ implies $0=t(v_s)=v_s+v_s^c$, whence $v_s^c=-v_s$: in other words, $v_s$ is an eigenvector with eigenvalue $-1$ for the $*$-involution.
Let $C$ denote the real matrix associated to the $*$-involution $A\to A\,\ a\mapsto a^c$ with respect to a fixed basis of $A$. Since $(a^c)^c=a$ for all $a\in A$, the equality $C^2=I_{d+1}$ holds true and the complex spectrum of $C^2$ is $\{1\}$. Thus, the complex spectrum of $C$ is $\{\pm1\}$. The generalized eigenspace relative to $\pm1$ is the eigenspace relative to $\pm1$, i.e., $\ker(C\mp I_{d+1})$, because the equality $C^2=I_{d+1}$ implies $(C\mp I_{d+1})^2=C^2\mp2 C+I_{d+1}=\mp2(C\mp I_{d+1})$. By Jordan's Theorem, there exist $p,\nu\in\nn$ with $p\geq1,\nu\geq m$ and $p+\nu=d+1$, as well as a basis $\B:=(w_0,w_1,\ldots,w_d)$ of $A$ with $w_0=1,w_p=v_1,\ldots,w_{p+m-1}=v_m$, such that
\[\begin{bmatrix}
I_p&0_{p,\nu}\\
0_{\nu,p}&-I_\nu
\end{bmatrix}\]
is the matrix associated to the $*$-involution $A\to A,\ a\mapsto a^c$ with respect to $\B$.
\end{proof}

We will need to endow $A$ with a Hilbert space structure, as follows.

\begin{definition}
Fix an ordered real vector basis $\B'=(v_0,v_1,\ldots,v_d)$ of $A$ with $v_0=1$. We denote by the symbols $\langle\cdot,\cdot\rangle=\langle\cdot,\cdot\rangle_{\B'}$ and $\Vert\cdot\Vert=\Vert\cdot\Vert_{\B'}$ the standard Euclidean scalar product and norm associated to $\B'$.
\end{definition}

In other words, we consider the real space isomorphism
\[L_{\B'}:\rr^{d+1}\to A\,,\quad L_{\B'}(x_0,\ldots,x_d)=\sum_{s=0}^dx_s\,v_s=\sum_{s=0}^dv_s\,x_s\]
and endow $A$ with the Hilbert space structure that makes $L_{\B'}$ a Hilbert space isomorphism. For future reference, we make the following remark.

\begin{remark}\label{rmk:endomorphisms}
There exists a (real linear) isomorphism $\mathcal{CON}:\rr^{d+1}\to\rr^{d+1}$ such that
\[\left(L_{\B'}\circ \mathcal{CON}\circ L_{\B'}^{-1}\right)(x)=x^c\]
for all $x\in A$. For every $a\in A$, there exist unique (real linear) endomorphisms $\L_a,\R_a:\rr^{d+1}\to\rr^{d+1}$ such that
\[\left(L_{\B'}\circ \L_a\circ L_{\B'}^{-1}\right)(x)=ax\,,\quad \left(L_{\B'}\circ \R_a\circ L_{\B'}^{-1}\right)(x)=xa\]
for all $x\in A$. Moreover: $\L_a$ is an isomorphism if, and only if, $a$ is not a left zero divisor in $A$; $\R_a$ is an isomorphism if, and only if, $a$ is not a right zero divisor in $A$.
\end{remark}

\begin{example}[Division algebras]
Let us assume $A=\oo$ (or $\hh$, or $\cc$). The standard basis $\B'=\{1,i,j,k,\ell,\ell i,\ell j,\ell k\}$ (or $\B'=\{1,i,j,k\}$, or $\B'=\{1,i\}$) of $\oo$ (or $\hh$, or $\cc$, respectively) is fitted. For every nonzero element $a\in A$, the isomorphism $\L_a$ is a conformal transformation: namely, a rotation about the origin composed with a dilation whose scaling factor is $\Vert a\Vert$. The same is true for $\R_a$.
\end{example}

\begin{example}[$C\ell(0,3)$]
The standard basis $\B'=(e_\emptyset,e_1,e_2,e_3,e_{12},e_{13},e_{23},e_{123})$ of the Clifford algebra $C\ell(0,3)$ is fitted. If we set $a:=\frac12(1+e_{123}),b:=\frac12(1-e_{123})$ in $C\ell(0,3)$, then $a^2=a$ and $ab=0$. The endomorphism $\L_a:\rr^8\to\rr^8$ has rank $4$ and the direct sum decomposition $\rr^8=\L_a(\rr^8)\oplus\ker(\L_a)$ corresponds to the decomposition $C\ell(0,3)=a\,C\ell(0,2)+b\,C\ell(0,2)$.
\end{example}

\begin{example}[Dual quaternions]
In $\D\hh=\hh+\epsilon\hh$, by direct inspection, the standard basis $\B'=(1,i,j,k,\epsilon,\epsilon i,\epsilon j,\epsilon k)$ is fitted. The endomorphism $\L_\epsilon:\rr^8\to\rr^8$ has rank $4$. Both the image $\L_\epsilon(\rr^8)$ and the kernel $\ker(\L_\epsilon)$ correspond to the $4$-subspace $\epsilon\hh$ of $\D\hh$.
\end{example}

We would like the Hilbert space structure we defined on $A$ to be as adapted to the $*$-algebra structure as possible.

\begin{definition}
Fix an ordered real vector basis $\B'=(v_0,v_1,\ldots,v_d)$ of $A$ with $v_0=1$. Consider the symmetric real bilinear form $\llbracket a,b\rrbracket=\llbracket a,b\rrbracket_{\B'}:=\frac12\langle t(ab^c),1\rangle_{\B'}$. If the matrix associated to $\llbracket\cdot,\cdot\rrbracket$ with respect to $\B'$ takes the form
\[\begin{bmatrix}
I_p&0_{p,\nu}&0_{p,\zeta}\\
0_{\nu,p}&-I_\nu&0_{\nu,\zeta}\\
0_{\zeta,p}&0_{\zeta,\nu}&0_{\zeta,\zeta}
\end{bmatrix}
\]
for some $p,\nu,\zeta\in\nn$, then $\B'$ is called an \emph{adapted basis} of $A$. We call $(p,\nu,\zeta)$ the \emph{signature} of $\B'$. An adapted basis $\B'$ with signature $(d+1,0,0)$ is called a \emph{distinguished basis} of $A$.
\end{definition}

\begin{remark}\label{rmk:distinguished}
Assume $\B'=(v_0,v_1,\ldots,v_d)$ to be an adapted basis of $A$ with signature $(p,\nu,\zeta)$. If $a=\sum_{s=0}^da_sv_s,b=\sum_{s=0}^db_sv_s\in A$, then
\[\llbracket a,b\rrbracket=\sum_{s=0}^{p-1}a_sb_s-\sum_{s=p}^{p+\nu-1}a_sb_s,\qquad \langle n(a),1\rangle=\llbracket a,a\rrbracket=\sum_{s=0}^{p-1}a_s^2-\sum_{s=p}^{p+\nu-1}a_s^2\leq\Vert a\Vert^2\,.\]
In particular, $\llbracket a,x\rrbracket=\langle a,x\rangle$ and $\llbracket x,x\rrbracket=\Vert x\Vert^2$ for all $x\in\Span(v_0,\ldots,v_{p-1})$.
The following are equivalent:
\begin{enumerate}
\item $\B'$ is a distinguished basis of $A$;
\item $\llbracket\cdot,\cdot\rrbracket$ is positive definite;
\item $\llbracket a,b\rrbracket=\langle a,b\rangle$ for all $a,b\in A$;
\item $\langle n(a),1\rangle=\llbracket a,a\rrbracket=\Vert a\Vert^2$ for all $a\in A$.
\end{enumerate}
\end{remark}

\begin{examples}[Division algebras]
Let $A\in\{\cc,\hh,\oo\}$. By direct inspection, the standard basis $\B'$ is a fitted distinguished basis of $A$. The functions $t,n:A\to A$ take values in $\rr$. We have $\langle a,1\rangle=\llbracket a,1\rrbracket=\frac12t(a)$, $\langle a,b\rangle=\llbracket a,b\rrbracket=\frac12t(ab^c)$ and $\Vert a\Vert^2=\llbracket a,a\rrbracket=n(a)$ for all elements $a,b\in A$.
\end{examples}

\begin{examples}[Clifford algebras]
The standard basis $\B'=(e_K)_{K\in\mathscr{P}(m)}$ of $C\ell(0,m)$ is a fitted distinguished basis of $C\ell(0,m)$, by direct inspection.

For $m\leq2$, we find the two division algebras $C\ell(0,1)\simeq\cc$ and $C\ell(0,2)\simeq\hh$ and the considerations made in the previous examples apply.

If, instead, $m\geq3$, for the element $a=1+e_{123}$ we have $a^c=a,a^2=2a$, whence $t(a)=2a=2\langle a,1\rangle+2e_{123}=2\llbracket a,1\rrbracket+2e_{123}$, $n(a)=2a=\Vert a\Vert^2+2e_{123}=\llbracket a,a\rrbracket+2e_{123}$.
\end{examples}

We now provide an example of an adapted basis that is not distinguished.

\begin{example}[Dual quaternions]
Within the $*$-algebra $\D\hh=\hh+\epsilon\hh$ of dual quaternions, we already remarked that the standard basis $\B'=(1,i,j,k,\epsilon,\epsilon i,\epsilon j,\epsilon k)$ is fitted. Let us prove that it is also adapted, but not distinguished.

For $a_1=p_1+\epsilon q_1,a_2=p_2+\epsilon q_2$ (with $p_1,p_2,q_1,q_2\in\hh$) we have $t(a_1) = t(p_1)+\epsilon t(q_1)$ and
\begin{align*}
t(a_1(a_2)^c)&=t\big(p_1(p_2)^c+\epsilon(p_1(q_2)^c+q_1(p_2)^c)\big)=2\langle p_1,p_2\rangle+2\epsilon(\langle p_1,q_2\rangle+\langle q_1,p_2\rangle)\,,\\
n(a_1)&=a_1a_1^c=n(p_1)+\epsilon t(p_1q_1^c)=\Vert p_1\Vert^2+2\epsilon\langle p_1,q_1\rangle\,.
\end{align*}
Thus,
\[\llbracket a_1,a_2\rrbracket=\langle p_1,p_2\rangle\,,\qquad\langle n(a_1),1\rangle=\Vert p_1\Vert^2\leq\Vert a_1\Vert^2\,.\]
It follows at once that the matrix associated to $\llbracket\cdot,\cdot\rrbracket$ with respect to $\B'$ is
\[\begin{bmatrix}
I_4&0_{4,4}\\
0_{4,4}&0_{4,4}
\end{bmatrix}
\]
and that $\B'$ is an adapted basis with signature $(4,0,4)$.
\end{example}

An adapted basis always exists, as shown in the next proposition.

\begin{proposition}\label{prop:adaptedbasis}
It is always possible to complete $1$ to an adapted basis $\B'$ of $A$.
\end{proposition}

\begin{proof}
Lemma~\ref{lem:fitted} guarantees that $A$ has a fitted basis $(1,w_1,\ldots,w_d)$. In particular, the hyperplane $\mathcal{H}:=\Span(w_1,\ldots,w_d)$ of $A$ is preserved by conjugation and $A$ decomposes into the direct sum $\rr\oplus\mathcal{H}$. We can now define the real linear map $\re:A\to\rr$ to act as the identity on $\rr$ and to vanish identically in $\mathcal{H}$. Moreover, $a\in\mathcal{H}$ implies $a^c\in\mathcal{H}$, whence $t(a)=a+a^c\in\mathcal{H}$ and $\re(t(a))=0$. We can define a symmetric bilinear form $\mathscr{B}:A\times A\to\rr$ by means of the formula $\mathscr{B}(a,b):=\frac12\re(t(ab^c))$ and let $(p,\nu,\zeta)$ denote the signature of $\mathscr{B}$. We remark that $\mathscr{B}(1,1)=\frac12\re(t(1))=1$ and that $\mathscr{B}(a,1)=\frac12\re(t(a))=0$ for all $a\in\mathcal{H}$. By Sylvester's theorem, we can complete $v_0=1$ to a basis $\B'=(v_0,v_1,\ldots,v_d)$ of $A$, with $v_1,\ldots,v_d\in\mathcal{H}$, so that, for all distinct $s,u\in\{0,1,\ldots,d\}$: $\mathscr{B}(v_s,v_u)=0$; $\mathscr{B}(v_s,v_s)=1$ if $0\leq s\leq p-1$; $\mathscr{B}(v_s,v_s)=-1$ if $p\leq s\leq p+\nu-1$; $\mathscr{B}(v_s,v_s)=0$ if $p+\nu\leq s\leq d$. Clearly, $p\geq1$. We now endow $A$ with the standard Euclidean scalar product $\langle\cdot,\cdot\rangle=\langle\cdot,\cdot\rangle_{\B'}$ and norm $\Vert\cdot\Vert=\Vert\cdot\Vert_{\B'}$ associated to $\B'$. The very definition of $\langle\cdot,\cdot\rangle$ implies that $\langle1,1\rangle=1=\re(1)$ and that $\langle v_s,1\rangle=0=\re(v_s)$ for all $s\in\{1,\ldots,d\}$, whence $\langle a,1\rangle=\re(a)$ for all $a\in A$. Recalling the definition $\llbracket a,b\rrbracket:=\frac12\langle t(ab^c),1\rangle$, we conclude that $\llbracket\cdot,\cdot\rrbracket=\mathscr{B}(\cdot,\cdot)$. Thus, $\B'$ is an adapted basis with signature $(p,\nu,\zeta)$, as desired.
\end{proof}

Although all standard bases in our examples are both fitted and adapted, we do not know an a priori reason why a general $*$-algebra should possess a basis that is both fitted and adapted.

The work~\cite{perotti} defined the \emph{quadratic cone} of $A$ as
\[Q_A:=\rr\cup\{x\in A\setminus\rr:t(x)\in\rr,n(x)\in\rr,4n(x)>t(x)^2\}\]
and proved the property
\[Q_A=\bigcup_{J\in\s_A}\cc_J\,,\]
where $\cc_J:=\rr+J\rr$ for all $J\in\s_A$. Since each $\cc_J$ is $*$-isomorphic to $\cc$, we can make the following remarks for every $x=\alpha+\beta J\in Q_A$ (with $\alpha,\beta\in\rr,J\in\s_A$): the conjugate $x^c=\alpha-\beta J$ belongs to $\cc_J\subset Q_A$; $t(x)=2\alpha\in\rr$; $n(x)=n(x^c)=\alpha^2+\beta^2$ is a positive real number; provided $x\neq0$, the element $x$ has a multiplicative inverse, namely $x^{-1}=n(x)^{-1}x^c=x^cn(x)^{-1}$, which still belongs to $Q_A$. In particular, $x\in Q_A\setminus\{0\}$ is neither a left nor a right zero divisor and the endomorphisms $\L_x,\R_x$ defined in Remark~\ref{rmk:endomorphisms} are isomorphisms. Our previous assumption $\s_A\neq\emptyset$ guarantees that $\rr\subsetneq Q_A$.

The next remark is a simple application of~\cite[Proposition 1.11]{gpsalgebra}. We recall that the \emph{associative nucleus} of $A$ is the real vector subspace of all elements $a\in A$ such that $a(xy)=(ax)y$ for all $x,y\in A$. The associative nucleus of $A$ includes the real axis $\rr$.

\begin{remark}\label{rmk:multiplicativen}
Assume the trace function $t:A\to A$ to take values in the associative nucleus of $A$, a fact which is always true if $A$ is associative. Take any $a,x\in A$. If $n(x)$ belongs to the commutative center of $A$, then $n(ax)=n(a)n(x)=n(x)n(a)$.
If $n(x^c)$ belongs to the commutative center of $A$, then $n((xa)^c)=n(a^cx^c)=n(a^c)n(x^c)=n(x^c)n(a^c)$. In particular: if $x\in Q_A$ (which implies $n(x)=n(x^c)\in\rr$), then
\begin{align*}
n(ax)&=n(a)n(x)=n(x)n(a)\,,\\
n((xa)^c)&=n(a^c)n(x)=n(x)n(a^c)\,.
\end{align*}
\end{remark}

\begin{examples}[Division algebras]
The complex field $\cc$, the skew field $\hh$ of real quaternions and the real algebra $\oo$ of octonions are alternative real $*$-algebras of dimensions $2,4,8$, respectively. The equalities $Q_\cc=\cc,Q_\hh=\hh,Q_\oo=\oo$ hold true. The sets $\s_\cc,\s_\hh,\s_\oo$ are, respectively, the $0,2,6$-dimensional unit spheres in the respective subspaces $t(x)=0$, each called the sphere of \emph{imaginary units}. For all elements $x,y$, we have $n(x^c)=n(x)$ and, since $t,n$ take values in $\rr$, $n(xy)=n(x)n(y)=n(y)n(x)=n(yx)$.
\end{examples}

\begin{examples}[Clifford algebras]
For any $m\in\nn^*$, consider the Clifford algebra $C\ell(0,m)$. The sets $\s_{C\ell(0,m)}$ and $Q_{C\ell(0,m)}$ are nested proper real algebraic subsets of $C\ell(0,m)$. While Remark~\ref{rmk:multiplicativen} holds true, if $m\geq4$ then $n(a\,b)$ does not equal $n(a)\,n(b)$ for general $a,b\in C\ell(0,m)$. For instance:  the elements $a:=\frac12(1+e_{123}),b:=\frac12(1+e_{1234})$ are preserved by conjugation and have $n(a)=a^2=a$ and $n(b)=b^2=b$, while $n(ab)=ab\,(ab)^c=a\,n(b)\,a=aba\neq ab=n(a)\,n(b)$ because
\begin{align*}
4ab&=(1+e_{123})(1+e_{1234})=1+e_4+e_{123}+e_{1234}\,,\\
4aba&=\frac12(1+e_4+e_{123}+e_{1234})(1+e_{123})\\
&=\frac12(1+e_4+e_{123}+e_{1234}+e_{123}-e_{1234}+1-e_4)=1+e_{123}\,.
\end{align*}
\end{examples}

\begin{example}[Dual quaternions]
Within the associative $*$-algebra $\D\hh$ of dual quaternions, $t(p+\epsilon q) = t(p)+\epsilon t(q)$ and $n(p+\epsilon q)=n(p)+\epsilon t(pq^c)$ for all $p,q \in \hh$. The set $\s_{\D\hh}=\{p+\epsilon q: p\in\s_\hh,q\in\im(\hh),\langle p,q\rangle=0\}$ is a $4$-dimensional algebraic subset of $\D\hh$, while $Q_{\D\hh}$ is a $6$-dimensional semialgebraic subset of $\D\hh$. By direct inspection, $n(a)=n(a^c)$ for all $a\in\D\hh$. Since $t$ and $n$ take values in the commutative center $\rr+\epsilon\rr$ of $\D\hh$, we conclude that $n(ab)=n(a)n(b)=n(b)n(a)=n(ba)$ for all $a,b\in\D\hh$.
\end{example}


\subsection{Hypercomplex subspaces}\label{subsec:hypercomplexsubspaces}

From now on, we will focus on specific subsets of the quadratic cone $Q_A$, constructed in~\cite[\S3]{perotticr} (cf.~\cite[Lemma 1.4]{volumeintegral}).

\begin{definition}
Let $M$ be a real vector subspace of our $*$-algebra $A$. An ordered real vector basis $(v_0,v_1,\ldots,v_m)$ of $M$ is called a \emph{hypercomplex basis} of $M$ if: $m\geq1$; $v_0=1$; $v_s\in\s_A$ and $v_sv_t=-v_tv_s$ for all distinct $s,t\in\{1,\ldots,m\}$. The subspace $M$ is called a \emph{hypercomplex subspace} of $A$ if $\rr\subsetneq M\subseteq Q_A$.
\end{definition}

Equivalently, a basis $(1,v_1,\ldots,v_m)$ is a hypercomplex basis if, and only if, $t(v_s)=0,n(v_s)=1$ and $t(v_sv_t^c)=0$ for all distinct $s,t\in\{1,\ldots,m\}$. Here and later in the paper, the double use of the letter $t$ as the trace function and as an index should bear no confusion, as the trace function is always followed either by an argument or by its domain. We remark that, for any $\ell\in\{1,\ldots,m\}$, the shortened ordered set $(v_0,v_1,\ldots,v_\ell)$ is a hypercomplex basis of its span. In the special case $m=1$ the hypercomplex subspace $M$ is always a $*$-subalgebra of $A$, isomorphic to the complex field. When $m\geq2$, the hypercomplex subspace $M$ is not, in general, a $*$-subalgebra of $A$. The next theorem was proven partly in~\cite[\S3]{perotticr}, partly in~\cite{unifiednotion}.

\begin{theorem}\label{thm:hypercomplexbasis}
Let $M$ be a real vector subspace of $A$. Then $M$ is a hypercomplex subspace of $A$ if, and only if, $M$ admits a hypercomplex basis $\B=(v_0,v_1,\ldots,v_m)$. If this is the case, if we complete $\B$ to a real vector basis $\B'=(v_0,v_1,\ldots,v_m,v_{m+1},\ldots,v_d)$ of $A$ and if we endow $A$ with $\langle\cdot,\cdot\rangle=\langle\cdot,\cdot\rangle_{\B'}$ and $\Vert\cdot\Vert=\Vert\cdot\Vert_{\B'}$, then
\begin{align}
&t(xy^c)=t(yx^c)=2\langle x,y\rangle\,,\label{eq:cliffordscalarproduct}\\
&n(x)=n(x^c)=\Vert x\Vert^2\,,\label{eq:cliffordnorm}
\end{align}
for all $x,y\in M$.
\end{theorem}

We can draw from Theorem~\ref{thm:hypercomplexbasis} a useful consequence.

\begin{corollary}\label{cor:compact}
Under the hypotheses of Theorem~\ref{thm:hypercomplexbasis}, the intersection $\s_A\cap M$ is a compact set: namely, the unit $(m-1)$-sphere centered at the origin in $\Span(v_1,\ldots,v_m)$, with respect to the norm $\Vert\cdot\Vert$.
\end{corollary}

Later in this work, we will need to control the norms of products of a specific form. We study this matter in the next remark and in the subsequent proposition.

\begin{remark}\label{rmk:norminequality}
Let us define $\omega=\omega_{\B,\B'}:=\max_{u\in M,v\in A,\Vert u\Vert=1=\Vert v\Vert}\Vert uv\Vert$. By construction, $\omega\geq\Vert 1\Vert=1$. Moreover, for all $x\in M$ and $a\in A$,
\[\Vert xa\Vert\leq\omega\,\Vert x\Vert\,\Vert a\Vert\,.\]
\end{remark}

\begin{proposition}\label{prop:norm}
Assume the trace function $t:A\to A$ to take values in the associative nucleus of $A$, which is always true if $A$ is associative. Under the hypotheses of Theorem~\ref{thm:hypercomplexbasis}, choose any $x,y\in M$ and any $a,b\in A$. Then
\[n(ax)=n(a)\Vert x\Vert^2=\Vert x\Vert^2 n(a),\quad n((xa)^c)=n(a^c)\Vert x\Vert^2=\Vert x\Vert^2 n(a^c)\]
and
\[n(xy) = \Vert x\Vert^2\Vert y\Vert^2 = \Vert y\Vert^2 \Vert x\Vert^2=n((xy)^c)\,.\]
If $\B'$ is a distinguished basis of $A$, it follows that $\Vert ax\Vert=\Vert a\Vert \Vert x\Vert = \Vert x\Vert \Vert a\Vert$, $\Vert (xa)^c\Vert=\Vert a^c\Vert \Vert x\Vert = \Vert x\Vert \Vert a^c\Vert$ and $\Vert xy\Vert = \Vert x\Vert\,\Vert y\Vert = \Vert y\Vert\,\Vert x\Vert= \Vert(xy)^c\Vert$.

If $\B'$ is fitted, then $\Vert a\Vert=\Vert a^c\Vert$. If $\B'$ is fitted and adapted, then $\llbracket a,x\rrbracket=\langle a,x\rangle$  and $\llbracket x,x\rrbracket=\Vert x\Vert^2$. Finally: if $\B'$ is a fitted distinguished basis of $A$, then
\begin{equation}\label{eq:cliffordmultiplicativenorm}
\Vert ax\Vert=\Vert a\Vert\,\Vert x\Vert=\Vert x\Vert\,\Vert a\Vert=\Vert xa\Vert\,,
\end{equation}
whence $\omega_{\B,\B'}=1$.
\end{proposition}

\begin{proof}
The first two equalities are applications of Remark~\ref{rmk:multiplicativen}, where we take into account Theorem~\ref{thm:hypercomplexbasis}. If, moreover, $\B'$ is a distinguished basis of $A$, then Remark~\ref{rmk:distinguished} guarantees that  $\Vert ax\Vert^2=\langle n(ax),1\rangle$, $\Vert a\Vert^2=\langle n(a),1\rangle$, $\Vert (xa)^c\Vert^2=\langle n((xa)^c),1\rangle$, $\Vert a^c\Vert^2=\langle n(a^c),1\rangle$, $\Vert xy\Vert^2=\langle n(xy),1\rangle$, and $\Vert(xy)^c\Vert^2=\langle n((xy)^c),1\rangle$.

If $\B'$ is fitted, then the $*$-involution $p\mapsto p^c$ preserves the norm $\Vert\cdot\Vert=\Vert\cdot\Vert_{\B'}$ because it maps every element of $\B'$ either into itself or into its opposite. Now assume $\B'$, which includes as its first $m+1$ elements the elements of the hypercomplex basis $\B$, to be fitted and adapted. Taking into account that $x\in\Span(\B)$, Remark~\ref{rmk:distinguished} guarantees that $\llbracket a,x\rrbracket=\langle a,x\rangle$ and $\llbracket x,x\rrbracket=\Vert x\Vert^2$. Finally: if a distinguished $\B'$ is also fitted, then we can substitute $\Vert xa\Vert$ for $\Vert (xa)^c\Vert$ and $\Vert a\Vert$ for $\Vert a^c\Vert$.
\end{proof}

\begin{example}[Paravectors]\label{ex:paravectors}
The space of paravectors $\rr^{m+1}$ is a hypercomplex subspace of the Clifford algebra $C\ell(0,m)$, with hypercomplex basis $\B=(e_\emptyset,e_1,\ldots,e_m)$. We complete $\B$ to the standard basis $\B'=(e_K)_{K\in\mathscr{P}(m)}$ of $C\ell(0,m)$, which is fitted and distinguished. Thus, Theorem~\ref{thm:hypercomplexbasis} and Proposition~\ref{prop:norm} yield that equalities~\eqref{eq:cliffordscalarproduct},~\eqref{eq:cliffordnorm} and~\eqref{eq:cliffordmultiplicativenorm} hold true for all $a\in C\ell(0,m),x,y\in\rr^{m+1}$. In particular, $\omega_{\B,\B'}=1$.

On the other hand, for $m\geq3$, the norm $\Vert\cdot\Vert$ is not multiplicative over general elements of $C\ell(0,m)$. For instance: the elements $a=1+e_{123}$ and $b=1-e_{123}$ have $ab=0$, whence $\Vert ab\Vert=0\neq2=\Vert a\Vert\Vert b\Vert$.

For every $m\geq2$, an example of basis $\B''$ of $C\ell(0,m)$ that is fitted but not adapted can be obtained from $\B'$ by substituting $\mu e_{12}$ for $e_{12}$ (for some $\mu$ such that $0<\mu<1$). Indeed, in this case $n(\mu e_{12})=\mu^2\not\in\{1,0,-1\}$. Moreover, in this case we have $\omega_{\B,\B''}\geq\mu^{-1}>1$ because $\Vert e_1e_2\Vert_{\B''}=\Vert e_{12}\Vert_{\B''}=\mu^{-1}\Vert\mu e_{12}\Vert_{\B''}=\mu^{-1}$.
\end{example}

\begin{examples}[{\cite[Example 1.15]{gpsalgebra}}]\label{ex:svectors}
For every $h\in\{1,\ldots,m\}$ with $h\equiv1\,\mathrm{mod}\,4$, the set
\[V_h:=\left\{x_0+\sum_{1\leq k_1<\ldots<k_h\leq m}x_{k_1\ldots k_h}e_{k_1\ldots k_h} : x_0, x_{k_1\ldots k_h}\in\rr\right\}\]
is a hypercomplex subspace of $C\ell(0,m)$. It has hypercomplex basis $\B=(e_{k_1\ldots k_h})_{1\leq k_1<\ldots<k_h\leq m}$, whence $\dim V_h=\binom{m}{h}+1\geq\left(\frac{m}{h}\right)^h+1$. If we set $h(m):=4\lfloor \frac{m+2}8\rfloor+1$ (whence $\frac{m}2-2<h(m)\leq\frac{m}2+2$), then $\dim V_{h(m)}$ grows exponentially with $m$. Again, we can complete $\B$ to the fitted distinguished basis $\B'=(e_K)_{K\in\mathscr{P}(m)}$ of $C\ell(0,m)$. Equalities~\eqref{eq:cliffordscalarproduct},~\eqref{eq:cliffordnorm} and~\eqref{eq:cliffordmultiplicativenorm} hold true for all $a\in C\ell(0,m),x,y\in V_h$. Thus, $\omega_{\B,\B'}=1$.
\end{examples}

Further examples of hypercomplex subspaces can be constructed by means of the next lemma, also proven in~\cite{unifiednotion}.

\begin{lemma}\label{lem:equal2mod4}
Let $M$ be a real vector subspace of $A$ with a hypercomplex basis $\B=(v_0,v_1,\ldots,v_m)$ and set $\widehat v:=v_1\cdots v_m$. The ordered set $\widehat\B:=(v_0,v_1,\ldots,v_m,\widehat v)$ is a hypercomplex basis of $\widehat M:=\Span(v_0,v_1,\ldots,v_m,\widehat v)$ if, and only if, $m\equiv2\,\mathrm{mod}\,4$. If this is the case, then not only $M$ but also $\widehat M$ is a hypercomplex subspace of $A$.
\end{lemma}

Several applications of Lemma~\ref{lem:equal2mod4} follow.

\begin{example}\label{ex:2mod4}
A further example of hypercomplex subspace of $C\ell(0,m)$ is the space $W_h=\Span(e_\emptyset,e_1,e_2,\ldots,e_h,e_{12\ldots h})$, for any $h\leq m$ with $h\equiv2\,\mathrm{mod}\,4$. Once more, we can complete the hypercomplex basis $\B=(e_\emptyset,e_1,e_2,\ldots,e_h,e_{12\ldots h})$ to the fitted distinguished basis $\B'=(e_K)_{K\in\mathscr{P}(m)}$ of $C\ell(0,m)$. Equalities~\eqref{eq:cliffordscalarproduct},~\eqref{eq:cliffordnorm} and~\eqref{eq:cliffordmultiplicativenorm} hold true for all $a\in C\ell(0,m),x,y\in W_h$ and $\omega_{\B,\B'}=1$.
\end{example}

We single out the case $h=m=2$, as follows.

\begin{example}[Quaternions]\label{ex:quaternions}
Within the real algebra of quaternions $\hh=C\ell(0,2)$, the subspace of paravectors $\rr^{2+1}$ is a hypercomplex subspace, with hypercomplex basis $(e_\emptyset,e_1,e_2)$. The whole algebra $\hh$ is also a hypercomplex subspace of $\hh$, with hypercomplex basis $(e_\emptyset,e_1,e_2,e_{12})=(1,i,j,k)$. In all examples and statements concerning $\hh$, unless otherwise stated, we shall assume $\B=\B'=(1,i,j,k)$ and endow $\hh$ with its standard scalar product and norm. For all $x,y\in\hh$, we have $t(xy^c)=t(yx^c)=2\langle x,y\rangle$,
$n(x)=n(x^c)=\Vert x\Vert^2$, and $\Vert xy\Vert=\Vert x\Vert\,\Vert y\Vert=\Vert y\Vert\,\Vert x\Vert=\Vert yx\Vert$. Moreover, $\omega_{\B,\B'}=1$.
\end{example}

\begin{example}[Dual quaternions]
Within the $*$-algebra $\D\hh=\hh+\epsilon\hh$ of dual quaternions, two elements $v_1=p_1+\epsilon q_1,v_2=p_2+\epsilon q_2$ belong to $\s_{\D\hh}$ if, and only if, $p_1,p_2\in\s_\hh,q_1,q_2\in\im(\hh)$ and $\langle p_1,q_1\rangle=0=\langle p_2,q_2\rangle$. If this is the case, then $\mathcal{A}:=(1,v_1,v_2)$ is a hypercomplex basis of the $3$-space $\Span(1,v_1,v_2)$ if, and only if, $t(v_1v_2^c)=0$. This happens if, and only if, $\langle p_1,p_2\rangle=0$ and $\langle q_1,p_2\rangle=-\langle q_2,p_1\rangle$. If this is the case, Lemma~\ref{lem:equal2mod4} guarantees that $\widehat{\mathcal{A}}:=(1,v_1,v_2,v_1v_2)$ is a hypercomplex basis of the $4$-space $\Span(1,v_1,v_2,v_1v_2)$. Both $\mathcal{A}$ and $\widehat{\mathcal{A}}$ can be completed to the basis $\mathcal{A}'=(1,v_1,v_2,v_1v_2,\epsilon,\epsilon v_1,\epsilon v_2,\epsilon v_1v_2)$ of $\D\hh$.

For instance, if for some $\alpha,\beta\in\rr$ we set $v_1=i+\epsilon(\alpha j+\beta k),v_2=j+\epsilon(-\alpha i+\beta k)$ and \[v_3:=v_1v_2=ij+\epsilon(i(-\alpha i+\beta k)+(\alpha j+\beta k)j)=k-\beta\epsilon(i+j)\,,\]
then $\Span(1,v_1,v_2,v_3)$ is a hypercomplex subspace of $\D\hh$ and its basis $(1,v_1,v_2,v_3)$ can be completed to the basis $(1,v_1,v_2,v_3,\epsilon,\epsilon i,\epsilon j,\epsilon k)$ of $\D\hh$.

Choosing $\alpha=0=\beta$, we find that $\hh=\hh+\epsilon0$ is a hypercomplex subspace of $\D\hh$, whose hypercomplex basis $\B=(1,i,j,k)$ can be completed to the standard basis $\B'=(1,i,j,k,\epsilon,\epsilon i,\epsilon j,\epsilon k)$. While this basis is adapted but not distinguished, we still find that $\omega_{\B,\B'}=1$. Indeed, for $x=x+\epsilon 0,a=p+\epsilon q$ (with $x,p,q\in\hh$) and for $\Vert\cdot\Vert=\Vert\cdot\Vert_{\B'}$, we find that $xa=xp+\epsilon xq$ has $\Vert xa\Vert^2=\Vert xp\Vert^2+\Vert xq\Vert^2=\Vert x\Vert^2\Vert a\Vert^2$. On the other hand, if $\lambda$ is such that $\Vert a b\Vert\leq\lambda\Vert a\Vert\Vert b\Vert$ for all $a,b\in\D\hh$, then $\lambda\geq\frac{2}{\sqrt{3}}>1$. Indeed, the element $a=\sqrt{\frac23}+\epsilon i \sqrt{\frac13}$ has $\Vert a\Vert=1$ but $a^2=\frac23+\epsilon i\frac{2\sqrt{2}}3$ has $\Vert a^2\Vert=\sqrt{\frac49+\frac89}=\frac{2}{\sqrt{3}}=\frac{2}{\sqrt{3}}\Vert a\Vert^2$.
\end{example}

Examples of hypercomplex subspaces are available in nonassociative settings, too.

\begin{example}[Octonions]
$\cc,\hh,\oo$ are examples of hypercomplex subspaces of $\oo$. In any example concerning $\oo$, we shall assume $\B=\B'=(1,i,j,k,l,li,lj,lk)$ and endow $\oo$ with its standard scalar product and norm. Not only $\B'$ is a fitted distinguished basis of $\oo$: we also have $t(xy^c)=t(yx^c)=2\langle x,y\rangle$ and $n(x)=n(x^c)=\Vert x\Vert^2$ for all $x,y\in\oo$. Moreover, $\Vert xy\Vert=\Vert x\Vert\,\Vert y\Vert=\Vert y\Vert\,\Vert x\Vert=\Vert yx\Vert$ for all $x,y\in\oo$ and $\omega_{\B,\B'}=1$.
\end{example}


\subsection{Monogenic functions on hypercomplex subspaces}\label{subsec:monogenic}

For the present subsection, $M$ is a fixed hypercomplex subspace of the alternative $*$-algebra $A$, with a fixed hypercomplex basis $\B=(v_0,v_1,\ldots,v_m)$. Moreover, $\B$ is completed to a real vector basis $\B'=(v_0,v_1,\ldots,v_m,v_{m+1},\ldots,v_d)$ of $A$. We also fix a \emph{domain} $G$ in the hypercomplex subspace $M$, i.e., a nonempty connected open subset $G$ of $M$. We let the symbol $\overline{G}$ denote the closure of $G$ in $M$. Clearly, the real vector space $\mathscr{C}^1(G,A)$ of $\mathscr{C}^1$ functions $G\to A$ is both a left $A$-module and a right $A$-module. The same is true for the real vector space $\mathscr{C}^1(\overline{G},A)$, comprising restrictions to $\overline{G}$ of $\mathscr{C}^1$ functions from some open neighborhood of $\overline{G}$ to $A$. Following~\cite[Definition 2]{perotticr}, we give the next definition. In addition to the previously defined $L_{\B'}:\rr^{d+1}\to A$, we will use
\[L_{\B}:\rr^{m+1}\to M\,,\quad L_{\B}(x_0,\ldots,x_m)=\sum_{s=0}^mx_s\,v_s=\sum_{s=0}^mv_s\,x_s\,.\]

\begin{definition}\label{def:monogenic}
Let $\phi,\psi\in\mathscr{C}^1(G,A)$. For $s\in\{0,\ldots,m\}$, we define $\widetilde G:=L_\B^{-1}(G)$ and
\[\partial_s\phi=L_{\B'}\circ\left(\frac{\partial}{\partial x_s}\left(L_{\B'}^{-1}\circ \phi\circ (L_\B)_{|_{\widetilde G}}\right)\right)\circ (L_\B^{-1})_{|_G}\in\mathscr{C}^0(G,A)\,.\]
Moreover, we define $\debar_\B$ and $\partial_\B$ by means of the equalities
\begin{align*}
&\debar_\B \phi:=\sum_{s=0}^mv_s\,\partial_s\phi\,,&\partial_\B \phi:=\sum_{s=0}^mv_s^c\,\partial_s\phi=\partial_0\phi-\sum_{s=1}^mv_s\,\partial_s\phi\,,\\
&\psi\debar_\B:=\sum_{s=0}^m(\partial_s\psi)v_s\,,&\psi\partial_\B:=\sum_{s=0}^m(\partial_s\psi)v_s^c=\partial_0\psi-\sum_{s=1}^m(\partial_s\psi)v_s\,.
\end{align*}
The right $A$-submodule of those $\phi\in\mathscr{C}^1(G,A)$ such that $\debar_\B \phi\equiv0$ is called the \emph{left kernel of} $\debar_\B$. Its elements are called \emph{left-monogenic with respect to $\B$}. The left $A$-submodule of those $\psi\in\mathscr{C}^1(G,A)$ such that $\psi\debar_\B\equiv0$ is called the \emph{right kernel of} $\debar_\B$. Its elements are called \emph{right-monogenic with respect to $\B$}.

If $G$ is bounded and has a $\mathscr{C}^1$ boundary $\partial G$ and if $\phi,\psi\in\mathscr{C}^1(\overline{G},A)$, we similarly define $\debar_\B \phi,\partial_\B \phi,\psi\debar_\B,\psi\partial_\B\in\mathscr{C}^0(\overline{G},A)$ and call $\phi$ \emph{left-monogenic with respect to $\B$} if $\debar_\B \phi\equiv0$, $\psi$ \emph{right-monogenic with respect to $\B$} if $\psi\debar_\B\equiv0$.

The operator $\Delta_\B:\mathscr{C}^2(G,A)\to\mathscr{C}^0(G,A)$ is defined by the formula
\[\Delta_\B \phi:=\sum_{s=0}^m\partial_s(\partial_s\phi)\,.\]
The elements of the kernel of $\Delta_\B$ are termed \emph{harmonic with respect to $\B$}.
\end{definition}

\begin{remark}\label{rmk:incrementalratio}
For any $s\in\{0,\ldots,m\}$, we have
\[\partial_s\phi(x)=\lim_{\rr\ni \varepsilon\to0}\varepsilon^{-1}\left(\phi(x+\varepsilon v_s)-\phi(x)\right)\]
at all $x\in G$. In particular, the operator $\partial_s$ does not depend on the whole basis $\B'$ of $A$ chosen but only on the choice of $v_s$. As a consequence, the operators $\debar_\B,\partial_\B,\Delta_\B$ do not depend on the whole basis $\B'$ of $A$ chosen, but only on the basis $\B$ of $M$ chosen.
\end{remark}

Our notation $\Delta_\B$ is consistent with~\cite[\S3]{perotticr}, while our $\debar_\B,\partial_\B$ are twice the operators with the same symbols constructed in~\cite[\S3]{perotticr}. In accordance to~\cite[Proposition 5 (b)]{perotticr}, we make the next remark.

\begin{remark}\label{rmk:harmonic}
The equalities $\Delta_\B\phi=\partial_\B\debar_\B \phi=\debar_\B\partial_\B \phi=\phi\,\partial_\B\debar_\B=\phi\,\debar_\B\partial_\B$ hold true for all $\phi\in\mathscr{C}^{2}(G,A)$. In particular, every $\mathscr{C}^{2}$ function that is left- or right-monogenic with respect to $\B$ is automatically harmonic with respect to $\B$, whence real analytic.
\end{remark}

For future use, we give the next definition.

\begin{definition}
For any $u\in\nn$ and $\h=(h_1,\ldots,h_u)\in\nn^u$, we adopt the notations $|\h|:=\sum_{s=1}^uh_s$ and $\h!:=\prod_{s=1}^uh_s!$. Now, for any $\h=(h_0,\ldots,h_m)\in\nn^{m+1}$, we define the operator $\nabla_\B^\h:\mathscr{C}^{|\h|}(G,A)\to\mathscr{C}^0(G,A)$ by setting
\[\nabla_\B^\h \phi:=\partial_0^{h_0}(\partial_1^{h_1}(\ldots(\partial_m^{h_m}\phi)\ldots))\]
for all $\phi\in\mathscr{C}^{|\h|}(G,A)$.
\end{definition}


\section{Properties of monogenic functions on hypercomplex subspaces}\label{sec:monogenicproperties}

This section studies left-monogenic functions on hypercomplex subspaces of an associative $*$-algebra $A$. The associativity assumption will be mentioned explicitly in all definitions and results in this section, because later in the paper we will drop it and go back to general alternative $*$-algebras. A motivation for the associativity assumption is keeping our presentation simple. Moreover, this assumption makes the properties stated here the perfect analogs over hypercomplex subspaces of the results collected in the monograph~\cite{librogurlebeck2}. Thus, the reader with experience in Clifford monogenic function theory will be easily convinced that these properties hold true and feel free to move swiftly to the next section. All proofs relative to the present section are postponed to the final Appendix.

Within our associative $*$-algebra $A$, throughout the present section fix a hypercomplex subspace $M$, with a fixed hypercomplex basis $\B=(v_0,v_1,\ldots,v_m)$. Moreover, $\B$ is completed to a real vector basis $\B'=(v_0,v_1,\ldots,v_m,v_{m+1},\ldots,v_d)$ of $A$ and $A$ is endowed with the standard Euclidean scalar product $\langle\cdot,\cdot\rangle$ and norm $\Vert\cdot\Vert$ associated to $\B'$.


\subsection{Monogenic polynomial maps on a hypercomplex subspace}\label{subsec:monogenicpolynomial}

This subsection is devoted to polynomial left-monogenic functions. 
The basic examples of polynomial left-monogenic functions are the hypercomplex analogs of Fueter variables and Fueter polynomial functions. We perform the same construction on any hypercomplex subspace. In particular, we overcome the traditional distinction between quaternionic and Clifford  Fueter polynomials.

\begin{definition}
Assume $A$ to be associative. For $s\in\{1,\ldots,m\}$, we define the $s$-th \emph{hypercomplex Fueter variable} as
\[\zeta_s=\zeta_s^\B:=x_s-x_0v_s\,.\]
Let us consider the elements $\epsilon_1=(1,0,\ldots,0),\epsilon_2=(0,1,\ldots,0),\ldots,\epsilon_m=(0,0,\ldots,1)$ of $\nn^m$. For all $\k=(k_1,\ldots,k_m)\in\zz^m$, we define the \emph{hypercomplex Fueter polynomial function} $\P_\k^\B:M\to A$ so that the following formulas hold true for $x\in M$:
\begin{align*}
&\P_\k^\B:\equiv0&&\mathrm{\ if\ }\k\not\in\nn^m\\
&\P_\k^\B:\equiv1&&\mathrm{\ if\ }\k=(0,\ldots,0)\\
&|\k|\,\P_\k^\B(x):=\sum_{s=1}^mk_s\,\P_{\k-\epsilon_s}^\B(x)\,\zeta_s^\B&&\mathrm{\ if\ }\k\in\nn^m\setminus\{(0,\ldots,0)\}
\end{align*}
\end{definition}

\begin{example}
We have:
\begin{align*}
&\P_{(0,0,0\ldots,0)}^\B\equiv 1\\
&\P_{(1,0,0,\ldots,0)}^\B=\zeta_1,\ \P_{(0,1,0,\ldots,0)}^\B=\zeta_2,\ \P_{(0,0,1,\ldots,0)}^\B=\zeta_3,\ldots\\
&\P_{(1,1,0,\ldots,0)}^\B=\frac{1}{2}(\zeta_2\zeta_1+\zeta_1\zeta_2),\ \P_{(1,0,1,\ldots,0)}^\B=\frac{1}{2}(\zeta_3\zeta_1+\zeta_1\zeta_3),\ \ldots,\ \P_{(0,1,1,\ldots,0)}^\B=\frac{1}{2}(\zeta_3\zeta_2+\zeta_2\zeta_3),\ldots\\
&\P_{(1,1,1,0,\ldots,0)}^\B=\frac{1}{6}(\zeta_3\zeta_2\zeta_1+\zeta_2\zeta_3\zeta_1+\zeta_3\zeta_1\zeta_2+\zeta_1\zeta_3\zeta_2+\zeta_2\zeta_1\zeta_3+\zeta_1\zeta_2\zeta_3),\ \ldots
\end{align*}
\end{example}

Before we even prove that the $\P_\k^\B$'s are left-monogenic with respect to $\B$, we wish to establish the following properties, which will be crucial later in the paper. We point out that
\[L_{\B}:\rr^{m+1}\to M\,,\quad L_{\B}(x_0,\ldots,x_m)=\sum_{s=0}^mx_s\,v_s=\sum_{s=0}^mv_s\,x_s\]
is a Hilbert (sub)space isomorphism.

\begin{proposition}\label{prop:fueterpolynomialsindependofbasis}
Assume $A$ to be associative. For $\k=(k_1,\ldots,k_m)\in\zz^m$, the following properties hold true.
\begin{enumerate}
\item There exists a map $p_\k=(p_\k^0,p_\k^1,\ldots,p_\k^m):\rr^{m+1}\to\rr^{m+1}$ such that $\P_\k^\B=L_\B\circ p_\k\circ L_\B^{-1}$ for any hypercomplex basis $\B$ of a hypercomplex subspace $M$ of $A$.
\item The equality $(k_u+1)\,p^s_\k=k_s\,p^u_{\k+\epsilon_u-\epsilon_s}$ holds true for all distinct $s,u\in\{1,\ldots,m\}$.
\item For all $x\in M$,
\begin{align}
&\sum_{s=1}^mk_s\,v_s\,\P_{\k-\epsilon_s}^\B(x)=\sum_{s=1}^mk_s\,\P_{\k-\epsilon_s}^\B(x)\,v_s\,,\label{eq:technical1}\\
&|\k|\,\P_\k^\B(x)=\sum_{s=1}^mk_s\,\zeta_s^\B\,\P_{\k-\epsilon_s}^\B(x)\label{eq:technical2}
\end{align}
\item The equality $\partial_s\P_\k^\B=k_s\,\P_{\k-\epsilon_s}^\B$ holds true for all $s\in\{1,\ldots,m\}$.
\end{enumerate}
\end{proposition}

\begin{remark}\label{rmk:propertiesfueterpolynomials}
Let $\k,\k'\in\zz^m$. As a consequence of property {\it 1}, $\P_\k^\B(M)\subseteq M$ and
\[\Vert\P_\k^\B(x_0v_0+x_1v_1+\ldots+x_mv_m)\Vert=\Vert p_\k(x_0,x_1,\ldots,x_m)\Vert_{\rr^{m+1}}\]
for all $(x_0,x_1,\ldots,x_m)\in\rr^{m+1}$. By property {\it 2}, if $\k=k_s\epsilon_s$, then $\P_\k^\B(M)\subseteq \rr+\rr v_s$. As a consequence of property {\it 4}, we obtain: $\nabla_\B^{(0,\k)}\P_\k^\B\equiv\k!$; if $k'_u>k_u$ for some $u\in\{1,\ldots,m\}$, then $\nabla_\B^{(0,\k')}\P_\k^\B\equiv0$.
\end{remark}

In the classical context of Clifford algebras, where $M=\rr^{m+1}$ is the paravector space within $A=C\ell(0,m)$, the property $\P_\k^\B(M)\subseteq M$ for all $\k\in\zz^m$ was proven in~\cite{malonekhabilitation}.

In analogy with~\cite[Proposition 9.21]{librogurlebeck2}, we find that the $\P_\k^\B$'s generate all polynomial left-monogenic functions.

\begin{definition}
Within the right $A$-module of functions $M\to A$ that are left-monogenic with respect to $\B$, for any $k\in\nn$, we define $U^\B_k$ to be the right $A$-submodule of those $P:M\to A$ such that $P(x_0+x_1v_1+\ldots+x_mv_m)$ is a $k$-homogeneous polynomial map in the real variables $x_0,x_1,\ldots,x_m$.
\end{definition}

\begin{proposition}\label{prop:basismonogenicpolynomials}
Assume $A$ to be associative and fix $k\in\nn$. Then $\left\{\P_\k^\B\right\}_{|\k|=k}$ is a right $A$-basis for $U^\B_k$. Namely, for all $P\in U^\B_k$, the equality
\begin{equation}\label{eq:monogenicpolynomialexpansion}
P(x)=\sum_{|\k|=k}\P_\k^\B(x)\,\frac{1}{\k!}\nabla_\B^{(0,\k)}P(0)
\end{equation}
holds true at all $x\in M$. 
\end{proposition}


\subsection{Integral representation of functions on a hypercomplex subspace}\label{subsec:monogenicintegral}

Let us fix a domain $G$ in the hypercomplex subspace $M$ of $A$. Our next aim is establishing integral representations for functions $G\to A$. We will work with respect to the coordinates $x_0,\ldots,x_m$, i.e., using the previously-defined Hilbert subspace isomorphism $L_\B:\rr^{m+1}\to M\subseteq A$. We will also use Hilbert space isomorphism $L_{\B'}:\rr^{d+1}\to A$. Differential forms were set up in~\cite[\S A.1]{librogurlebeck2} over a general associative algebra and not specifically over Clifford algebras. Recalling that we are assuming $A$ to be associative, we adopt the same setup. In particular, we set
\begin{align*}
d\sigma&=d\sigma_x=dx_0\wedge dx_1\wedge\ldots\wedge dx_m\,,\\
dx_s^*&=(-1)^sdx_0\wedge\ldots\wedge dx_{s-1}\wedge dx_{s+1}\wedge\ldots\wedge dx_m\,,\\
dx^*&=\sum_{s=0}^mv_s\,dx_s^*\,.
\end{align*}
Volume integration was set up in~\cite[Definition A.2.1]{librogurlebeck2} specifically over Clifford algebras. We now define and study it over our associative $*$-algebra $A$. We use, for all integrable functions $\phi_0,\ldots,\phi_d:G\to\rr$ the notation
\[\int_{G}\left(\phi_0,\ldots,\phi_d\right)\,d\sigma:=\left(\int_{G}\phi_0\,d\sigma,\ldots,\int_{G}\phi_d\,d\sigma\right)\in\rr^{d+1}\,.\]

\begin{definition}\label{def:volumeintegral}
Assume $A$ to be associative and fix a domain $G$ in the hypercomplex subspace $M$ of $A$. For $\phi:G\to A$, we set
\[\int_{G}\phi\,d\sigma:=L_{\B'}\left(\int_{G}L_{\B'}^{-1}\circ\phi\,d\sigma\right)\,.\]
In other words, if we have a decomposition $\phi=\sum_{s=0}^d\phi_s\,v_s=\sum_{s=0}^dv_s\,\phi_s$, where $\phi_0,\ldots,\phi_d:G\to\rr$ are real-valued integrable functions, we call $\phi$ \emph{integrable} and define its integral as
\[\int_{G}\phi\,d\sigma:=\sum_{s=0}^dv_s\int_{G}\phi_s\,d\sigma=\sum_{s=0}^d\left(\int_{G}\phi_s\,d\sigma\right)v_s\,.\]
\end{definition}

The integral in Definition~\ref{def:volumeintegral} has the properties described in the next proposition, which subsumes~\cite[Proposition A.2.2]{librogurlebeck2}.

\begin{proposition}\label{prop:propertiesintegral}
Assume $A$ to be associative and fix a domain $G$ in the hypercomplex subspace $M$ of $A$. The following properties hold true for all integrable $\phi,\psi:G\to A$, all $a,b\in A$ and all disjoint domains $G_1,G_2$ in $M$:
\begin{enumerate}
\item $G=G_1\cup G_2\Rightarrow\int_{G}\phi\,d\sigma=\int_{G_1}\phi\,d\sigma+\int_{G_2}\phi\,d\sigma$.
\item $\int_{G}(a\phi+b\psi)\,d\sigma=a\int_{G}\phi\,d\sigma+b\int_{G}\psi\,d\sigma$.
\item $\int_{G}(\phi a+\psi b)\,d\sigma=\left(\int_{G}\phi\,d\sigma\right)a+\left(\int_{G}\psi\,d\sigma\right)b$.
\item $\left(\int_{G}\phi\,d\sigma\right)^c=\int_{G}\phi^c\,d\sigma$.
\item $\Vert\int_{G}\phi\,d\sigma\Vert\leq\int_{G}\Vert \phi\Vert\,d\sigma$.
\end{enumerate}
\end{proposition}

Assume $G$ to be bounded and to have a $\mathscr{C}^1$ boundary $\partial G$. For any choice of $A$-valued $\mathscr{C}^1$ functions $\phi,\psi$ on an open neighborhood of $\partial G$, the work~\cite[\S A.2.1.3]{librogurlebeck2} defines and studies the integral $\int_{\partial G}\psi(x)\,dx^*\phi(x)$. For fixed $p\in M, R>0$, let us adopt the notations $B^{m+1}(p,R):=\{x\in M:\Vert x-p\Vert<R\}$ and $\overline{B}^{m+1}(p,R):=\{x\in M:\Vert x-p\Vert\leq R\}$. We recall from~\cite[Example A.2.17]{librogurlebeck2}:

\begin{remark}\label{rmk:sphere}
Assume $G=B^{m+1}(p,R)$. For $0\leq r<R, w\in\partial B^{m+1}(0,1)$ and $x=p+rw\in G$, we have $d\sigma=r^m\,dr\,|do_w|$ in $G$, where $|do_w|$ denotes the surface element of the unit sphere $\partial B^{m+1}(0,1)$. Similarly, for $x=p+Rw\in\partial G$, we have $dx^*=R^m\,w\,|do_w|$ on $\partial G$. 
\end{remark}

The integral $\int_{\partial B^{m+1}(0,1)}\phi(w)\,|do_w|$ is defined in~\cite[\S A.2.1.3]{librogurlebeck2}, too. For future use, we establish the following inequality.

\begin{lemma}\label{lem:surfaceintegral}
$\left\Vert\int_{\partial B^{m+1}(0,1)}\phi(w)\,|do_w|\right\Vert\leq\int_{\partial B^{m+1}(0,1)}\Vert\phi(w)\Vert\,|do_w|$.
\end{lemma}

The next result subsumes~\cite[Theorem A.2.21 and Theorem A.2.22]{librogurlebeck2}.

\begin{theorem}[Gauss]\label{thm:gauss}
Assume $A$ to be associative and fix a bounded domain $G$ in the hypercomplex subspace $M$ of $A$, with a $\mathscr{C}^1$ boundary $\partial G$. Then
\[\int_{\partial G}\psi\,dx^*\phi = \int_{G}\left((\psi\debar_\B)\,\phi+\psi\,(\debar_\B \phi)\right)\,d\sigma\]
for any $\phi,\psi\in\mathscr{C}^1(\overline{G},A)$.
\end{theorem}

We will soon plug into the Gauss theorem, in the role of $\psi$, the function described in the next definition and lemma (which generalizes~\cite[Proposition 7.7]{librogurlebeck2}).

\begin{definition}
The \emph{Cauchy kernel} of $M$ is the function $E_m:M\setminus\{0\}\to A$ defined by the formula
\[E_m(x):=\frac{1}{\sigma_m}\frac{x^c}{\Vert x\Vert^{m+1}}\,,\quad\sigma_m:=2\frac{\Gamma^{m+1}(\frac12)}{\Gamma(\frac{m+1}2)}\,.\]
\end{definition}

In the last definition, the letter $\Gamma$ denotes the gamma function and the number $\sigma_m$ is the (surface) volume of the unit $m$-sphere in $\rr^{m+1}$.

\begin{lemma}\label{lem:cauchykernel}
If we fix $x\in M$, then the function
\[M\setminus\{x\} \to A\,,\quad y\mapsto E_m(y-x)\]
is both left- and right-monogenic with respect to $\B$.
\end{lemma}

We are now ready for the announced integral representation, which subsumes~\cite[Theorem 7.8 and Theorem 7.9]{librogurlebeck2}.

\begin{theorem}[Borel-Pompeiu]\label{thm:borelpompeiu}
Assume $A$ to be associative and fix a bounded domain $G$ in the hypercomplex subspace $M$ of $A$, with a $\mathscr{C}^1$ boundary $\partial G$. If $\phi\in\mathscr{C}^1(\overline{G},A)$, then
\[\int_{\partial G}E_m(y-x)\,dy^*\,\phi(y) - \int_{G}E_m(y-x)\,\debar_\B \phi(y)\,d\sigma_y
=\left\{\begin{array}{ll}
\phi(x)&\mathrm{if\ } x\in G\\
0&\mathrm{if\ } x\in M\setminus\overline{G}\,.
\end{array}\right.\]
\end{theorem}

In the special case when $\phi$ is left-monogenic with respect to $\B$, Theorem~\ref{thm:borelpompeiu} takes the special form described in the next corollary (see~\cite[Theorem 7.12 and Theorem 7.13]{librogurlebeck2} for the Clifford and complex cases).

\begin{corollary}[Cauchy Formula]\label{cor:monogeniccauchy}
Assume $A$ to be associative and fix a bounded domain $G$ in the hypercomplex subspace $M$ of $A$, with a $\mathscr{C}^1$ boundary $\partial G$. If $\phi\in\mathscr{C}^1(\overline{G},A)$ is left-monogenic with respect to $\B$, then
\[\int_{\partial G}E_m(y-x)\,dy^*\,\phi(y) =\left\{\begin{array}{ll}
\phi(x)&\mathrm{if\ } x\in G\\
0&\mathrm{if\ } x\in M\setminus\overline{G}\,.
\end{array}\right.\]
\end{corollary}

We conclude this subsection with the following property, which subsumes~\cite[Corollary 7.31]{librogurlebeck2}.

\begin{proposition}[Mean value property]\label{prop:monogenicmeanvalue}
Assume $A$ to be associative and fix an open ball $B^{m+1}=B^{m+1}(x,R)$ in the hypercomplex subspace $M$ of $A$. If $\phi\in\mathscr{C}^1(\overline{B}^{m+1},A)$ is left-monogenic with respect to $\B$, then
\[\phi(x)=\frac1{\sigma_m}\int_{\partial B^{m+1}(0,1)}\phi(x+R w)\,|do_w|\,.\]
\end{proposition}

After some preliminary work in the next subsection, Corollary~\ref{cor:monogeniccauchy} will be the key ingredient to endow every function $\phi$ that is left-monogenic with respect to $\B$ with a series expansion in the forthcoming Subsection~\ref{subsec:seriesexpansionmonogenic}.


\subsection{Properties of the reproducing kernel and harmonicity}\label{subsec:reproducingkernel}

This subsection studies the reproducing kernel $E_m(y-x)$, whose role was fundamental in Theorem~\ref{thm:borelpompeiu} and in Corollary~\ref{cor:monogeniccauchy}, and uses it to establish that left-monogenic functions are harmonic and real analytic.

Our first aim is expanding $E_m(y-x)$ into series. We begin by recalling some standard terminology.

\begin{definition}
Let $\mathscr{E},\mathscr{E}'$ be finite-dimensional Euclidean spaces and let $\Lambda$ be an open subset of $\mathscr{E}'$. For any $\{f_k\}_{k\in\nn}\subset\mathscr{C}^0(\Lambda,\mathscr{E})$, we say that the function series $\sum_{k\in\nn}f_k$ is \emph{normally convergent} in $\Lambda$ if, for any compact subset $C$ of $\Lambda$, the number series $\sum_{k\in\nn}\max_C\Vert f_k\Vert_{\mathscr{E}}$ converges. If this is the case, we call the function $f:\Lambda\to\mathscr{E}$ with $f(x)=\sum_{k\in\nn}f_k(x)$ the \emph{sum of the series} $\sum_{k\in\nn}f_k$.
\end{definition}

The chosen ordering of $\nn$ plays no role in assessing the convergence of $\sum_{k\in\nn}\max_C\Vert f_k\Vert_{\mathscr{E}}$ because each term is nonnegative. The theory of Banach spaces also guarantees that the result of the sum $\sum_{k\in\nn}f_k(x)$ does not depend on the chosen ordering of $\nn$ and that $f\in\mathscr{C}^0(\Lambda,\mathscr{E})$. Our summands $f_k$ will mostly take the form $f_k=\sum_{|\k|=k}g_\k$, for given $\{g_\k\}_{\k\in\nn^m}\subset\mathscr{C}^0(\Lambda,\mathscr{E})$.

We now expand the reproducing kernel $E_m(y-x)=\frac{1}{\sigma_m}\frac{y^c-x^c}{\Vert y-x\Vert^{m+1}}$ into a normally convergent series in the next theorem.

\begin{theorem}\label{thm:taylorkernel}
There exists a family $\big\{q_\k\big\}_{\k\in\nn^m}$, where, for $|\k|=k$, $q_{\k}:M\setminus\{0\}\to M$ is a $(k+1)$-homogeneous polynomial function, such that
\[E_m(y-x)=\frac1{\sigma_m}\sum_{k\in\nn}\sum_{|\k|=k}\P_\k^\B(x)\frac{q_\k(y)}{\Vert y\Vert^{m+2k+1}}\]
for all $(x,y)\in\Lambda:=\{(x,y)\in M\times M : \Vert x\Vert<\Vert y\Vert\}$. Here, the series converges normally in $\Lambda$ because
\[\Big\Vert\sum_{|\k|=k}\P_\k^\B(x)\frac{q_\k(y)}{\Vert y\Vert^{m+2k+1}}\Big\Vert\leq\sqrt{2}\,\binom{k+m}{m}\,\,\Vert x\Vert^k\,\Vert y\Vert^{-m-k}\,.\]
In particular, $E_m(y-x)$ is a real analytic function in the real variables $x_0,x_1\ldots,x_m,y_0,y_1\ldots,y_m$.
\end{theorem}

We now wish to construct an analog of Corollary~\ref{cor:monogeniccauchy} for derivatives, subsuming~\cite[Corollary 7.28]{librogurlebeck2}. The resulting integral formula for derivatives will allow us to prove that every function that is left-monogenic with respect to $\B$ is harmonic with respect to $\B$ and real analytic. We recall that $\omega=\omega_{\B,\B'}\geq1$ is a constant such that $\Vert xa\Vert\leq\omega\,\Vert x\Vert \,\Vert a\Vert$ for all $x\in M,a\in A$ (see Remark~\ref{rmk:norminequality}). Moreover, by Proposition~\ref{prop:norm}: if $A$ is associative and $\B'$ is a fitted distinguished basis of $A$, then $\omega=1$.

\begin{theorem}[Integral formula for $\nabla_\B^\h \phi$]\label{thm:integralformulaforderivatives}
Assume $A$ to be associative.  Fix a domain $G$ in the hypercomplex subspace $M$ of $A$ and a function $\phi:G\to A$ that is left-monogenic with respect to $\B$. Then $\phi$ is harmonic with respect to $\B$ and real analytic. For every $\h\in\nn^{m+1}$: the function $\nabla_\B^\h \phi$ is still left-monogenic with respect to $\B$ and real analytic; given any open ball $B^{m+1}=B^{m+1}(p,R)$ whose closure $\overline{B}^{m+1}$ is contained in $G$,
\[\nabla_\B^\h \phi(x)=(-1)^{|\h|}\int_{\partial B^{m+1}}\left(\nabla_\B^{\h}E_m\right)(y-x)\,dy^*\,\phi(y)\]
for all $x\in B^{m+1}$; and, at the center $p$ of the ball $B^{m+1}$,
\[\Vert\nabla_\B^\h \phi(p)\Vert\leq\frac{C_m}{R^{|\h|}}\,\max_{\partial B^{m+1}}\Vert \phi\Vert\,,\qquad C_m:=\sigma_m\,\omega^2\,\max_{\partial B^{m+1}(0,1)}\Vert\nabla_\B^{\h}E_m\Vert\,.\]
\end{theorem}

The various properties stated in Theorem~\ref{thm:integralformulaforderivatives} have interesting consequences. The last inequality immediately yields the next corollary (see~\cite[Proposition 7.33]{librogurlebeck2} for the Clifford case).

\begin{corollary}[Liouville]
Assume $A$ to be associative. Let $\phi:M\to A$ be left-monogenic with respect to $\B$. If there exist $n\in\nn$ and $c>0$ such that
\[\Vert \phi(x)\Vert \leq c \Vert x\Vert^n\]
for all $x\in M$, then $\phi$ is a polynomial function and $\deg(\phi)\leq n$. In particular: if $\phi$ is bounded, then $\phi$ is constant.
\end{corollary}

Moreover, harmonicity allows to prove the next result, which subsumes~\cite[Theorem 7.32]{librogurlebeck2}.

\begin{theorem}[Maximum Modulus Principle]\label{thm:monogenicmaximummodulus}
Assume $A$ to be associative. Fix a domain $G$ in the hypercomplex subspace $M$ of $A$ and a function $\phi:G\to A$, left-monogenic with respect to $\B$. If the function $\Vert\phi\Vert:G\to\rr$ has a global maximum point in $G$, then $\phi$ is constant in $G$.
\end{theorem}

Finally, real analyticity will be the key ingredient to construct series representations of monogenic functions in the next subsection.


\subsection{Series expansions of monogenic functions on a hypercomplex subspace}\label{subsec:seriesexpansionmonogenic}

This subsection is devoted to series representations of functions that are left-monogenic with respect to the hypercomplex basis $\B$. The main result follows (see~\cite[Theorem 9.14 and Theorem 9.24]{librogurlebeck2} for the complex and Clifford cases).

\begin{theorem}[Series expansion]\label{thm:taylormonogenic}
Assume $A$ to be associative. Fix a domain $G$ in the hypercomplex subspace $M$ of $A$ and a function $\phi:G\to A$ that is left-monogenic with respect to $\B$. In every open ball $B^{m+1}(p,R)$ contained in $G$, the following series expansion is valid:
\[\phi(x)=\sum_{k\in\nn}\sum_{|\k|=k}\P_\k^\B(x-p)\,a_\k\,,\quad a_\k=\frac{1}{\k!}\nabla_\B^{(0,\k)}\phi(p)\,.\]
Here, the series converges normally in $B^{m+1}(p,R)$ because
\[\max_{\Vert x-p\Vert\leq r_1}\Big\Vert\sum_{|\k|=k}\P_\k^\B(x-p)\,a_\k\Big\Vert\leq\omega^2\,\sqrt{2}\;\binom{k+m}{m}\,\left(\frac{r_1}{r_2}\right)^k\,\max_{\Vert y-p\Vert=r_2}\Vert\phi(y)\Vert\]
whenever $0<r_1<r_2<R$.
\end{theorem}

Theorem~\ref{thm:taylormonogenic} has an extremely relevant consequence, which subsumes~\cite[Theorem 9.27]{librogurlebeck2}.

\begin{theorem}[Identity Principle]\label{thm:identitymonogenic}
Assume $A$ to be associative. Fix a domain $G$ in the hypercomplex subspace $M$ of $A$ and functions $\phi,\psi:G\to A$ that are left-monogenic with respect to $\B$. If $G$ contains a set of Hausdorff dimension $n\geq m$ where $\phi$ and $\psi$ coincide, then $\phi=\psi$ throughout $G$.
\end{theorem}

For future reference, we give the next definition and provide in the subsequent remark an equivalent restatement of Theorem~\ref{thm:taylormonogenic}.

\begin{definition}\label{def:deltaB}
Assume $A$ to be associative and fix $t_0\in\nn$. Let $M$ be a hypercomplex subspace of $A$, having a hypercomplex basis $\B=(v_0,v_1,\ldots,v_m)$ with $m\geq t_0$. Let $G$ be a domain in $M$ and let $\phi:G\to A$ be a left-monogenic function with respect to $\B$. For any $\h\in\nn^{m+1}$, define $\delta_\B^\h:\mathscr{C}^{|\h|}(G,A)\to\mathscr{C}^0(G,A)$ as
\begin{align*}
\delta_\B^{\h}\phi:=&\,(v_m^{h_m}\,v_{m-1}^{h_{m-1}}\cdots v_{t_0+1}^{h_{t_0+1}})^{-1}\nabla_\B^{\h}\phi\\
=&\,\partial_0^{h_0}\partial_1^{h_1}\ldots\partial_{t_0}^{h_{t_0}}(-v_{t_0+1}\partial_{t_0+1})^{h_{t_0+1}}\ldots(-v_m\partial_m)^{h_m}\phi\,.
\end{align*}
\end{definition}

In the last definition, the formula defining $\delta_\B^{\h}$ should be read as $\delta_\B^{\h}\phi:=\nabla_\B^{\h}\phi=\partial_0^{h_0}\partial_1^{h_1}\ldots\partial_{m}^{h_{m}}\phi$ in the special case when $m=t_0$.

\begin{remark}\label{rmk:alternatetaylor}
Assume $A$ to be associative and fix $t_0\in\nn$. Let $M$ be a hypercomplex subspace of $A$, having a hypercomplex basis $\B=(v_0,v_1,\ldots,v_m)$ with $m\geq t_0$. Let $G$ be a domain in $M$ and let $\phi:G\to A$ be a left-monogenic function with respect to $\B$. For every open ball $B^{m+1}=B^{m+1}(p,R)$ contained in $G$, the following series expansion is valid for $x\in B^{m+1}$:
\[\phi(x)=\sum_{k\in\nn}\sum_{|\k|=k}\P_\k^\B(x-p)v_m^{k_m}\,v_{m-1}^{k_{m-1}}\cdots v_{t_0+1}^{k_{t_0+1}}\frac{1}{\k!}\delta_\B^{(0,\k)}\phi(p)\,.\]
Here, the series converges normally in $B^{m+1}$ because
\begin{align*}
&\max_{\Vert x-p\Vert\leq r_1}\Big\Vert\sum_{|\k|=k}\P_\k^\B(x-p)v_m^{k_m}\,v_{m-1}^{k_{m-1}}\cdots v_{t_0+1}^{k_{t_0+1}}\frac{1}{\k!}\delta_\B^{(0,\k)}\phi(p)\Big\Vert\\
&\qquad\leq\omega^2\,\sqrt{2}\,\binom{k+m}{m}\left(\frac{r_1}{r_2}\right)^k\max_{\Vert y-p\Vert=r_2}\Vert\phi(y)\Vert
\end{align*}
whenever $0<r_1<r_2<R$.
\end{remark}


\section{Regularity in hypercomplex subspaces}\label{sec:hypercomplexregularity}

We henceforth make the following assumption.

\begin{assumption}\label{ass:alternative}
$V$ is a hypercomplex subspace of the alternative real $*$-algebra $A$, having a hypercomplex basis $\B=(v_0,v_1,\ldots,v_n)$ for some $n\in\nn^*$. After completing $\B$ to a real vector basis $\B'=(v_0,v_1,\ldots,v_n,v_{n+1},\ldots,v_d)$ of $A$, we endow $A$ with the standard Euclidean scalar product $\langle\cdot,\cdot\rangle=\langle\cdot,\cdot\rangle_{\B'}$ and norm $\Vert\cdot\Vert=\Vert\cdot\Vert_{\B'}$ associated to $\B'$.
\end{assumption}

There is a good reason to change notations with respect to Section~\ref{sec:monogenic}, which proved properties valid for an $(m+1)$-dimensional hypercomplex subspace $M$ of $A$, for domains $G\subseteq M$ and for functions $\phi:G\to A$. Indeed, we will apply those properties not only when $m=n$ and $M=V$, but also when $m<n$ and $M$ is a specific $(m+1)$-dimensional hypercomplex subspace of $A$ contained in $V$. The precise construction of these subspaces is the subject of the next subsection.


\subsection{$T$-fans}\label{subsec:Tfans}

Within our hypercomplex subspace $V$, we now construct some useful fans.

\begin{definition}
For $0\leq\ell\leq m\leq n$, we consider the $(m-\ell+1)$-dimensional subspace
\[\rr_{\ell,m}:=\Span(v_\ell,\ldots,v_m)\,.\]
Its unit $(m-\ell)$-sphere is denoted by $\s_{\ell,m}$. 
\end{definition}

For instance: $\rr_{0,n}=V$ and $\s_{1,n}=\{\sum_{t=1}^nx_tv_t\in A:\sum_{t=1}^nx_t^2=1\}=\s_A\cap V$. In general, $\s_{\ell,m}$ is a subset of $\s_A\cap V$ if, and only if, $\ell\geq1$. We recall that, by Corollary~\ref{cor:compact}, $\s_A\cap V$ is a compact set.

\begin{definition}
For any \emph{number of steps} $\tau\in\{0,\ldots,n\}$ and any \emph{list of steps} $T=(t_0,\ldots,t_\tau)\in\nn^{\tau+1}$, with $0\leq t_0<t_1<\ldots<t_\tau=n$, we define the \emph{$T$-fan} as
\[\rr_{0,t_0}\subsetneq\rr_{0,t_1}\subsetneq\ldots\subsetneq\rr_{0,t_\tau}=V\,.\]
The first subspace, $\rr_{0,t_0}$, is called the \emph{mirror}. We define the \emph{$T$-torus} as
\[\torus:=\s_{t_0+1,t_1}\times\ldots\times\s_{t_{\tau-1}+1,t_\tau}\]
when $\tau\geq1$ and as $\torus:=\emptyset$ when $\tau=0$ (whence $t_0=n$).
\end{definition}

We assume henceforth $\tau\in\{0,\ldots,n\}$ and $T=(t_0,\ldots,t_\tau)\in\nn^{\tau+1}$ (with $0\leq t_0<t_1<\ldots<t_\tau=n$) are fixed. Necessarily, $n\geq t_0+\tau$. The mirror $\rr_{0,t_0}$ of the $T$-fan is either the real axis $\rr$ or a hypercomplex subspace of $A$, while all other elements of the $T$-fan are hypercomplex subspaces of $A$. Moreover, if $\tau\geq1$ then, for every $h\in\{1,\ldots,\tau\}$, the sphere $\s_{t_{h-1}+1,t_h}$ is a $(t_h-t_{h-1}-1)$-dimensional subset of $\s_A\cap V$ and the $T$-torus $\torus$ is a $(n-t_0-\tau)$-dimensional compact set contained in $(\s_A)^\tau$.

\begin{example}[Paravectors]
If $V$ is the space $\rr^{n+1}$ of paravectors in $C\ell(0,n)$ (see Example~\ref{ex:paravectors}), then the $T$-fan is
\[\rr^{t_0+1}\subsetneq\rr^{t_1+1}\subsetneq\ldots\subsetneq\rr^{t_\tau+1}=\rr^{n+1}\,.\]
\end{example}

\begin{example}[Quaternions]\label{ex:quaternions2}
If $V=\hh$ within $\hh$ (see Example~\ref{ex:quaternions}): the $(3)$-fan is $\hh$; the $(2,3)$-fan is $\rr+i\rr+j\rr\subsetneq\hh$; the $(1,3)$-fan is $\cc\subsetneq\hh$; the $(0,3)$-fan is $\rr\subsetneq\hh$; the $(1,2,3)$-fan is $\cc\subsetneq\rr+i\rr+j\rr\subsetneq\hh$; the $(0,2,3)$-fan is $\rr\subsetneq\rr+i\rr+j\rr\subsetneq\hh$; the $(0,1,3)$-fan is $\rr\subsetneq\cc\subsetneq\hh$; and the $(0,1,2,3)$-fan is $\rr\subsetneq\cc\subsetneq\rr+i\rr+j\rr\subsetneq\hh$.
\end{example}

The work~\cite{unifiednotion} includes the next remark and lemma, which are useful tools to define and study the concept of $J$-monogenic function for any $J\in\torus$.

\begin{remark}\label{rmk:decomposedvariable}
Every $x=\sum_{\ell=0}^nv_\ell x_\ell\in V$ can be decomposed as $x=x^0+x^1+\ldots+x^\tau$, where $x^h:=\sum_{\ell=t_{h-1}+1}^{t_h}x_\ell v_\ell\in\rr_{t_{h-1}+1,t_h}$ (with $t_{-1}:=-1$). The decomposition is orthogonal, whence unique. When $\tau\geq1$, there exist $\beta=(\beta_1,\ldots,\beta_\tau)\in\rr^\tau$ and $J=(J_1,\ldots,J_\tau)\in\torus$ such that
\begin{equation}\label{eq:decomposedvariable}
x=x^0+\beta_1J_1+\ldots+\beta_\tau J_\tau\,.
\end{equation}
Equality~\eqref{eq:decomposedvariable} holds true exactly when, for each $h\in\{1,\ldots,\tau\}$: either $x^h\neq0,\beta_h=\pm\Vert x^h\Vert$ and $J_h=\frac{x^h}{\beta_h}$; or $x^h=0,\beta_h=0$ and $J_h$ is any element of $\s_{t_{h-1}+1,t_h}$.
\end{remark}

\begin{lemma}
If $\tau\geq1$, fix $J=(J_1,\ldots,J_\tau)\in\torus$ and set
\[\rr^{t_0+\tau+1}_J:=\Span(\B_J)\,,\quad\B_J:=(v_0,v_1,\ldots,v_{t_0},J_1,\ldots,J_\tau)\,.\]
If $\tau=0$ (whence $t_0=n\geq1$), set $J:=\emptyset,\B_\emptyset:=(v_0,v_1,\ldots,v_{t_0})=\B,\rr^{t_0+1}_\emptyset:=\Span(\B_\emptyset)=V$. 
In either case, $\B_J$ is a hypercomplex basis of $\rr^{t_0+\tau+1}_J$, which is therefore a hypercomplex subspace of $A$ contained in $V$. Moreover, if $J'\in\torus$, then the equality $\rr^{t_0+\tau+1}_J=\rr^{t_0+\tau+1}_{J'}$ is equivalent to $J'\in\{\pm J_1\}\times\ldots\times\{\pm J_\tau\}$.
\end{lemma}

For future use, we remark that if $J=(J_1,\ldots,J_{\tau-1},J_\tau)\in\torus$ and if $J'=(J_1,\ldots,J_{\tau-1},J_\tau')$ for some $J_\tau'\in\s_{t_{\tau-1}+1,t_\tau}\setminus\{\pm J_\tau\}$, then $\rr^{t_0+\tau+1}_J\cap\rr^{t_0+\tau+1}_{J'}$ is a $(t_0+\tau)$-dimensional space, which we may identify with the $(J_1,\ldots,J_{\tau-1})$-slice $\rr^{t_0+\tau}_{(J_1,\ldots,J_{\tau-1})}$ of the hypercomplex subspace $\rr_{0,t_{\tau-1}}$ of $A$ if $\tau>1$ and with the mirror $\rr_{0,t_0}$ if $\tau=1$. We also make the following remark.

\begin{remark}\label{rmk:basischange}
For all $J\in\torus$, the hypercomplex basis $\B_J:=(v_0,v_1,\ldots,v_{t_0},J_1,\ldots,J_\tau)$ of $\rr^{t_0+\tau+1}_J$ can always be completed to a basis $(\B_J)'$ of $A$ that is orthonormal with respect to $\langle\cdot,\cdot\rangle_{\B'}$, so that $\langle\cdot,\cdot\rangle_{(\B_J)'}=\langle\cdot,\cdot\rangle_{\B'}$ and $\Vert\cdot\Vert_{(\B_J)'}=\Vert\cdot\Vert_{\B'}$.
\end{remark}


\subsection{$T$-regularity}\label{subsec:Tregularity}

This subsection is devoted to defining a new notion of regularity for functions $f:\Omega\to A$, where $\Omega$ is a domain in $V$.

\begin{definition}\label{def:Jmonogenic}
If $\tau\geq1$, fix $J=(J_1,\ldots,J_\tau)\in\torus$. If $\tau=0$, set $J:=\emptyset$. Over any domain $G$ in $\rr^{t_0+\tau+1}_J$, the \emph{$J$-Cauchy-Riemann operator} $\debar_J:\mathscr{C}^{1}(G,A)\to\mathscr{C}^0(G,A)$ is defined as $\debar_J:=\debar_{\B_J}$ and the operators $\partial_J:\mathscr{C}^{1}(G,A)\to\mathscr{C}^0(G,A)$ and $\Delta_J:\mathscr{C}^{2}(G,A)\to\mathscr{C}^0(G,A)$ are defined as $\partial_J:=\partial_{\B_J}$ and $\Delta_J:=\Delta_{\B_J}$, according to Definition~\ref{def:monogenic}. Explicitly, referring to the decomposition~\eqref{eq:decomposedvariable} of the variable $x$, we have
\begin{align*}
\debar_J&=\partial_{x_0}+v_1\partial_{x_1}+\ldots+v_{t_0}\partial_{x_{t_0}}+J_1\partial_{\beta_1}+\ldots+J_\tau\partial_{\beta_\tau}\,,\\
\partial_J&=\partial_{x_0}-v_1\partial_{x_1}-\ldots-v_{t_0}\partial_{x_{t_0}}-J_1\partial_{\beta_1}-\ldots-J_\tau\partial_{\beta_\tau}\,,\\
\Delta_J&=\partial_{x_0}^2+\partial_{x_1}^2+\ldots+\partial_{x_{t_0}}^2+\partial_{\beta_1}^2+\ldots+\partial_{\beta_\tau}^2\,.
\end{align*}
The left kernel of $\debar_J$ is denoted by $\mon_J(G,A)$ and its elements are called \emph{$J$-monogenic} functions. The elements of the kernel of $\Delta_J$ are called \emph{$J$-harmonic} functions.
\end{definition}

In the special case when $\tau=0$, whence $t_0=n$, our last definition sets $\debar_\emptyset:=\debar_\B=\partial_{x_0}+v_1\partial_{x_1}+\ldots+v_n\partial_{x_n}$, as well as $\partial_\emptyset:=\partial_\B=\partial_{x_0}-v_1\partial_{x_1}-\ldots-v_n\partial_{x_n}$ and $\Delta_\emptyset:=\Delta_\B=\partial_{x_0}^2+\partial_{x_1}^2+\ldots+\partial_{x_n}^2$.

For all $J\in\torus$, the class $\mon_J(G,A)$ is a real vector space (a right $A$-module if $A$ is associative). Moreover, $J$-monogenicity is preserved under composition with translations by elements of $\rr^{t_0+\tau+1}_J$. Using the formal definition $\debar_J:=\debar_{\B_J}$ is necessary to guarantee, for $J,J'\in\torus$,
\begin{equation}\label{eq:CRwellposed}
\debar_J=\debar_{J'}\ \Longleftarrow\ \rr^{t_0+\tau+1}_J=\rr^{t_0+\tau+1}_{J'}\,.
\end{equation}
Similar considerations apply to $\partial_J,\Delta_J$. Remark~\ref{rmk:harmonic} and Theorem~\ref{thm:integralformulaforderivatives} allow the next remark.

\begin{remark}\label{rmk:J-harmonic}
Fix any domain $G$ in $\rr^{t_0+\tau+1}_J$. The equalities $\debar_J\partial_J=\partial_J\debar_J=\Delta_J$ hold true on $\mathscr{C}^{2}(G,A)$. Moreover, $J$-monogenic functions are real analytic and $J$-harmonic.
\end{remark}

The work~\cite{unifiednotion} also defined the new concept of $T$-regular function. Its special case with $A=C\ell(0,n),V=\rr^{n+1}$ and $\tau=1$ was independently constructed in~\cite{xsgeneralizedpartialslice} (see also~\cite{xsannouncement}) under the name of \emph{generalized partial-slice monogenic function}.

\begin{definition}
For $J\in\torus$ (or $J=\emptyset$, in case $\tau=0$) and for $Y\subseteq V,f:Y\to A$, the intersection $Y_J:=Y\cap\rr^{t_0+\tau+1}_J$ is called the \emph{$J$-slice} of $Y$. Consider the restriction $f_J:=f_{|_{Y_J}}$. Let $\Omega$ be a domain in $V$. A function $f:\Omega\to A$ is termed \emph{$T$-regular} if, the restriction $f_J:\Omega_J\to A$ is $J$-monogenic for every $J\in\torus$, if $\tau\geq1$ (for $J=\emptyset$, if $\tau=0$). If, moreover, $f(\rr^{t_0+\tau+1}_J)\subseteq\rr^{t_0+\tau+1}_J$ for all $J\in\torus$, then $f$ is called \emph{$T$-slice preserving}. The class of $T$-regular functions $\Omega\to A$ is denoted by $\reg_T(\Omega,A)$.
\end{definition}

In the last definition, of course we consider the condition ``$f_J:\Omega_J\to A$ is $J$-monogenic'' automatically fulfilled when $\Omega_J=\emptyset$. For future use, we make the next remark.

\begin{remark}\label{rmk:mirrortranslation}
Given a domain $\Omega$ in $V$, the class $\reg_T(\Omega,A)$ is a real vector space (a right $A$-module if $A$ is associative). Moreover, if $f\in\reg_T(\Omega,A)$ and $p\in\rr_{0,t_0}$, then setting $g(x):=f(x+p)$ defines a $g\in\reg_T(\Omega-p,A)$. This function $g$ is $T$-slice preserving if, and only if, $f$ is.
\end{remark}

As remarked in~\cite{unifiednotion}, $T$-regularity subsumes some of the best-known function theories over $C\ell(0,n)$.

\begin{example}[Paravectors]
Fix a domain $\Omega$ within the paravector subspace $\rr^{n+1}$ of $C\ell(0,n)$ (see Example~\ref{ex:paravectors}). For any function $f:\Omega\to C\ell(0,n)$:
\begin{itemize}
\item $f$ is $(n)$-regular if, and only if, it is in the kernel of the operator $\partial_{x_0}+e_1\partial_{x_1}+\ldots+e_{n}\partial_{x_n}$; this is the definition of \emph{monogenic} function (see, e.g.,~\cite{librosommen,librocnops,librogurlebeck2});
\item $f$ is $(0,n)$-regular if, and only if, for any $J_1\in\s_{1,n}=\s_{C\ell(0,n)}\cap\rr^{n+1}$, the restriction $f_{J_1}$ to the planar domain $\Omega_{J_1}\subseteq\cc_{J_1}$ is a holomorphic map $(\Omega_{J_1},J_1)\to(C\ell(0,n),J_1)$; this is the same as being \emph{slice-monogenic},~\cite{israel} (or \emph{slice-hyperholomorphic},~\cite{librodaniele2}).
\end{itemize}
\end{example}

The work~\cite{unifiednotion} also contains a complete classification of $T$-regularity over the hypercomplex subspace $\hh$ of $\hh$. Not only $T$-regularity subsumes the best-known quaternionic function theories. It also includes an entirely new function theory, called $(1,3)$-regularity and studied in some detail in~\cite{unifiednotion}.

\begin{example}[Quaternions]\label{ex:quaternions3}
Let $\Omega$ be a domain in $\hh$ and consider a function $f:\Omega\to\hh$. Then:
\begin{itemize}
\item $f$ is $(3)$-regular $\Leftrightarrow f$ belongs to the kernel of the left Cauchy-Riemann-Fueter operator $\partial_{x_0}+i\partial_{x_1}+j\partial_{x_2}+k\partial_{x_3} \Leftrightarrow f$ is a \emph{left Fueter-regular} function (see~\cite{fueter1,fueter2,sudbery});
\item $f$ is $(2,3)$-regular $\Leftrightarrow (\partial_{x_0}+i\partial_{x_1}+j\partial_{x_2}+J_1\partial_{\beta_1})f(x_0+ix_1+jx_2+\beta_1J_1)\equiv0$ for all $J_1\in\s_{3,3}=\{\pm k\}$ $\Leftrightarrow f$ is left Fueter-regular (because of~\eqref{eq:CRwellposed});
\item $f$ is $(1,3)$-regular $\Leftrightarrow$
\[\hskip35pt\debar_{J_1}f(x_0+ix_1+\beta_1J_1):=(\partial_{x_0}+i\partial_{x_1}+J_1\partial_{\beta_1})f(x_0+ix_1+\beta_1J_1)\equiv0\]
for all $J_1$ in the $(1,3)$-torus $\s_{2,3}$, which is simply the circle $\s^1:=\s_\hh\cap(j\rr+k\rr)$;
\item $f$ is $(0,3)$-regular $\Leftrightarrow $ for any $J_1\in\s_{1,3}=\s_\hh$, the restriction $f_{J_1}$ to the planar domain $\Omega_{J_1}\subseteq\cc_{J_1}$ is a holomorphic map $(\Omega_{J_1},J_1)\to(\hh,J_1) \Leftrightarrow f$ is a \emph{slice-regular} function,~\cite{librospringer2} (or \emph{Cullen-regular} in the original articles~\cite{cras,advances});
\item $f$ is $(1,2,3)$-regular $\Leftrightarrow (\partial_{x_0}+i\partial_{x_1}+J_1\partial_{\beta_1}+J_2\partial_{\beta_2})f(x_0+ix_1+\beta_1J_1+\beta_2J_2)\equiv0$ for all $(J_1,J_2)\in\s_{2,2}\times\s_{3,3}=\{\pm j\}\times\{\pm k\} \Leftrightarrow f$ is left Fueter-regular (because of~\eqref{eq:CRwellposed});
\item $f$ is $(0,1,3)$-regular $\Leftrightarrow (\partial_{x_0}+J_1\partial_{\beta_1}+J_2\partial_{\beta_2})f(x_0+\beta_1J_1+\beta_2J_2)\equiv0$ for all $(J_1,J_2)\in\s_{1,1}\times\s_{2,3}=\{\pm i\}\times\s^1 \Leftrightarrow f$ is $(1,3)$-regular (because of~\eqref{eq:CRwellposed});
\item $f$ is $(0,2,3)$-regular $\Leftrightarrow (\partial_{x_0}+J_1\partial_{\beta_1}+J_2\partial_{\beta_2})f(x_0+\beta_1J_1+\beta_2J_2)\equiv0$ for all $(J_1,J_2)\in\s_{1,2}\times\s_{3,3}=(\s_\hh\cap(i\rr+j\rr))\times\{\pm k\}$; if $\Phi:\hh\to\hh$ denotes the unique real vector space isomorphism mapping the standard basis $(1,i,j,k)$ into $(1,k,-j,i)$, then $g\mapsto\Phi^{-1}\circ g\circ\Phi$ is a bijection $\reg_{(0,2,3)}(\Omega,\hh)\to\reg_{(0,1,3)}(\Phi^{-1}(\Omega),\hh)$;
\item $f$ is $(0,1,2,3)$-regular $\Leftrightarrow (\partial_{x_0}+J_1\partial_{\beta_1}+J_2\partial_{\beta_2}+J_3\partial_{\beta_3})f(x_0+\beta_1J_1+\beta_2J_2+\beta_3J_3)\equiv0$ for all $(J_1,J_2,J_3)\in\s_{1,1}\times\s_{2,2}\times\s_{3,3}=\{\pm i\}\times\{\pm j\}\times\{\pm k\} \Leftrightarrow f$ is left Fueter-regular (because of~\eqref{eq:CRwellposed}).
\end{itemize}
\end{example}

Additionally, within $\hh$, the nonstandard choice of the $(2)$-fan $\rr+j\rr+k\rr$ with $\B=(1,-k,j),\B'=(1,-k,j,i)$, recovers as $(2)$-regular functions the theory of~\cite{moisilteodorescu}, for the reasons explained in~\cite[page 30]{perotticr}. Our general construction allows to treat all these cases at once. This is in contrast with most literature: even in the simple case of quaternions, Fueter-regular functions cannot be studied simultaneously with monogenic functions because the space of paravectors in $C\ell(0,2)=\hh$ is properly included in $\hh$.

For the hypercomplex subspace $\oo$ of $\oo$:

\begin{example}[Octonions]
Fix a domain $\Omega$ in $\oo$ and a function $f:\Omega\to\oo$. Then:
\begin{itemize}
\item $f$ is $(7)$-regular $\Leftrightarrow f$ belongs to the kernel of the octonionic Cauchy-Riemann operator $\partial_{x_0}+i\partial_{x_1}+j\partial_{x_2}+k\partial_{x_3}+l\partial_{x_4}+(li)\partial_{x_5}+(lj)\partial_{x_6}+(lk)\partial_{x_7}\Leftrightarrow f$ is an octonionic monogenic function (see~\cite{sce} and the recent~\cite{krausshardifferentialtopological});
\item $f$ is $(0,7)$-regular $\Leftrightarrow$ for any octonionic imaginary unit $J_1$, the restriction $f_{J_1}$ to the planar domain $\Omega_{J_1}\subseteq\cc_{J_1}$ is a holomorphic map $(\Omega_{J_1},J_1)\to(\oo,J_1) \Leftrightarrow f$ is a slice-regular function (see~\cite{rocky}).
\end{itemize}
\end{example}

$T$-regularity yields $\widehat{T}$-regularity with respect to a list of steps $\widehat{T}$ shorter than $T$, in the sense specified by the next lemma.

\begin{lemma}\label{lem:lowerdimlregularity}
Fix $\tau\geq2$, a list $T=(t_0,t_1,\ldots,t_\tau)$ of $\tau$ steps, a domain $\Omega$ in $V$, and a function $f\in\reg_T(\Omega,A)$. Fix $\sigma$ with $1\leq\sigma\leq\tau$ and consider the list $\widehat{T}:=(t_0+\sigma,t_{\sigma+1},\ldots,t_\tau)$ of $\tau-\sigma$ steps. For any $\widehat J=(J_1,\ldots,J_{\sigma})\in\s_{t_0+1,t_1}\times\ldots\times\s_{t_{\sigma-1}+1,t_{\sigma}}$, the ordered set
\[\widehat\B:=(v_0,\ldots,v_{t_0},J_1,\ldots,J_{\sigma},v_{t_{\sigma}+1},\ldots,v_{t_{\tau}})\]
is a hypercomplex basis of $\widehat V=\Span(\widehat\B)$, which is therefore a hypercomplex subspace of $A$. With respect to $\widehat\B$, the restriction of $f$ to $\widehat\Omega_{\widehat J}:=\Omega\cap\widehat V$ is a $\widehat{T}$-regular function.
\end{lemma}

\begin{proof}
The function $f_{|_{\widehat\Omega_{\widehat J}}}$ is $\widehat{T}$-regular if, and only if, for any $(J_{\sigma+1},\ldots,J_\tau)\in\s_{t_{\sigma}+1,t_{\sigma+1}}\times\cdots\times\s_{t_{\tau-1}+1,t_{\tau}}$ and for $\widehat\B_{(J_{\sigma+1},\ldots,J_\tau)}:=(v_0,\ldots,v_{t_0},J_1,\ldots,J_{\sigma},J_{\sigma+1},\ldots,J_\tau)$, the restriction of $f_{|_{\widehat\Omega_{\widehat J}}}$ to the set
\[\widehat\Omega_{\widehat J}\cap\Span(\widehat\B_{(J_{\sigma+1},\ldots,J_\tau)})=\Omega_J\,,\quad J:=(J_1,\ldots,J_{\sigma},J_{\sigma+1},\ldots,J_\tau)\,,\]
which is simply $f_J$, is left-monogenic with respect to the basis $\widehat\B_{(J_{\sigma+1},\ldots,J_\tau)}$; equivalently,
\[(\partial_{x_0}+\ldots+v_{t_0}\partial_{x_{t_0}}+J_1\partial_{\beta_1}+\ldots+J_{\sigma}\partial_{\beta_{\sigma}}+J_{\sigma+1}\partial_{\beta_{\sigma+1}}+\ldots+J_\tau\partial_{\beta_\tau}) f_J\equiv0\,,\]
i.e., $f_J$ is $J$-monogenic. But the last property is true under our hypothesis $f\in\reg_T(\Omega,A)$.
\end{proof}

In general, a $T$-regular function $f$ needs not be continuous, even though all restrictions $f_J$ are real analytic and $J$-harmonic by Remark~\ref{rmk:J-harmonic}.

\begin{example}
Assume $\tau=1$, whence $T=(t_0,t_1)=(t_0,n)$ and the $T$-torus $\torus$ is the $(n-t_0-1)$-sphere $\s_{t_0+1,n}$. Pick $I\in\s_{t_0+1,n}$. Since $\rr^{t_0+2}_I\cap\rr^{t_0+2}_J=\rr_{0,t_0}$ for all $J\in\s_{t_0+1,n}\setminus\{\pm I\}$, we may define a $T$-regular $f:V\setminus\rr_{0,t_0}\to A$ by setting $f_I:\equiv1\equiv:f_{-I}$ in $\rr^{t_0+2}_I\setminus\rr_{0,t_0}=\rr^{t_0+2}_{-I}\setminus\rr_{0,t_0}$ and $f_J:\equiv0$ in $\rr^{t_0+2}_J\setminus\rr_{0,t_0}$ for all $J\in\s_{t_0+1,n}\setminus\{\pm I\}$.
\end{example}

To get better-behaved $T$-regular functions $f:\Omega\to A$, we need to carefully choose the domain $\Omega$.

\begin{definition}\label{def:Tslicedomain}
A domain $\Omega\subseteq V$ is called a \emph{$T$-slice domain} if it intersects the mirror $\rr_{0,t_0}$ and if, for any $J\in\torus$, the $J$-slice $\Omega_J$ is connected (whence a domain in $\rr^{t_0+\tau+1}_J$). 
\end{definition}

Over $T$-slice domains within an associative $*$-algebra, we will prove an Identity Principle and a Maximum Modulus Principle in the forthcoming Section~\ref{sec:seriesexpansion}.

Another relevant property for the domain $\Omega$ of a $T$-regular function is symmetry, defined according to the following construction.

\begin{definition}\label{def:reflections}
For all $\beta=(\beta_1,\ldots,\beta_\tau)\in\rr^\tau,J=(J_1,\ldots,J_\tau)\in\torus$, we set the notation
\[\beta\,J:=\beta_1J_1+\ldots+\beta_\tau J_\tau\in V\,.\]
If $\tau=0$, for $\beta\in\rr^0=\{0\}$ and $J=\emptyset$ we define $\beta\,J$ to be the zero element of $V$.
For any $h\in\{1,\ldots,\tau\}$, we define the reflection
\[\rr^\tau\to\rr^\tau\,,\quad\beta=(\beta_1,\ldots,\beta_\tau)\mapsto\overline{\beta}^h:=(\beta_1,\ldots,\beta_{h-1},-\beta_h,\beta_{h+1},\ldots,\beta_\tau)\,.\]
For future use, we also define, for all $H\in\mathscr{P}(\tau)$, the reflection $\rr^\tau\to\rr^\tau,\ \beta\mapsto\overline{\beta}^H$ as follows: $\beta\mapsto\overline{\beta}^\emptyset$ is the identity map, while for $H=\{h_1,\ldots,h_p\}$ (with $1\leq h_1<\ldots<h_p\leq\tau$) the map $\beta\mapsto\overline{\beta}^H$ is the composition of the $p$ reflections $\beta\mapsto\overline{\beta}^{h_1},\ldots,\beta\mapsto\overline{\beta}^{h_p}$. 
\end{definition}

We point out that $\beta\,J$ is not a product, but just a shorthand for the second part of the decomposition~\eqref{eq:decomposedvariable} of the variable $x$. In other words, $\beta\,J$ is defined so that $L_{\B_J}(x_0,x_1,\ldots,x_{t_0},\beta)=x_0+x_1v_1+\ldots+x_{t_0}v_{t_0}+\beta\,J$.

\begin{definition}\label{def:Tsymmetric}
For any $D\subseteq\rr_{0,t_0}\times\rr^\tau$, we set
\[\Omega_D:=\{\alpha+\beta\,J: (\alpha,\beta)\in D, J\in\torus\}\]
if $\tau\geq1$ (and $\Omega_D:=\{\alpha\in V: (\alpha,0)\in D\}$ if $\tau=0$). A subset of $V$ is termed \emph{$T$-symmetric} if it equals $\Omega_D$ for some $D\subseteq\rr_{0,t_0}\times\rr^\tau$. The \emph{$T$-symmetric completion} $\widetilde{Y}$ of a set $Y\subseteq V$ is the smallest $T$-symmetric subset of $V$ containing $Y$. For each point $x\in V$, we denote by $\torus_x$ the $T$-symmetric completion of the singleton $\{x\}$.
\end{definition}

Let $\Omega$ be a $T$-symmetric $T$-slice domain and assume $A$ to be associative: we will prove in the forthcoming Subsection~\ref{subsec:representationformula} that every $T$-regular function $f:\Omega\to A$ has a specific symmetry property. This result, called Representation Formula, is well-known in the special cases of quaternionic slice-regular functions (see~\cite[Theorem 3.1]{advancesrevised}) and Clifford slice-monogenic functions (see~\cite[Theorem 2.2.18]{librodaniele2} and references therein). As a consequence of the Representation Formula, we will find that every $T$-regular function $f:\Omega\to A$ is real analytic.


\subsection{Integral representation}\label{subsec:integralrepresentation}

When $A$ is associative, a Cauchy-type representation of $T$-regular functions is readily obtained, using that of $J$-monogenic functions. We recall Assumption~\ref{ass:alternative} and Remark~\ref{rmk:basischange}.

\begin{proposition}[Cauchy Formula]\label{prop:cauchy}
Assume $A$ to be associative. Let $\Omega$ be a domain in $V$ and $f\in\reg_T(\Omega,A)$. Fix $J\in\torus$ and a bounded domain $G$ in $\rr^{t_0+\tau+1}_J$ , with a $\mathscr{C}^1$ boundary $\partial G$, such that $\overline{G}\subset\Omega_J$. Then
\[\frac{1}{\sigma_{t_0+\tau}}\int_{\partial G}\frac{(y-x)^c}{\Vert y-x\Vert^{t_0+\tau+1}}\,dy^*\,f_J(y) =\left\{\begin{array}{ll}
f(x)&\mathrm{if\ } x\in G\\
0&\mathrm{if\ } x\in \rr^{t_0+\tau+1}_J\setminus\overline{G}\,.
\end{array}\right.\]
\end{proposition}

\begin{proof}
The thesis follows immediately by applying Corollary~\ref{cor:monogeniccauchy} to the $J$-monogenic function $f_J\in\mathscr{C}^1(\overline{G},A)$.
\end{proof}

The same is true for the mean value property. Here, and throughout the paper, for any $p\in V, R>0$, we adopt the notations $B(p,R):=\{x\in V:\Vert x-p\Vert<R\}$ and $\overline{B}(p,R):=\{x\in V:\Vert x-p\Vert\leq R\}$.

\begin{proposition}[Mean value property]
Assume $A$ to be associative. Let $\Omega$ be a domain in $V$ and $f\in\reg_T(\Omega,A)$. Fix $J\in\torus$, a point $p\in\Omega_J$ and a radius $R>0$ such that $\overline{B}(p,R)_J\subset\Omega_J$. Then
\[f(p)=\frac1{\sigma_m}\int_{\partial B_J}f_J(p+R w)\,|do_w|\,,\]
where $B:=B(0,1)$.
\end{proposition}

\begin{proof}
The thesis follows immediately by applying Proposition~\ref{prop:monogenicmeanvalue} to the $J$-monogenic function $f_J\in\mathscr{C}^1(\overline{B}(p,R)_J,A)$.
\end{proof}

Providing series representations of $T$-regular functions will require a much stronger effort, starting with the construction of appropriate polynomial functions in Section~\ref{sec:polynomials} and ending in Section~\ref{sec:seriesexpansion}.


\section{Polynomial regular functions}\label{sec:polynomials}

With our Assumption~\ref{ass:alternative} standing, in this section we also assume $A$ to be associative. We construct examples of $T$-regular polynomial functions and study the general properties of such functions. In particular, we construct a basis for the $A$-module described in the next definition.

\begin{definition}
For all $k\in\nn$, let $U_k$ denote the right $A$-submodule of $\reg_T(V,A)$ consisting of those elements $f\in\reg_T(V,A)$ such that $f(x_0v_0+\ldots+x_nv_n)$ is a $k$-homogeneous polynomial map in the $n+1$ real variables $x_0,\ldots,x_n$.
\end{definition}

In Subsection~\ref{subsec:T_k}, we construct for every $k\in\nn$ a finite set of $k$-homogeneous polynomial functions $\F_k$. In Subsection~\ref{subsec:partialderivatives} we construct adapted partial derivatives, useful to study not only $U_k$ but also $\reg_T(\Omega,A)$ for any domain $\Omega$ in $V$. Using these adapted derivatives, we prove in Subsection~\ref{subsec:propertiespolynomial} that $\F_k$ is a basis of $U_k$ for any $k\in\nn$.

\subsection{The polynomial functions $\T_\k$}\label{subsec:T_k}

Our construction of basic polynomial functions will subsume the following well-established cases, treating them all at once. We recall that Subsection~\ref{subsec:monogenicpolynomial} defined and studied the Fueter variables $\zeta_s^\B$ and polynomials $\P_\k^\B$, which are left-monogenic with respect to a hypercomplex basis $\B$. Moreover, Subsection~\ref{subsec:Tfans} associated to each $J$ in the $T$-torus $\torus$ (or to $J=\emptyset$, if $\tau=0$) a hypercomplex basis $\B_J$.

\begin{example}[\cite{sudbery}]
Let $A=\hh$, $V=\hh$ and $\B=(1,i,j,k)$. $(3)$-regular functions are Fueter-regular functions. In this case, the Fueter variables $\zeta_1^{\B_\emptyset},\zeta_2^{\B_\emptyset},\zeta_3^{\B_\emptyset}$ form a basis of $U_1$. In general, Fueter polynomials $\big\{\P_\k^{\B_\emptyset}\big\}_{|\k|=k}$ form a basis of $U_k$.
\end{example}

\begin{example}[\cite{librospringer2}]
Let $A=\hh$, $V=\hh$ and $\B=(1,i,j,k)$. $(0,3)$-regular functions are slice-regular functions. In this case, the full variable $x=x_0+ix_1+jx_2+kx_3$ is a basis of $U_1$ and, in general, its $k$th power is a basis of $U_k$.

When we restrict the full variable $x$ to the $J_1$-slice $\cc_{J_1}=\{x_0+\beta_1J_1:x_0,\beta_1\in\rr\}$ for $J_1\in\s_{0,2}$, we find $x_0+\beta_1J_1=(\beta_1-x_0J_1)J_1$, which equals the Fueter variable $\zeta^{\B_{J_1}}_1=\beta_1-x_0J_1$ only up to the multiplicative constant $J_1$. The restriction to $\cc_{J_1}$ of the  $k$th power gives $(x_0+\beta_1J_1)^k=\big(\zeta^{\B_{J_1}}_1\big)^kJ_1^k=\P^{\B_{J_1}}_{(k)}(x_0+\beta_1J_1)J_1^k$. On the other hand, it is not possible for odd $|\k|$ to construct a single function $\P_\k:\hh\to\hh$ by setting $(\P_\k)_{J_1}:=\P^{\B_{J_1}}_\k$ for all $J_1\in\s_{0,2}$: indeed, for $J_1\neq J_1'$, $(\P_\k)_{J_1}(x_0+\beta_1J_1)=(\beta_1-x_0J_1)^{|\k|}$ and $(\P_\k)_{J_1'}(x_0+\beta_1'J_1')=(\beta_1'-x_0J_1')^{|\k|}$ would not agree along $\cc_{J_1}\cap\cc_{J_1'}=\rr$.
\end{example}

\begin{example}[\cite{librogurlebeck2}]
Let $A=C\ell(0,n)$, $V=\rr^{n+1}$ and $\B=(e_\emptyset,e_1,\ldots,e_n)$. As we already mentioned, $(n)$-regular functions are the classical Clifford monogenic functions. The Fueter variables $\zeta_1^{\B_\emptyset},\ldots,\zeta_n^{\B_\emptyset}$ form a basis of $U_1$. In general, Fueter polynomials $\big\{\P_\k^{\B_\emptyset}\big\}_{|\k|=k}$ form a basis of $U_k$.
\end{example}

\begin{example}[\cite{librodaniele2}]
Let $A=C\ell(0,n)$, $V=\rr^{n+1}$ and $\B=(e_\emptyset,e_1,\ldots,e_n)$. $(0,n)$-regular functions are slice-monogenic functions. In this case, the full variable $x=x_0+e_1x_1+\dots+e_nx_n$ is a basis of $U_1$ and, in general, its $k$th power is a basis of $U_k$.

Again: for any $J_1\in\s_{0,n}$, the restriction to $\cc_{J_1}$ of the  $k$th power gives $\P^{\B_{J_1}}_{(k)}\,J_1^k$.
\end{example}

These instructive examples motivate us to construct some polynomial functions $\T_\k$, as follows. While the definition is rather technical, will soon prove that $(\T_\k)_J=\P_\k^{\B_J}\,J_\tau^{k_{t_0+\tau}}\cdots J_2^{k_{t_0+2}}\,J_1^{k_{t_0+1}}$ for all $J=(J_1,\ldots,J_\tau)\in\torus$. Recall that, for $x=x_0+v_1x_1+\ldots+v_nx_n$, we have set
\begin{align*}
x^0&=x_0+v_1x_1+\ldots+v_{t_0}x_{t_0}\,,\\
x^1&=v_{t_0+1}x_{t_0+1}+\ldots+v_{t_1}x_{t_1}\,,\\
x^2&=v_{t_1+1}x_{t_1+1}+\ldots+v_{t_2}x_{t_2}\,,\\
\vdots\\
x^\tau&=v_{t_{\tau-1}+1}x_{t_{\tau-1}+1}+\ldots+v_{t_\tau}x_{t_\tau}\,,
\end{align*}
where $t_\tau=n$ by construction.

\begin{definition}
Assume $A$ to be associative. We set $\T_\k:\equiv0$ if $\k\in\zz^{t_0+\tau}\setminus\nn^{t_0+\tau}$ and $\T_\k:\equiv1$ if $\k=(0,\ldots,0)$. For $\k\in\nn^{t_0+\tau}\setminus\{(0,\ldots,0)\}$, we define recursively
\begin{align*}
|\k|\T_\k(x)&:=\sum_{s=1}^{t_0}k_s\T_{\k-\epsilon_s}(x)\left(x_s-(-1)^ax_0v_s\right)+\sum_{s=t_0+1}^{t_0+\tau}(-1)^{b_s}k_s\T_{\k-\epsilon_s}(x)\left(x_0+(-1)^{a_s}x^{s-t_0}\right)
\end{align*}
where $a:=\sum_{u=t_0+1}^{t_0+\tau}k_u, a_s:=a-k_s$ and $b_s:=\sum_{u=s+1}^{t_0+\tau}k_u$. For all $k\in\nn$, we define $\F_k:=\{\T_\k\}_{|\k|=k}$.
\end{definition}

We notice that $a_s+b_s=\sum_{u=t_0+1}^{s-1}k_u+2b_s$, whence $(-1)^{a_s+b_s}=(-1)^{c_s},c_s:=\sum_{t_0<u<s}k_u$.

\begin{example}\label{ex:lowerdegreepolynomials}
$\F_0=\{\T_{(0,\ldots,0)}\}=\{1\}$. $\F_1$ consists of the functions
\begin{align*}
&\T_{\epsilon_s}(x)=x_s-x_0v_s&&1\leq s\leq t_0\,,\\
&\T_{\epsilon_{t_0+u}}(x)=x_0+x^u&&1\leq u\leq\tau\,.
\end{align*}
We call the functions $\T_{\epsilon_1},\ldots,\T_{\epsilon_{t_0}}$ the \emph{$T$-Fueter variables} and the functions $\T_{\epsilon_{t_0+1}},\ldots,\T_{\epsilon_{t_0+\tau}}$ the \emph{$T$-Cullen variables}.
$\F_2$ consists of the functions
\begin{align*}
&\T_{2\epsilon_s}(x)=(x_s-x_0v_s)^2&&1\leq s\leq t_0\,,\\
&\T_{2\epsilon_{t_0+u}}(x)=(x_0+x^u)^2&&1\leq u\leq\tau\,,\\
&\T_{\epsilon_s+\epsilon_u}(x)=\frac{1}{2}((x_u-x_0v_u)(x_s-x_0v_s)+(x_s-x_0v_s)(x_u-x_0v_u))&&1\leq s<u\leq t_0\,,\\
&\T_{\epsilon_s+\epsilon_{t_0+u}}(x)=\frac{1}{2}((x_0+x^u)(x_s+x_0v_s)+(x_s-x_0v_s)(x_0+x^u))&&1\leq s\leq t_0,1\leq u\leq\tau\,,\\
&\T_{\epsilon_{t_0+s}+\epsilon_{t_0+u}}(x)=\frac{1}{2}(-(x_0+x^u)(x_0-x^s)+(x_0+x^s)(x_0-x^u))&&1\leq s<u\leq\tau\,.
\end{align*}
\end{example}

Within $\reg_{(n)}(V,A)$, corresponding to the choice $\tau=0$, the set $\F_1$ consists only of Fueter variables. Within $\reg_{(0,n)}(V,A)$, corresponding to the choice $t_0=0$ and $\tau=1$, the set $\F_1$ consists only of a single Cullen variable, which is the full variable $x=x_0+x^1=x_0+x_1v_1+\ldots+x_nv_n$. These facts are consistent with the examples we gave at the beginning of this subsection.

In the special case with $A=C\ell(0,n),V=\rr^{n+1}$ and $\tau=1$, the work~\cite{xsgeneralizedpartialslice} (see also~\cite{xsannouncement}) constructed a different basis $(z_0,\ldots,z_{t_0})$ of the right $A$-module $U_1$ and used it to construct a basis of $U_k$ for each $k\in\nn$. When translated into our current notations, $z_s:=x_s+\beta_1J_1v_s$. There is no obvious extension of this construction to the case $\tau\geq2$, which is of interest here.

Our next aim is proving that the set $\F_k$ is a basis for $U_k$. We begin with the next lemma, which expresses each restriction of $\T_\k$ to a $J$-slice $\rr^{t_0+\tau+1}_J$ in terms of the polynomial function $\P_\k^{\B_J}:\rr^{t_0+\tau+1}_J\to A$ constructed in Subsection~\ref{subsec:monogenicpolynomial}.

\begin{lemma}\label{lem:PTnabladelta}
Assume $A$ to be associative. Fix $\k\in\zz^{t_0+\tau},
J=(J_1,\ldots,J_\tau)\in\torus$. The restriction $(\T_\k)_J$ of $\T_\k$ to the slice $\rr^{t_0+\tau+1}_J$ is the $J$-monogenic polynomial function
\[(\T_\k)_J=\P_\k^{\B_J}\,J_\tau^{k_{t_0+\tau}}\cdots J_2^{k_{t_0+2}}\,J_1^{k_{t_0+1}}\,.\]
\end{lemma}

\begin{proof}
Our proof is by induction. When $\k\in\zz^{t_0+\tau}\setminus\nn^{t_0+\tau}$, the thesis follows from the equalities $\T_\k\equiv0$ and $\P_\k^{\B_J}\equiv0$. When $\k=(0,\ldots,0)$, the thesis follows from the equalities $\T_\k\equiv1$,  $J_\tau^{k_{t_0+\tau}}\cdots J_2^{k_{t_0+2}}\,J_1^{k_{t_0+1}}=1$ and $\P_\k^{\B_J}\equiv1$. Let us prove that the thesis holds for $\k\in\nn^{t_0+\tau}\setminus\{(0,\ldots,0)\}$, assuming it to hold for $\k-\epsilon_s$ for all $s\in\{1,\ldots,t_0+\tau\}$. The definition of $\P_\k^{\B_J}$, with the notation $\zeta_s=\zeta_s^{\B_J}$, yields
\begin{align*}
|\k|\P_\k^{\B_J}\,J_\tau^{k_{t_0+\tau}}\cdots J_2^{k_{t_0+2}}\,J_1^{k_{t_0+1}}&=\sum_{s=1}^{t_0+\tau}k_s\P_{\k-\epsilon_s}^{\B_J}\zeta_s\,J_\tau^{k_{t_0+\tau}}\cdots J_2^{k_{t_0+2}}\,J_1^{k_{t_0+1}}\,.
\end{align*}
Now, for $s\in\{1,\ldots,t_0\}$, we have $\zeta_s=x_s-x_0v_s$, whence $\zeta_sJ_u=J_u\zeta_s^c$ for all $u\in\{1,\ldots,\tau\}$. Let us set the notations $C^{2m}(x):=x$ and $C^{2m+1}(x):=x^c$ for all $m\in\nn$ and use again the notations $a:=\sum_{u=t_0+1}^{t_0+\tau}k_u, a_s:=a-k_s$ and $b_s:=\sum_{u=s+1}^{t_0+\tau}k_u$. We compute the sum of the first $t_0$ summands as
\begin{align*}
&\sum_{s=1}^{t_0}k_s\,\P_{\k-\epsilon_s}^{\B_J}\zeta_s\,J_\tau^{k_{t_0+\tau}}\cdots J_2^{k_{t_0+2}}\,J_1^{k_{t_0+1}}\\
&=\sum_{s=1}^{t_0}k_s\,\P_{\k-\epsilon_s}^{\B_J}J_\tau^{k_{t_0+\tau}}\cdots J_2^{k_{t_0+2}}\,J_1^{k_{t_0+1}}C^a(\zeta_s)\\
&=\sum_{s=1}^{t_0}k_s\,(\T_{\k-\epsilon_s})_JC^a(\zeta_s)\,,
\end{align*}
where the last equality follows from our induction hypothesis. Now let us compute the sum of the last $\tau$ summands. For distinct $u,w\in\{1,\ldots,\tau\}$ and for $s=t_0+u$ (whence $u=s-t_0$), we find that $\zeta_s=\beta_u-x_0J_u$ has the following properties: $\zeta_sJ_u=x_0+\beta_uJ_u=J_u\zeta_s$; $\zeta_sJ_w=J_w\zeta_s^c$; and $(\zeta_sJ_u)J_w=J_wC(\zeta_sJ_u)$. We are now ready to compute
\begin{align*}
&\sum_{s=t_0+1}^{t_0+\tau}k_s\,\P_{\k-\epsilon_s}^{\B_J}\zeta_s\,J_\tau^{k_{t_0+\tau}}\cdots J_2^{k_{t_0+2}}\,J_1^{k_{t_0+1}}\\
&=\sum_{s=t_0+1}^{t_0+\tau}k_s\,\P_{\k-\epsilon_s}^{\B_J}J_\tau^{k_{t_0+\tau}}\cdots J_{s-t_0+1}^{k_{s+1}}J_{s-t_0}^{k_s}C^{b_s}(\zeta_s)J_{s-t_0-1}^{k_{s-1}}\cdots J_1^{k_{t_0+1}}\\
&=\sum_{s=t_0+1}^{t_0+\tau}(-1)^{b_s}k_s\,\P_{\k-\epsilon_s}^{\B_J}J_\tau^{k_{t_0+\tau}}\cdots J_{s-t_0+1}^{k_{s+1}}J_{s-t_0}^{k_s-1}C^{b_s}(\zeta_sJ_{s-t_0})J_{s-t_0-1}^{k_{s-1}}\cdots J_1^{k_{t_0+1}}\\
&=\sum_{s=t_0+1}^{t_0+\tau}(-1)^{b_s}k_s\,\P_{\k-\epsilon_s}^{\B_J}J_\tau^{k_{t_0+\tau}}\cdots J_{s-t_0+1}^{k_{s+1}}J_{s-t_0}^{k_s-1}J_{s-t_0-1}^{k_{s-1}}\cdots J_1^{k_{t_0+1}}C^{a_s}(\zeta_sJ_{s-t_0})\\
&=\sum_{s=t_0+1}^{t_0+\tau}(-1)^{b_s}k_s\,(\T_{\k-\epsilon_s})_JC^{a_s}(\zeta_sJ_{s-t_0})\,,
\end{align*}
where we used the fact that $J_{s-t_0}C^{b_s}(\zeta_s)=(-1)^{b_s}C^{b_s}(\zeta_sJ_{s-t_0})$. Overall, we have
\begin{align*}
&|\k|\P_\k^{\B_J}\,J_\tau^{k_{t_0+\tau}}\cdots J_2^{k_{t_0+2}}\,J_1^{k_{t_0+1}}=\sum_{s=1}^{t_0}k_s\,(\T_{\k-\epsilon_s})_JC^a(\zeta_s)+\sum_{s=t_0+1}^{t_0+\tau}(-1)^{b_s}k_s\,(\T_{\k-\epsilon_s})_JC^{a_s}(\zeta_sJ_{s-t_0})\\
&=\sum_{s=1}^{t_0}k_s\T_{\k-\epsilon_s}(x)\left(x_s-(-1)^ax_0v_s\right)+\sum_{s=t_0+1}^{t_0+\tau}(-1)^{b_s}k_s\,(\T_{\k-\epsilon_s})_J(x_0+(-1)^{a_s}\beta_{s-t_0}J_{s-t_0})\\
&=|\k|(\T_\k)_J\,,
\end{align*}
where the last equality follows from the definition of $\T_\k$, taking into account that $\beta_{s-t_0}J_{s-t_0}$ is the restriction to the $J$-slice $\rr^{t_0+\tau+1}_J$ of the function $x^{s-t_0}$.
\end{proof}

After some preparation in the next subsection, we will go back to the polynomials $\T_\k$ in the forthcoming Subsection~\ref{subsec:propertiespolynomial} and finally prove that $\F_k$ is a basis for $U_k$.

\subsection{Adapted partial derivatives}\label{subsec:partialderivatives}

This subsection is devoted to constructing some new differential operators. These operators will play an important role to prove that $\F_k$ is a basis of $U_k$, but also later in the paper. We begin with the next remark, which uses Definition~\ref{def:deltaB} with our current choice of $t_0$ (which we fixed in Subsection~\ref{subsec:Tfans} as a part of the list of steps $T=(t_0,t_1,\ldots,t_\tau)$).

\begin{remark}
Assume $A$ to be associative. Fix $\h\in\nn^{t_0+\tau+1}$, $J\in\torus$ ($J:=\emptyset$ if $\tau=0$) and a domain $G$ in $\rr^{t_0+\tau+1}_J$. Let us consider the operator $\delta_J^\h:=\delta_{\B_J}^\h:\mathscr{C}^{|\h|}(G,A)\to\mathscr{C}^0(G,A)$. If $\tau=0$, then
\begin{align*}
\delta_\emptyset^\h&=\nabla_{\B_\emptyset}^\h=\,\partial_{x_0}^{h_0}\partial_{x_1}^{h_1}\ldots\partial_{x_{t_0}}^{h_{t_0}}\,.
\end{align*}
If, instead, $\tau\geq1$ and $J=(J_1,\ldots,J_\tau)$, then
\begin{align*}
\delta_J^{\h}&=(J_\tau^{h_{t_0+\tau}}\cdots J_1^{h_{t_0+1}})^{-1}\nabla_{\B_J}^{\h}\\
&=\partial_{x_0}^{h_0}\partial_{x_1}^{h_1}\ldots\partial_{x_{t_0}}^{h_{t_0}}(-J_1\partial_{\beta_1})^{h_{t_0+1}}\ldots(-J_\tau\partial_{\beta_\tau})^{h_{t_0+\tau}}\,.
\end{align*}
\end{remark}

The notations in the last remark refer to the decomposition~\eqref{eq:decomposedvariable} of the variable $x$. We are now ready for the announced construction.

\begin{definition}\label{def:delta}
Assume $A$ to be associative. Fix $\h\in\nn^{t_0+\tau+1}$ and a domain $\Omega$ in $V$. If $\tau=0$, then we set $\delta^{\h}:=\delta_\emptyset^{\h}=\nabla_{\B_\emptyset}^\h=\nabla_\B^\h$. If $\tau\geq1$, then for any $f\in\reg_T(\Omega,A)$ we define the function $\delta^{\h}f:\Omega\to A$ to fulfill the equalities
\[(\delta^{\h}f)_J=\delta_J^{\h}f_J=(J_\tau^{h_{t_0+\tau}}\cdots J_1^{h_{t_0+1}})^{-1}\nabla_{\B_J}^{\h}f_J\]
in $\Omega_J$, for all $J\in\torus$.
\end{definition}

Definition~\ref{def:delta} will only be fully justified after the next theorem will be proven.

\begin{theorem}\label{thm:deltawellposed}
Assume $A$ to be associative. Fix $\tau\geq1$, a list $T=(t_0,t_1,\ldots,t_\tau)$ of $\tau$ steps, a domain $\Omega$ in $V$, and a function $f\in\reg_T(\Omega,A)$. For all $\h\in\nn^{t_0+\tau+1}$ and all $J,J'\in\torus$, the functions $\delta_J^{\h}f_J$ and $\delta_{J'}^{\h}f_{J'}$ coincide in $\Omega_J\cap\Omega_{J'}$.
\end{theorem}

To prove Theorem~\ref{thm:deltawellposed}, we will use the next lemma.

\begin{lemma}\label{lem:deltawhentauis1}
Assume $A$ to be associative and $\tau=1$ (whence $T=(t_0,t_1)=(t_0,n)$). Fix $J=J_1\in\torus=\s_{t_0+1,n}$ and a domain $G\subseteq\rr^{t_0+2}_J$. If $\phi\in\mon_J(G,A)$, then
\begin{align*}
\delta_J^{(h_0,\ldots,h_{t_0},2m)}\phi&=\partial_{x_0}^{h_0}\partial_{x_1}^{h_1}\ldots\partial_{x_{t_0}}^{h_{t_0}}(\partial_{x_0}^2+\partial_{x_1}^2+\ldots+\partial_{x_{t_0}}^2)^m\phi\,,\\
\delta_J^{(h_0,\ldots,h_{t_0},2m+1)}\phi&=\partial_{x_0}^{h_0}\partial_{x_1}^{h_1}\ldots\partial_{x_{t_0}}^{h_{t_0}}(\partial_{x_0}^2+\partial_{x_1}^2+\ldots+\partial_{x_{t_0}}^2)^m(\partial_{x_0}+v_1\partial_{x_1}+\ldots+v_{t_0}\partial_{x_{t_0}})\phi
\end{align*}
for all $h_0,\ldots,h_{t_0},m\in\nn$.
\end{lemma}

\begin{proof}
Since the function $\phi:G\to A$ is $J$-monogenic, we have
\[0\equiv\debar_J\phi=(\partial_{x_0}+v_1\partial_{x_1}+\ldots+v_{t_0}\partial_{x_{t_0}}+J_1\partial_{\beta_1})\phi\]
(whence $0\equiv\Delta_J\phi=(\partial_{x_0}^2+\partial_{x_1}^2+\ldots+\partial_{x_{t_0}}^2+\partial_{\beta_1}^2)\phi$). These equalities will allow us to prove the thesis, by induction on $m$.

For $\h=(h_0,\ldots,h_{t_0},0)$, the equality $\delta_J^{\h}\phi=(\partial_{x_0}^{h_0}\partial_{x_1}^{h_1}\ldots\partial_{x_{t_0}}^{h_{t_0}})\phi$ is the very definition of $\delta_J^{\h}$.

Now assume the thesis true for $\h=(h_0,\ldots,h_{t_0},2m)$ and notice that $\delta_J^{\h}\phi=(-1)^m\nabla_{\B_J}^\h\phi$ is still a $J$-monogenic function. For $\h'=(h_0,\ldots,h_{t_0},2m+1)$ and $\h''=(h_0,\ldots,h_{t_0},2m+2)$, we compute
\begin{align*}
\delta_J^{\h'}\phi&=(\partial_{x_0}^{h_0}\partial_{x_1}^{h_1}\ldots\partial_{x_{t_0}}^{h_{t_0}}(-J_1\partial_{\beta_1})^{2m+1})\phi\\
&=(-J_1\partial_{\beta_1})\delta_J^{\h}\phi\\
&=(\partial_{x_0}+v_1\partial_{x_1}+\ldots+v_{t_0}\partial_{x_{t_0}})\delta_J^{\h}\phi\,,\\
\delta_J^{\h''}\phi&=(\partial_{x_0}^{h_0}\partial_{x_1}^{h_1}\ldots\partial_{x_{t_0}}^{h_{t_0}}(-J_1\partial_{\beta_1})^{2m+2})\phi\\
&=-\partial_{\beta_1}^2\delta_J^{\h}\phi\\
&=(\partial_{x_0}^2+\partial_{x_1}^2+\ldots+\partial_{x_{t_0}}^2)\delta_J^{\h}\phi\,,
\end{align*}
whence the desired conclusion follows.
\end{proof}

\begin{proof}[Proof of Theorem~\ref{thm:deltawellposed}]
We will prove the thesis by induction on $\tau$.

For $\tau=1$, we apply Lemma~\ref{lem:deltawhentauis1} to the function $f_J\in\mon_J(\Omega_J,A)$ and to the function $f_{J'}\in\mon_{J'}(\Omega_{J'},A)$. It follows immediately that $\delta_J^{\h}f_J$ and $\delta_{J'}^{\h}f_{J'}$ coincide in the intersection $\Omega_J\cap\Omega_{J'}$ of their domains.

We now take the induction step from $\tau-1$ to $\tau\geq2$. Let us apply Lemma~\ref{lem:lowerdimlregularity} with $\sigma=1$: for every $\widehat J=J_1\in\s_{t_0+1,t_1}$ the restriction $f_{|_{\widehat\Omega_{\widehat J}}}$ is a $\widehat{T}$-regular function, where $\widehat{T}:=(t_0+1,t_2,\ldots,t_\tau)$ is a list of $\tau-1$ steps. For any $(J_2,\ldots,J_\tau),(J_2',\ldots,J_\tau')\in\s_{t_1+1,t_2}\times\cdots\times\s_{t_{\tau-1}+1,t_{\tau}}$, our induction hypothesis yields that the further restrictions $f_{(J_1,J_2,\ldots,J_\tau)},f_{(J_1,J_2',\ldots,J_\tau')}$ of $f_{|_{\widehat\Omega_{\widehat J}}}$ are such that, for all $\h\in\nn^{t_0+\tau+1}$, the functions
\begin{align*} 
\delta_{(J_2,\ldots,J_\tau)}^{\h}f_{(J_1,J_2,\ldots,J_\tau)}&=\partial_{x_0}^{h_0}\partial_{x_1}^{h_1}\ldots\partial_{x_{t_0}}^{h_{t_0}}\partial_{\beta_1}^{h_{t_0+1}}(-J_2\partial_{\beta_2})^{h_{t_0+2}}\ldots(-J_\tau\partial_{\beta_\tau})^{h_{t_0+\tau}}f_{(J_1,J_2,\ldots,J_\tau)}\,,
\\
\delta_{(J_2',\ldots,J_\tau')}^{\h}f_{(J_1,J_2',\ldots,J_\tau')}&=\partial_{x_0}^{h_0}\partial_{x_1}^{h_1}\ldots\partial_{x_{t_0}}^{h_{t_0}}\partial_{\beta_1}^{h_{t_0+1}}(-J_2'\partial_{\beta_2'})^{h_{t_0+2}}\ldots(-J_\tau'\partial_{\beta_\tau'})^{h_{t_0+\tau}}f_{(J_1,J_2',\ldots,J_\tau')}\,,
\end{align*}
coincide in $\Omega_{(J_1,J_2,\ldots,J_\tau)}\cap\Omega_{(J_1,J_2',\ldots,J_\tau')}$. For any $J=(J_1,J_2,\ldots,J_\tau),J'=(J_1',J_2',\ldots,J_\tau')\in\torus$, let us prove that $\delta_J^\h f_J$ and $\delta_{J'}^\h f_{J'}$ coincide in $\Omega_J\cap\Omega_{J'}$ by induction on $h_{t_0+1}$.
\begin{itemize}
\item If $\h=(h_0,\ldots,h_{t_0},0,h_{t_0+2},\ldots,h_{t_\tau})$, then
\begin{align*} 
\delta_J^\h f_J&=\partial_{x_0}^{h_0}\partial_{x_1}^{h_1}\ldots\partial_{x_{t_0}}^{h_{t_0}}(-J_2\partial_{\beta_2})^{h_{t_0+2}}\ldots(-J_\tau\partial_{\beta_\tau})^{h_{t_0+\tau}}f_J\,,
\\
\delta_{J'}^\h f_{J'}&=\partial_{x_0}^{h_0}\partial_{x_1}^{h_1}\ldots\partial_{x_{t_0}}^{h_{t_0}}(-J_2'\partial_{\beta_2'})^{h_{t_0+2}}\ldots(-J_\tau'\partial_{\beta_\tau'})^{h_{t_0+\tau}}f_{J'}\,,
\end{align*}
coincide in $\Omega_J\cap\Omega_{J'}$. This is obvious when $J_1'=\pm J_1$, which yields the equalities $f_{J'}=f_{(J_1,J_2',\ldots,J_\tau')}$ and $\Omega_J\cap\Omega_{J'}=\Omega_{(J_1,J_2,\ldots,J_\tau)}\cap\Omega_{(J_1,J_2',\ldots,J_\tau')}$. It is also true when $J_1'\neq\pm J_1$, which yields the proper inclusion $\Omega_J\cap\Omega_{J'}\subsetneq\Omega_{(J_1,J_2,\ldots,J_\tau)}\cap\Omega_{(J_1,J_2',\ldots,J_\tau')}$ and the chain of equalities
\begin{align*}
&\left(\partial_{x_0}^{h_0}\partial_{x_1}^{h_1}\ldots\partial_{x_{t_0}}^{h_{t_0}}(-J_2'\partial_{\beta_2'})^{h_{t_0+2}}\ldots(-J_\tau'\partial_{\beta_\tau'})^{h_{t_0+\tau}}f_{J'}\right)_{|_{\Omega_J\cap\Omega_{J'}}}\\
&=\left(\partial_{x_0}^{h_0}\partial_{x_1}^{h_1}\ldots\partial_{x_{t_0}}^{h_{t_0}}(-J_2'\partial_{\beta_2'})^{h_{t_0+2}}\ldots(-J_\tau'\partial_{\beta_\tau'})^{h_{t_0+\tau}}f_{(J_1,J_2',\ldots,J_\tau')}\right)_{|_{\Omega_J\cap\Omega_{J'}}}\\
&=\left(\partial_{x_0}^{h_0}\partial_{x_1}^{h_1}\ldots\partial_{x_{t_0}}^{h_{t_0}}(-J_2\partial_{\beta_2})^{h_{t_0+2}}\ldots(-J_\tau\partial_{\beta_\tau})^{h_{t_0+\tau}}f_J\right)_{|_{\Omega_J\cap\Omega_{J'}}}\,.
\end{align*}
\item Now assume the thesis proven for all $\h$ of the form $\h=(h_0,\ldots,h_{t_0},2m,h_{t_0+2},\ldots,h_{t_\tau})$. We set $\h':=(h_0,\ldots,h_{t_0},2m+1,h_{t_0+2},\ldots,h_{t_\tau})$ and compute
\begin{align*} 
\delta_J^{\h'} f_J&=(-J_1\partial_{\beta_1})\,\delta_J^\h f_J=-J_1\,\delta_J^\h\,\partial_{\beta_1}\,f_J\\	
&=-J_1\,\delta_J^\h\,J_1\,(\partial_{x_0}+v_1\partial_{x_1}+\ldots+v_{t_0}\partial_{x_{t_0}}+J_2\partial_{\beta_2}+\ldots+J_\tau\partial_{\beta_\tau})\,f_J\\
&=(-1)^{a_0}\delta_J^{\h_0}f_J+(-1)^{a_1}\left(v_1\delta_J^{\h_1}f_J+\ldots+v_{t_0}\delta_J^{\h_{t_0}}f_J\right)\\
&\quad+(-1)^{a_2}\delta_J^{\h_{t_0+2}}f_J+\ldots+(-1)^{a_\tau}\delta_J^{\h_{t_0+\tau}}f_J\,,
\end{align*}
for appropriate natural numbers $a_0,a_1,\ldots,a_\tau$ and for
\begin{align*} 
\h_0&:=(h_0+1,h_1,\ldots,h_{t_0},2m,h_{t_0+2},\ldots,h_{t_\tau})\,,\\
\h_1&:=(h_0,h_1+1,\ldots,h_{t_0},2m,h_{t_0+2},\ldots,h_{t_\tau})\,,\\
\vdots\\
\h_{t_0}&:=(h_0,h_1,\ldots,h_{t_0}+1,2m,h_{t_0+2},\ldots,h_{t_\tau})\,,\\
\h_{t_0+2}&:=(h_0,h_1,\ldots,h_{t_0},2m,h_{t_0+2}+1,\ldots,h_{t_\tau})\,,\\
\vdots\\
\h_{t_0+\tau}&:=(h_0,h_1,\ldots,h_{t_0},2m,h_{t_0+2},\ldots,h_{t_\tau}+1)\,.
\end{align*}
Similar computations prove that
\begin{align*} 
\delta_{J'}^{\h'} f_{J'}&=(-1)^{a_0}\delta_{J'}^{\h_0}f_{J'}+(-1)^{a_1}\left(v_1\delta_{J'}^{\h_1}f_{J'}+\ldots+v_{t_0}\delta_{J'}^{\h_{t_0}}f_{J'}\right)\\
&\quad+(-1)^{a_2}\delta_{J'}^{\h_{t_0+2}}f_{J'}+\ldots+(-1)^{a_\tau}\delta_{J'}^{\h_{t_0+\tau}}f_{J'}\,.
\end{align*}
For any $s\in\{0,\ldots,t_0,t_0+2,\ldots,t_\tau\}$, since the $(t_0+1)$-component of $\h_s$ equals $2m$, our induction hypothesis guarantees that the functions $\delta_J^{\h_s}f_J$ and $\delta_{J'}^{\h_s}f_{J'}$ coincide in $\Omega_J\cap\Omega_{J'}$. We immediately conclude that $\delta_J^{\h'} f_J$ and $\delta_{J'}^{\h'} f_{J'}$ coincide in $\Omega_J\cap\Omega_{J'}$. To complete the current induction step, we set $\h'':=(h_0,\ldots,h_{t_0},2m+2,h_{t_0+2},\ldots,h_{t_\tau})$ and compute
\begin{align*} 
\delta_J^{\h''} f_J&=-\partial_{\beta_1}^2\,\delta_J^\h\,f_J=\delta_J^\h\,(-\partial_{\beta_1}^2)\,f_J\\	
&=\delta_J^{\h}\,(\partial_{x_0}^2+\partial_{x_1}^2+\ldots+\partial_{x_{t_0}}^2+\partial_{\beta_2}^2+\ldots+\partial_{\beta_\tau}^2)\,f_J\\
&=\delta_J^{\k_0}f_J+\delta_J^{\k_1}f_J+\ldots+\delta_J^{\k_{t_0}}f_J-\delta_J^{\k_{t_0+2}}f_J-\ldots-\delta_J^{\k_{t_0+\tau}}f_J\,,
\end{align*}
where
\begin{align*}
\k_0&:=(h_0+2,h_1,\ldots,h_{t_0},2m,h_{t_0+2},\ldots,h_{t_\tau})\,,\\
\k_1&:=(h_0,h_1+2,\ldots,h_{t_0},2m,h_{t_0+2},\ldots,h_{t_\tau})\,,\\
\vdots\\
\k_{t_0}&:=(h_0,h_1,\ldots,h_{t_0}+2,2m,h_{t_0+2},\ldots,h_{t_\tau})\,,\\
\k_{t_0+2}&:=(h_0,h_1,\ldots,h_{t_0},2m,h_{t_0+2}+2,\ldots,h_{t_\tau})\,,\\
\vdots\\
\k_{t_0+\tau}&:=(h_0,h_1,\ldots,h_{t_0},2m,h_{t_0+2},\ldots,h_{t_\tau}+2)\,.
\end{align*}
Similar computations prove that
\begin{align*} 
\delta_{J'}^{\h''} f_{J'}&=\delta_{J'}^{\k_0}f_{J'}+\delta_{J'}^{\k_1}f_{J'}+\ldots+\delta_{J'}^{\k_{t_0}}f_{J'}-\delta_{J'}^{\k_{t_0+2}}f_{J'}-\ldots-\delta_{J'}^{\k_{t_0+\tau}}f_{J'}\,.
\end{align*}
Again: for any $s\in\{0,\ldots,t_0,t_0+2,\ldots,t_\tau\}$, our induction hypothesis guarantees that the functions $\delta_J^{\k_s}f_J$ and $\delta_{J'}^{\k_s}f_{J'}$ coincide in $\Omega_J\cap\Omega_{J'}$. It follows that $\delta_J^{\h''} f_J$ and $\delta_{J'}^{\h''} f_{J'}$ coincide in $\Omega_J\cap\Omega_{J'}$, as desired. This completes the induction step from $\tau-1$ to $\tau\geq2$.
\end{itemize}
The thesis is now proven for all $\tau\in\nn$, as desired.
\end{proof}

We can use Lemma~\ref{lem:deltawhentauis1} once more, to establish the next result.

\begin{proposition}\label{prop:tau1}
Assume $A$ to be associative. If $\tau=1$ and $f\in\reg_T(\Omega,A)$, then
\begin{align*}
\delta^{(h_0,\ldots,h_{t_0},2m)}f&=\partial_{x_0}^{h_0}\partial_{x_1}^{h_1}\ldots\partial_{x_{t_0}}^{h_{t_0}}(\partial_{x_0}^2+\partial_{x_1}^2+\ldots+\partial_{x_{t_0}}^2)^mf\\
\delta^{(h_0,\ldots,h_{t_0},2m+1)}f&=\partial_{x_0}^{h_0}\partial_{x_1}^{h_1}\ldots\partial_{x_{t_0}}^{h_{t_0}}(\partial_{x_0}^2+\partial_{x_1}^2+\ldots+\partial_{x_{t_0}}^2)^m(\partial_{x_0}+v_1\partial_{x_1}+\ldots+v_{t_0}\partial_{x_{t_0}})f
\end{align*}
for all $h_0,\ldots,h_{t_0},m\in\nn$. As a consequence: for any $\h\in\nn^{t_0+2}$, the function $\delta^\h f_{|_{\Omega\cap\rr_{0,t_0}}}$ is completely determined by $f_{|_{\Omega\cap\rr_{0,t_0}}}$.
\end{proposition}

\begin{proof}
The first statement follows by applying Lemma~\ref{lem:deltawhentauis1} to the function $f_J\in\mon_J(\Omega_J,A)$, if we take into account that $\delta^\h f$ is defined so that $(\delta^\h f)_J=\delta_J^\h f_J$ and that $\partial_{x_s}f_J=(\partial_{x_s}f)_J$ for all $s\in\{0,\ldots,t_0\}$. The first statement and the equalities $(\partial_{x_s}f)_{|_{\Omega\cap\rr_{0,t_0}}}=\partial_{x_s}(f_{|_{\Omega\cap\rr_{0,t_0}}})$ (valid for all $s\in\{0,\ldots,t_0\}$) yield the second statement.
\end{proof}

This proposition is consistent with~\cite{xsgeneralizedpartialslice}, which dealt with the special case with $A=C\ell(0,n),V=\rr^{n+1}$ and $\tau=1$: in that special case,~\cite[Theorem 3.27]{xsgeneralizedpartialslice} showed that $f\in\reg_T(B(0,R),A)$ is uniquely determined by its restriction to $B(0,R)\cap\rr_{0,t_0}$.

None of the phenomena described in Proposition~\ref{prop:tau1} generalizes to the case $\tau\geq2$, which is of interest here. Our forthcoming Example~\ref{ex:deltanotbasedonmirror}, where $A=C\ell(0,4),V=\rr^5$ and $\tau=2$, shows that the operator $\delta^\h$ cannot, in general, be expressed in terms of iterates of $\partial_{x_0},\partial_{x_1},\ldots,\partial_{x_{t_0}}$ only. Additionally, it shows that, for $f\in\reg_T(\Omega,A)$, the function $\delta^\h f_{|_{\Omega\cap\rr_{0,t_0}}}$ is not uniquely determined by $f_{|_{\Omega\cap\rr_{0,t_0}}}$.

\subsection{Properties of polynomial functions $\T_\k$}\label{subsec:propertiespolynomial}

We are now ready to prove that $\F_k$ is a basis of $U_k$ for any $k\in\nn$.

\begin{theorem}\label{thm:polynomialexpansion}
Assume $A$ to be associative. For every $k\in\nn$, the family $\F_k$ is a basis for $U_k$. Namely, for every $P\in U_k$,
\begin{equation}\label{eq:polynomialexpansion}
P(x)=\sum_{|\k|=k}\T_\k(x)\,c_\k,\quad c_\k:=\frac{1}{\k!}\delta^{(0,\k)}P(0)
\end{equation}
for all $x\in V$. In particular, $\delta^{(0,\k)}\T_\k(0)=1$ and $\delta^{(0,\k)}\T_{\k'}(0)=0$ when $\k\neq\k'$. 
\end{theorem}

\begin{proof}
Fix $k\in \nn$. We take several steps to prove our thesis.

Let us first prove the inclusion $\F_k\subseteq U_k$. By construction, each function $\T_\k$ is a $|\k|$-homogenous polynomial. By  Lemma~\ref{lem:PTnabladelta}, for all $J\in\torus$, the restriction $(\T_\k)_J$ is a $J$-monogenic polynomial function. Thus, $\T_\k\in\reg_T(V,A)$. The desired inclusion follows.

We now aim at proving that the elements of $\F_k$ are linearly independent. For $\{c_\k\}_{|\k|=k}\subset A$, assume $P(x):=\sum_{|\k|=k}\T_\k(x)\,c_\k$ to vanish identically in $V$. By Lemma~\ref{lem:PTnabladelta}, for any $J\in\torus$, the restriction
\[P_J=\sum_{|\k|=k}(\T_\k)_J\,c_\k=\sum_{|\k|=k}\P_\k^{\B_J}\,J_\tau^{k_{t_0+\tau}}\cdots J_2^{k_{t_0+2}}\,J_1^{k_{t_0+1}}\,c_\k\]
vanishes identically in $\rr^{t_0+\tau+1}_J$. By Proposition~\ref{prop:basismonogenicpolynomials}, $J_\tau^{k_{t_0+\tau}}\cdots J_2^{k_{t_0+2}}\,J_1^{k_{t_0+1}}\,c_\k=0$ for all $\k\in\nn^{t_0+\tau}$ with $|\k|=k$. Since $J_s$ has inverse $-J_s$ for all $s\in\{1,\ldots,\tau\}$, we conclude that $c_\k=0$ for all $\k\in\nn^{t_0+\tau}$, as desired.

Let us now prove that formula~\eqref{eq:polynomialexpansion} is true for all $P\in U_k$, whence the family $\F_k$ spans $U_k$. It suffices to prove that the polynomial function $\widetilde{P}:=\sum_{|\k|=k}\T_\k\,\frac{1}{\k!}\delta^{(0,\k)}P(0)$ coincides with $P$. This is true, because Lemma~\ref{lem:PTnabladelta}, Definition~\ref{def:delta} and Proposition~\ref{prop:basismonogenicpolynomials} yield
\begin{align*}
\widetilde{P}_J&=\sum_{|\k|=k}(\T_\k)_J\,\frac{1}{\k!}\delta_J^{(0,\k)}P(0)=\sum_{|\k|=k}\P_\k^{\B_J}\,J_\tau^{k_{t_0+\tau}}\cdots J_2^{k_{t_0+2}}\,J_1^{k_{t_0+1}}\,\frac{1}{\k!}\delta_J^{(0,\k)}P(0)\\
&=\sum_{|\k|=k}\P^{\B_J}_\k\frac{1}{\k!}\nabla_{\B_J}^{(0,\k)}P_J(0)=P_J
\end{align*}
for all $J\in\torus$.
\end{proof}

We are now in a position to give an example where two $T$-regular functions coincide on the mirror but have distinct adapted partial derivatives.

\begin{example}\label{ex:deltanotbasedonmirror}
Let $A=C\ell(0,4)$, $V=\rr^5$, $\B=(e_\emptyset,e_1,e_2,e_3,e_4)$, $T=(0,2,4)$ (whence $\tau=2$). Within $\reg_{(0,2,4)}(\rr^5,C\ell(0,4))$, let us consider $\F_1=\{\T_{\epsilon_1},\T_{\epsilon_2}\}$, where $\epsilon_1=(1,0)$ and $\epsilon_2=(0,1)$ and $\T_{\epsilon_1},\T_{\epsilon_2}$ are the $(0,2,4)$-Cullen variables $\T_{\epsilon_1}(x)=x_0+x^1=x_0+x_1e_1+x_2e_2$, $\T_{\epsilon_2}(x)=x_0+x^2=x_0+x_3e_3+x_4e_4$. While $\delta^{(1,0,0)}\T_{\epsilon_1}\equiv1\equiv\delta^{(1,0,0)}\T_{\epsilon_2}$, Theorem~\ref{thm:polynomialexpansion} guarantees that
\begin{align*}
\delta^{(0,\epsilon_1)}\T_{\epsilon_1}\equiv1\neq0\equiv\delta^{(0,\epsilon_1)}\T_{\epsilon_2}\\
\delta^{(0,\epsilon_2)}\T_{\epsilon_1}\equiv0\neq1\equiv\delta^{(0,\epsilon_2)}\T_{\epsilon_2}
\end{align*}
despite the fact that $\T_{\epsilon_1}(x)=x_0=\T_{\epsilon_2}(x)$ for all $x$ in the mirror $\rr_{0,0}=\rr$.
\end{example}

Before concluding this section, we use the family $\F_1$ to understand which $T,\widetilde T$ produce $\reg_T(\Omega,A)=\reg_{\widetilde T}(\Omega,A)$.

\begin{example}
We saw in Example~\ref{ex:quaternions3} that, for functions $\hh\to\hh$, $T$-regularity is Fueter-regularity exactly when $T\in\{(3),(2,3),(1,2,3),(0,1,2,3)\}$, while $T$-regularity is slice-regularity if, and only if, $T=(0,3)$. We also saw that $T=(1,3)$ yields the same class as $T=(0,1,3)$.
\end{example}

The phenomenon appearing in the previous example is consistent with the following fact. If $T=(t_0,t_1,\ldots,t_\tau)$ with $t_0>0$ and if we set $\widetilde T=(\widetilde t_0,\widetilde t_1,\ldots,\widetilde t_{\widetilde \tau})=(t_0-1,t_0,t_1,\ldots,t_\tau)$ (whence $\widetilde \tau=\tau+1$), then:
\begin{itemize}
\item the $T$-Fueter variables, excluding the last one, are exactly the $\widetilde T$-Fueter variables;
\item the last $T$-Fueter variable, namely $\T_{\epsilon_{t_0}}(x)=x_{t_0}-x_0v_{t_0}=(x_0+x_{t_0}v_{t_0})(-v_{t_0})$, coincides, up to the multiplicative constant $-v_{t_0}$, with the first $\widetilde T$-Cullen variable, i.e., $\widetilde \T_{\epsilon_{\widetilde t_0+1}}(x)=x_0+x_{\widetilde t_0+1}v_{\widetilde t_0+1}=x_0+x_{t_0}v_{t_0}$;
\item the $T$-Cullen variables are exactly the $\widetilde T$-Cullen variables, first one excluded.
\end{itemize}
We are going to prove that this mechanism, along with its iterations, is the only way to produce from $T$ a longer $\widetilde T$ such that $\reg_T(\Omega,A)=\reg_{\widetilde T}(\Omega,A)$. This will be a corollary to the next theorem.

\begin{theorem}\label{thm:fuetervscullen}
Assume $A$ to be associative. Let $\tau,\widetilde\tau\in\{0,\ldots,n\}$ and let $T=(t_0,\ldots,t_\tau)$ and $\widetilde T=(\widetilde t_0,\ldots,\widetilde t_{\widetilde\tau})$ be two lists of steps for $V$.
\begin{enumerate}
\item For $1\leq s\leq t_0$, the $T$-Fueter variable $\T_{\epsilon_s}(x)=x_s-x_0v_s$ is $\widetilde T$-regular if, and only if, either $s\leq\widetilde t_0$ or $\widetilde t_{u-1}+1=s=\widetilde t_u$ for some $u\geq1$. 
\item For $1\leq s\leq\tau$, consider the $T$-Cullen variable $\T_{\epsilon_{t_0+s}}(x)=x_0+x^s=x_0+v_{t_{s-1}+1}x_{t_{s-1}+1}+\ldots+v_{t_s}x_{t_s}$. Then $\T_{\epsilon_{t_0+s}}$ is $\widetilde T$-regular if, and only if, either $t_{s-1}+1=t_s\leq\widetilde t_0$ or there exists $u\in\{1,\ldots,\widetilde \tau\}$ such that $(t_{s-1},t_s)=(\widetilde t_{u-1},\widetilde t_u)$.
\end{enumerate}
\end{theorem}

\begin{proof}
Take any $J$ in the $\widetilde T$-torus $\widetilde\torus$ and consider the operator $\debar_J=\partial_{x_0}+v_1\partial_{x_1}+\ldots+v_{\widetilde t_0}\partial_{x_{\widetilde t_0}}+J_1\partial_{\beta_1}+\ldots+J_{\widetilde \tau}\partial_{\beta_{\widetilde \tau}}$ associated to $\widetilde T$.
\begin{enumerate}
\item We first deal, for $1\leq s\leq t_0$, with the $T$-Fueter variable $\T_{\epsilon_s}(x)=x_s-x_0v_s$. We separate three cases.
\begin{enumerate}
\item If $s\leq\widetilde t_0$, then for any $J\in\widetilde\torus$ we have
\[(\T_{\epsilon_s})_J(x_0+v_1x_1+\ldots+v_{\widetilde t_0}x_{\widetilde t_0}+J_1\beta_1+\ldots+J_{\widetilde \tau}\beta_{\widetilde \tau})=x_s-x_0v_s\,,\]
whence
\[\debar_J(\T_{\epsilon_s})_J=(\partial_{x_0}+v_s\partial_{x_s})(x_s-x_0v_s)\equiv v_s-v_s=0\,.\]
\item If there exists $u\geq1$ such that $\widetilde t_{u-1}+1=s=\widetilde t_u$, then $\s_{\widetilde t_{u-1}+1,\widetilde t_u}=\{\pm v_s\}$ and, for any $J\in\widetilde\torus$, we have $J_u=\pm v_s$ and
\[(\T_{\epsilon_s})_J(x_0+v_1x_1+\ldots+v_{\widetilde t_0}x_{\widetilde t_0}+J_1\beta_1+\ldots+J_{\widetilde \tau}\beta_{\widetilde \tau})=\pm\beta_u-x_0v_s\,,\]
whence
\[\debar_J(\T_{\epsilon_s})_J=(\partial_{x_0}+J_u\partial_{\beta_u})(\pm\beta_u-x_0v_s)\equiv \pm J_u-v_s=0\,.\]
\item If there exists $u\geq1$ such that $\widetilde t_{u-1}+1\leq s\leq\widetilde t_u$ and $\widetilde t_{u-1}+1<\widetilde t_u$, then $\s_{\widetilde t_{u-1}+1,\widetilde t_u}$ is a sphere of dimension at least $1$. Picking $J\in\widetilde\torus$ with $J_u\perp v_s$, we obtain
\[(\T_{\epsilon_s})_J(x_0+v_1x_1+\ldots+v_{\widetilde t_0}x_{\widetilde t_0}+J_1\beta_1+\ldots+J_{\widetilde \tau}\beta_{\widetilde \tau})=-x_0v_s\,,\]
whence
\[\debar_J(\T_{\epsilon_s})_J=\partial_{x_0}(-x_0v_s)\equiv-v_s\neq0\,.\]
\end{enumerate}
\item We now deal with the $T$-Cullen variable $\T_{\epsilon_{t_0+s}}(x)=x_0+x^s=x_0+v_{t_{s-1}+1}x_{t_{s-1}+1}+\ldots+v_{t_s}x_{t_s}$ for $1\leq s\leq\tau$. We separate three cases.
\begin{enumerate}
\item If $t_s\leq\widetilde t_0$, then for any $J\in\widetilde\torus$ we have
\[(\T_{\epsilon_{t_0+s}})_J(x_0+v_1x_1+\ldots+v_{\widetilde t_0}x_{\widetilde t_0}+J_1\beta_1+\ldots+J_{\widetilde \tau}\beta_{\widetilde \tau})=x_0+v_{t_{s-1}+1}x_{t_{s-1}+1}+\ldots+v_{t_s}x_{t_s}\,,\]
whence
\[\debar_J(\T_{\epsilon_{t_0+s}})_J\equiv 1+v_{t_{s-1}+1}^2+\ldots+v_{t_s}^2=1-(t_s-t_{s-1})=t_{s-1}+1-t_s\,.\]
Thus, $\T_{\epsilon_{t_0+s}}$ is $\widetilde T$-regular with $t_s\leq\widetilde t_0$ if, and only if, $t_{s-1}+1=t_s$.
\item If $t_{s-1}+1\leq\widetilde t_0<t_s$, then for any $J\in\widetilde\torus$ we have
\begin{align*}
&(\T_{\epsilon_{t_0+s}})_J(x_0+v_1x_1+\ldots+v_{\widetilde t_0}x_{\widetilde t_0}+J_1\beta_1+\ldots+J_{\widetilde \tau}\beta_{\widetilde \tau})\\
&=x_0+v_{t_{s-1}+1}x_{t_{s-1}+1}+\ldots+v_{\widetilde t_0}x_{\widetilde t_0}\\
&\quad+v_{\widetilde t_0+1}\langle v_{\widetilde t_0+1},J_1\beta_1+\ldots+J_{\widetilde \tau}\beta_{\widetilde \tau}\rangle+\ldots
+v_{t_s}\langle v_{t_s},J_1\beta_1+\ldots+J_{\widetilde \tau}\beta_{\widetilde \tau}\rangle\,,
\end{align*}
whence
\begin{align*}
\debar_J(\T_{\epsilon_{t_0+s}})_J&\equiv 1+v_{t_{s-1}+1}^2+\ldots+v_{\widetilde t_0}^2+J_1\left(v_{\widetilde t_0+1}\langle v_{\widetilde t_0+1},J_1\rangle+\ldots+v_{t_s}\langle v_{t_s},J_1\rangle\right)\\
&\quad+\ldots+J_{\widetilde\tau}\left(v_{\widetilde t_0+1}\langle v_{\widetilde t_0+1},J_{\widetilde\tau}\rangle+\ldots+v_{t_s}\langle v_{t_s},J_{\widetilde\tau}\rangle\right)\\
&=t_{s-1}+1-\widetilde t_0+J_1\,\mbox{proj}_{\rr_{\widetilde t_0+1,t_s}}(J_1)+\ldots+J_{\widetilde\tau}\,\mbox{proj}_{\rr_{\widetilde t_0+1,t_s}}(J_{\widetilde\tau})\\
&=t_{s-1}+1-\widetilde t_0+J_1\,\mbox{proj}_{\rr_{\widetilde t_0+1,t_s}}(J_1)+\ldots+J_u\,\mbox{proj}_{\rr_{\widetilde t_0+1,t_s}}(J_u)\,,
\end{align*}
where $u$ is the maximal element of $\{1,\ldots,\widetilde\tau\}$ such that $t_s\geq\widetilde t_{u-1}+1$. Let us choose $J\in\torus$ such that $J_1=v_{\widetilde t_0+1},\ldots,J_u=v_{\widetilde t_0+u}$: then for all $s\in\{1,\ldots,u\}$ we have $J_s\in\rr_{\widetilde t_0+1,t_s}$, whence $J_s\,\mbox{proj}_{\rr_{\widetilde t_0+1,t_s}}(J_s)=J_s\,J_s=-1$. With this choice,
\[\debar_J(\T_{\epsilon_{t_0+s}})_J=t_{s-1}+1-\widetilde t_0+J_1^2+\ldots +J_u^2=t_{s-1}+1-\widetilde t_0-u\leq -u<0\,.\]
Therefore, $\T_{\epsilon_{t_0+s}}$ is never $\widetilde T$-regular when $t_{s-1}+1\leq\widetilde t_0<t_s$.
\item If $\widetilde t_0\leq t_{s-1}$, then for any $J\in\widetilde\torus$ we have
\begin{align*}
&(\T_{\epsilon_{t_0+s}})_J(x_0+v_1x_1+\ldots+v_{\widetilde t_0}x_{\widetilde t_0}+J_1\beta_1+\ldots+J_{\widetilde \tau}\beta_{\widetilde \tau})\\
&=x_0+v_{t_{s-1}+1}\langle v_{t_{s-1}+1},J_1\beta_1+\ldots+J_{\widetilde \tau}\beta_{\widetilde \tau}\rangle+\ldots
+v_{t_s}\langle v_{t_s},J_1\beta_1+\ldots+J_{\widetilde \tau}\beta_{\widetilde \tau}\rangle\,,
\end{align*}
whence
\begin{align*}
\debar_J(\T_{\epsilon_{t_0+s}})_J&\equiv 1+J_1\left(v_{t_{s-1}+1}\langle v_{t_{s-1}+1},J_1\rangle+\ldots+v_{t_s}\langle v_{t_s},J_1\rangle\right)\\
&\quad+\ldots+J_{\widetilde \tau}\left(v_{t_{s-1}+1}\langle v_{t_{s-1}+1},J_{\widetilde \tau}\rangle+\ldots+v_{t_s}\langle v_{t_s},J_{\widetilde \tau}\rangle\right)\\
&=1+J_1\,\mbox{proj}_{\rr_{t_{s-1}+1,t_s}}(J_1)+\ldots+J_{\widetilde \tau}\,\mbox{proj}_{\rr_{t_{s-1}+1,t_s}}(J_{\widetilde \tau})\,.
\end{align*}
We separate three sub-cases.
\begin{enumerate}
\item[(c.1)] If $\rr_{t_{s-1}+1,t_s}=\rr_{\widetilde t_{u-1}+1,\widetilde t_u}$ for some $u\in\{1,\ldots,\widetilde \tau\}$, then
\[\debar_J(\T_{\epsilon_{t_0+s}})_J\equiv1+J_u\,\mbox{proj}_{\rr_{t_{s-1}+1,t_s}}(J_u)=1+J_u^2=0\,.\]
Therefore, $\T_{\epsilon_{t_0+s}}$ is $\widetilde T$-regular when $(t_{s-1},t_s)=(\widetilde t_{u-1},\widetilde t_u)$ for some $u\in\{1,\ldots,\widetilde \tau\}$.
\item[(c.2)] If $\rr_{t_{s-1}+1,t_s}\subsetneq\rr_{\widetilde t_{u-1}+1,\widetilde t_u}$ for some $u\in\{1,\ldots,\widetilde \tau\}$, then
\[\debar_J(\T_{\epsilon_{t_0+s}})_J\equiv1+J_u\,\mbox{proj}_{\rr_{t_{s-1}+1,t_s}}(J_u)\]
equals $1$ if we choose $J_u\in\s_{\widetilde t_{u-1}+1,\widetilde t_u}$ with $J_u\perp\rr_{t_{s-1}+1,t_s}$. Therefore, $\T_{\epsilon_{t_0+s}}$ is not $\widetilde T$-regular when there exists $u\in\{1,\ldots,\widetilde \tau\}$ such that $\widetilde t_{u-1}<t_{s-1}<t_s\leq\widetilde t_u$ or $\widetilde t_{u-1}\leq t_{s-1}<t_s<\widetilde t_u$.
\item[(c.3)] Assume now that, for all $u\in\{1,\ldots,\widetilde \tau\}$, the subspace $\rr_{t_{s-1}+1,t_s}$ is not contained in $\rr_{\widetilde t_{u-1}+1,\widetilde t_u}$. As a consequence, $v_{t_{s-1}+1}\in\rr_{\widetilde t_{u-1}+1,\widetilde t_u}$ and $v_{t_s}\in\rr_{\widetilde t_{w-1}+1,\widetilde t_w}$ for some $u,w\in\{1,\ldots,\widetilde \tau\}$ with $u<w$ and
\[\debar_J(\T_{\epsilon_{t_0+s}})_J\equiv1+J_u\,\mbox{proj}_{\rr_{t_{s-1}+1,t_s}}(J_u)+J_{u+1}^2+\ldots+J_{w-1}^2+J_w\,\mbox{proj}_{\rr_{t_{s-1}+1,t_s}}(J_w)\,.\]
Here, the sum $J_{u+1}^2+\ldots+J_{w-1}^2$ should be read as $0$ when $u+1=w$. Let us choose $J\in\torus$ with $J_u=v_{t_{s-1}+1},J_w=v_{t_s}$: then
\[\debar_J(\T_{\epsilon_{t_0+s}})_J\equiv1+v_{t_{s-1}+1}^2+J_{u+1}^2+\ldots+J_{w-1}^2+v_{t_s}^2=1-(w-u+1)=u-w<0\,.\]
Therefore, $\T_{\epsilon_{t_0+s}}$ is not $\widetilde T$-regular when there exist $u,w\in\{1,\ldots,\widetilde \tau\}$ with $u<w$ such that $\widetilde t_{u-1}\leq t_{s-1}<\widetilde t_u$ and $\widetilde t_{w-1}\leq t_s<\widetilde t_w$.\qedhere
\end{enumerate}
\end{enumerate}
\end{enumerate}
\end{proof}

We are now ready for the announced classification.

\begin{corollary}\label{cor:classification}
Assume $A$ to be associative. Let $\tau,\widetilde\tau\in\{0,\ldots,n\}$ and let $T=(t_0,\ldots,t_\tau)$ and $\widetilde T=(\widetilde t_0,\ldots,\widetilde t_{\widetilde\tau})$ be two lists of steps for $V$. The inclusion $\reg_T(\Omega,A)\subseteq\reg_{\widetilde T}(\Omega,A)$ is equivalent to the equality $\reg_T(\Omega,A)=\reg_{\widetilde T}(\Omega,A)$ and to the following property: one among the lists $T,\widetilde T$ comprises the other, possibly preceded by some steps of the form $(m,m+1)$.
\end{corollary}

\begin{proof}
Let us assume $\reg_T(\Omega,A)\subseteq\reg_{\widetilde T}(\Omega,A)$ and apply Theorem~\ref{thm:fuetervscullen}.
\begin{itemize}
\item Assume $\widetilde t_0<t_0$. Since the last $t_0-\widetilde t_0$ $T$-Fueter variables $\T_{\epsilon_{\widetilde t_0+1}},\ldots,\T_{\epsilon_{t_0}}$ are $\widetilde T$-regular, property {\it 1} in Theorem~\ref{thm:fuetervscullen} yields, for the first $t_0-\widetilde t_0+1$ elements of $\widetilde T$:
\[\left(\widetilde t_0,\widetilde t_1,\widetilde t_2,\ldots,\widetilde t_{t_0-\widetilde t_0-2},\widetilde t_{t_0-\widetilde t_0-1},\widetilde t_{t_0-\widetilde t_0}\right)=\left(\widetilde t_0,\widetilde t_0+1,\widetilde t_0+2,\ldots,t_0-2,t_0-1,t_0\right)\,.\]
Since the $T$-Cullen variables $\T_{\epsilon_{t_0+1}},\ldots,\T_{\epsilon_{t_0+\tau}}$ are $\widetilde T$-regular, property {\it 2} in Theorem~\ref{thm:fuetervscullen} yields that $\widetilde T$ includes the whole list $\left(t_0,t_1,\ldots,t_\tau\right)$. Taking into account that $\widetilde t_0<\widetilde t_1<\ldots<\widetilde t_{\widetilde\tau}=n$ and that $t_\tau=n$, it follows immediately that $\widetilde\tau=\tau+t_0-\widetilde t_0$ and that
\begin{align*}
\widetilde T&=\left(\widetilde t_0,\widetilde t_1,\widetilde t_2,\ldots,\widetilde t_{t_0-\widetilde t_0-2},\widetilde t_{t_0-\widetilde t_0-1},\widetilde t_{t_0-\widetilde t_0},\widetilde t_{t_0-\widetilde t_0+1},\ldots,\widetilde t_{t_0-\widetilde t_0+\tau}\right)\\
&=\left(\widetilde t_0,\widetilde t_0+1,\widetilde t_0+2,\ldots,t_0-2,t_0-1,t_0,t_1,\ldots,t_\tau\right)\,.
\end{align*}
In other words: the list $\widetilde T$ comprises some steps of the form $(m,m+1)$, juxtaposed with the whole list $T$. In such a case, \[\widetilde\torus=\{\pm v_{\widetilde t_0+1}\}\times\{\pm v_{\widetilde t_0+2}\}\times\cdots\times\{\pm v_{t_0-1}\}\times\{\pm v_{t_0}\}\times\torus\,.\]
Taking into account~\eqref{eq:CRwellposed}, we conclude that $\reg_T(\Omega,A)=\reg_{\widetilde T}(\Omega,A)$.
\item Assume $t_0\leq\widetilde t_0$. Since the last $t_\tau-\widetilde t_0$ $T$-Cullen variables $\T_{\epsilon_{\widetilde t_0+1}},\ldots,\T_{\epsilon_{t_\tau}}$ are $\widetilde T$-regular, property {\it 2} in Theorem~\ref{thm:fuetervscullen} yields that $\widetilde t_0,\widetilde t_1,\ldots,\widetilde t_{\widetilde \tau}=n$ are the last $\widetilde \tau+1$ elements of $T$. Since the preceding $T$-Cullen variables $\T_{\epsilon_{t_0+1}},\ldots,\T_{\epsilon_{\widetilde t_0}}$ (if any) are $\widetilde T$-regular, property {\it 1} in Theorem~\ref{thm:fuetervscullen} yields, for the first $\widetilde t_0-t_0+1$ elements of $T$:
\[\left(t_0,t_1,t_2,\ldots,t_{\widetilde t_0-t_0-2},t_{\widetilde t_0-t_0-1},t_{\widetilde t_0-t_0}\right)=\left(t_0,t_0+1,t_0+2,\ldots,\widetilde t_0-2,\widetilde t_0-1,\widetilde t_0\right)\,.\]
Taking into account that $t_0<t_1<\ldots<t_\tau=n$ and that $\widetilde t_{\widetilde\tau}=n$, it follows immediately that $\tau=\widetilde\tau+\widetilde t_0-t_0$ and that
\begin{align*}
T&=\left(t_0,t_1,t_2,\ldots,t_{\widetilde t_0-t_0-2},t_{\widetilde t_0-t_0-1},t_{\widetilde t_0-t_0},t_{\widetilde t_0-t_0+1},\ldots,t_{\widetilde t_0-t_0+\widetilde \tau}\right)\\
&=\left(t_0,t_0+1,t_0+2,\ldots,\widetilde t_0-2,\widetilde t_0-1,\widetilde t_0,\widetilde t_1,\ldots,\widetilde t_{\widetilde \tau}\right)\,.
\end{align*}
In other words: the list $T$ comprises some steps of the form $(m,m+1)$, followed by the whole list $\widetilde T$. In such a case, \[\torus=\{\pm v_{t_0+1}\}\times\{\pm v_{t_0+2}\}\times\cdots\times\{\pm v_{\widetilde t_0-1}\}\times\{\pm v_{\widetilde t_0}\}\times\widetilde\torus\,.\]
Taking into account~\eqref{eq:CRwellposed}, we conclude that $\reg_T(\Omega,A)=\reg_{\widetilde T}(\Omega,A)$.
\end{itemize}
The proof is now complete.
\end{proof}

In addition to the set-wise classification provided in Corollary~\ref{cor:classification}, we plan to perform in a forthcoming paper a classification of $\{\reg_T(\Omega,A):T\mathrm{\ list\ of\ steps}\}$ up to bijections. For instance, we constructed in Example~\ref{ex:quaternions3} an explicit bijection $\reg_{(0,2,3)}(\hh,\hh)\to\reg_{(0,1,3)}(\hh,\hh)$, based on the orthonormal change of basis from $(1,i,j,k)$ to $(1,k,-j,i)$.


\section{$T$-functions and strongly $T$-regular functions}\label{sec:Tfunctions}

This section defines and studies, on a $T$-symmetric set $\Omega_D$, classes of functions $\Omega_D\to A$ having some special symmetries. Throughout the section, in addition to Assumption~\ref{ass:alternative}, we assume $D$ to be a subset of $\rr_{0,t_0}\times\rr^\tau$, invariant under the reflection $(\alpha,\beta)\mapsto(\alpha,\overline{\beta}^h)$ for every $h\in\{1,\ldots,\tau\}$. We recall that we have defined: in Definition~\ref{def:reflections}, the symbols $\beta\,J$ and $\overline{\beta}^h$ for all $\beta\in\rr^\tau,J\in\torus,h\in\{1,\ldots,\tau\}$; in Definition~\ref{def:Tsymmetric}, the symbol $\Omega_D:=\{\alpha+\beta\,J:(\alpha,\beta)\in D\}$, as well as the notion of $T$-symmetric set.


\subsection{$T$-stem functions}\label{subsec:Tstem}

As a preparation to work with functions $\Omega_D\to A$, we deal in this subsection with functions $D\to A\otimes\rr^{2^\tau}$.

\begin{remark}
The tensor product $A\otimes\rr^{2^\tau}$ is a bilateral $A$-module. Indeed, let $(E_K)_{K\in\mathscr{P}(\tau)}$ denote the canonical real vector basis of $\rr^{2^\tau}$: if $a\in A$ and if $C=\sum_{K\in\mathscr{P}(\tau)}E_KC_K\in A\otimes\rr^{2^\tau}$, we set $a\,C:=\sum_{K\in\mathscr{P}(\tau)}E_K(a\,C_K)$ and $Ca:=\sum_{K\in\mathscr{P}(\tau)}E_K(C_Ka)$.
\end{remark}

Our choice of the notation $(E_K)_{K\in\mathscr{P}(\tau)}$ is to avoid possible confusion with the basis of $A$ in the special case when $A=C\ell(0,n)$. Let us recall from~\cite{unifiednotion} the notion of $T$-stem function, which subsumes the notion of stem function of~\cite[Definition 4]{perotti} and follows the lines of its multivariate generalization~\cite[Definition 2.2]{gpseveral}. 

\begin{definition}
Let $F:D\to A\otimes\rr^{2^\tau}$ be a map $F=\sum_{K\in\mathscr{P}(\tau)}E_KF_K$ with components $F_K:D\to A$. The map $F$ is called a \emph{$T$-stem function} if
\[F_K(\alpha,\overline{\beta}^h)=\left\{
\begin{array}{ll}
F_K(\alpha,\beta)&\mathrm{if\ }h\not\in K\\
-F_K(\alpha,\beta)&\mathrm{if\ }h\in K
\end{array}
\right.\]
for all $K\in\mathscr{P}(\tau)$, for all $(\alpha,\beta)\in D$, and for all $h\in\{1,\ldots,\tau\}$. For such a function $F$: we say that $F$ belongs to $\mathscr{C}^\infty(D,A\otimes\rr^{2^\tau})$ if $F_K\in\mathscr{C}^\infty(D,A)$ for all $K\in\mathscr{P}(\tau)$; we say that $F$ is real analytic if $F_K$ is real analytic for every $K\in\mathscr{P}(\tau)$.
\end{definition}

Clearly, the set of $T$-stem functions is a right $A$-module. Moreover, given a $T$-stem function $F$ on $D$ and a point $\rho\in\rr_{0,t_0}$, setting $G(\alpha,\beta):=F(\alpha+\rho,\beta)$ defines a $T$-stem function $G$ on $D-(\rho,0)$. We add the following remarks. We point out that $D$ is also invariant under the composition of reflections $\beta\mapsto\overline{\beta}^H$ constructed in Definition~\ref{def:reflections}.

\begin{remark}
If $F$ is a $T$-stem function on $D$, then $F_K(\alpha,\overline{\beta}^H)=(-1)^{|H\cap K|}\,F_K(\alpha,\beta)$ for all $H,K\in\mathscr{P}(\tau)$ and $(\alpha,\beta)\in D$.
\end{remark}

Let us set, for $\beta=(\beta_1,\ldots,\beta_\tau)\in\rr^\tau$, the additional notation $\beta^2:=(\beta_1^2,\ldots,\beta_\tau^2)\in\rr^\tau$.

\begin{remark}\label{rmk:whitney}
Let $F=\sum_{K\in\mathscr{P}(\tau)}E_KF_K:D\to A\otimes\rr^{2^\tau}$ be a $T$-stem function of class $\mathscr{C}^\infty$ (or real analytic). Set $D':=\{(\alpha,\beta^2):(\alpha,\beta)\in D\}$. By Whitney's Theorem~\cite[page 160]{whitney}, there exist an open neighborhood $W$ of $D'$ in $\rr_{0,t_0}\times\rr^\tau$, with $D'=\{(\alpha,\gamma)\in W:\gamma_1,\ldots,\gamma_\tau\geq0\}$, and a finite sequence $\{G_K\}_{K\in\mathscr{P}(\tau)}$ in $\mathscr{C}^\infty(W,A)$ (or consisting of real analytic functions $G_K:W\to A$, respectively) such that, for all $(\alpha,\beta)\in D$, the following equalities hold true: $F_\emptyset(\alpha,\beta)=G_\emptyset(\alpha,\beta^2)$ and
\[F_K(\alpha,\beta)=\beta_{k_1}\cdots\beta_{k_p}\,G_K(\alpha,\beta^2)\]
if $K=\{k_1,\ldots,k_p\}$ with $1\leq k_1<\ldots<k_p\leq\tau$.
\end{remark}


\subsection{$T$-functions and strongly $T$-regular functions}\label{subsec:Tfunctions}

This subsection is devoted to the announced construction of classes of functions $\Omega_D\to A$ having some special symmetries: they are called \emph{$T$-functions}. Here, and in the rest of the current section, we assume $A$ to be associative.

\begin{definition}\label{def:Jk}
Assume $A$ to be associative. Let $J\in\torus,K\in\mathscr{P}(\tau)$. If $K=\emptyset$, we set $J_\emptyset:=1$. For $K\neq\emptyset$, say $K=\{k_1,\ldots,k_p\}$ with $1\leq k_1<\ldots<k_p\leq\tau$, we define $J_K:=J_{k_1}J_{k_2}\cdots J_{k_{p-1}}J_{k_p}$.
\end{definition}

For $J\in\torus$ and $K\in\mathscr{P}(\tau)$ fixed, the map $a\mapsto J_Ka$ is a right $A$-module isomorphism with inverse $a\mapsto J_K^{-1}a$. Here, $J_K^{-1}$ denotes the multiplicative inverse of the element $J_K$ of $A$. We are now ready to restate, in the associative case, the definition given in~\cite{unifiednotion} of $T$-function. This notion subsumes the notion of slice function,~\cite[Definition 5]{perotti}, in its associative sub-case. The definition follows the lines of~\cite[Definition 2.5]{gpseveral}, in its associative sub-case.

\begin{definition}\label{def:Tfunction}
Assume $A$ to be associative. Let $F=\sum_{K\in\mathscr{P}(\tau)}E_KF_K:D\to A\otimes\rr^{2^\tau}$ be a $T$-stem function. The \emph{induced} function $f=\I(F):\Omega_D\to A$, is defined at $x=\alpha+\beta\,J\in\Omega_D$ by the formula
\[f(x):=\sum_{K\in\mathscr{P}(\tau)}J_K\,F_K(\alpha,\beta)\,.\]
A function induced by a $T$-stem function is called a \emph{$T$-function}. We denote the class of $T$-functions $\Omega_D\to A$ by the symbol $\slice(\Omega_D,A)$. If $\Omega_D$ is a domain in $V$, then the elements of the intersection $\sr(\Omega_D,A):=\slice(\Omega_D,A)\cap\reg_T(\Omega_D,A)$ are called \emph{strongly $T$-regular} functions.
\end{definition}

The notions of $T$-function and strongly $T$-regular function are interesting when $\tau\geq1$. In the special case $\tau=0$, every subset of $V$ is $T$-symmetric, every domain $\Omega$ in $V$ is a $T$-symmetric domain and every function $f:\Omega\to A$ is a $T$-function, induced by a $T$-stem function $F=F_\emptyset$, which coincides with $f$ up to identifying $A\otimes\rr^0$ with $A$.

About the map $\I$ introduced in Definition~\ref{def:Tfunction}, we prove the following proposition.

\begin{proposition}\label{prop:isomorphismI}
Assume $A$ to be associative. The map $\I$ from the class of $T$-stem functions on $D$ to $\slice(\Omega_D,A)$ is well-defined. Moreover, the set $\slice(\Omega_D,A)$ is a right $A$-module and $\I$ is a right $A$-module isomorphism. Finally, $\sr(\Omega_D,A)$ is a right $A$-module.
\end{proposition}

\begin{proof}
In the forthcoming Lemma~\ref{lem:isomorphismIna} and Theorem~\ref{thm:representationformulana}, under weaker hypotheses, we will prove that $\I$ is a well-defined real space isomorphism. Additionally, it is clear from Definition~\ref{def:Tfunction} that, for any $T$-stem function $F:D\to A\otimes\rr^{2^\tau}$ and any $a\in A$, $\I(F a)=\I(F)a$. It follows at once that $\slice(\Omega_D,A)$ is a right $A$-module and that $\I$ is a right $A$-module isomorphism. Finally, Remark~\ref{rmk:mirrortranslation} guarantees that $\reg_T(\Omega_D,A)$ is a right $A$-module, whence $\sr(\Omega_D,A)=\slice(\Omega_D,A)\cap\reg_T(\Omega_D,A)$ is a right $A$-module, too.
\end{proof}

The properties of $T$-stem functions on $D$, along with Remark~\ref{rmk:mirrortranslation}, allow us to make the following observation.

\begin{remark}\label{rmk:mirrortranslationofslice}
Fix $p\in\rr_{0,t_0}$. If $f\in\slice(\Omega_D,A)$ (or $f\in\sr(\Omega_D,A)$), then setting $g(x):=f(x+p)$ defines a $g\in\slice(\Omega_D-p,A)$ (a $g\in\sr(\Omega_D-p,A)$, respectively).
\end{remark}

 For a $T$-function $f=\I(F)$, we now prove a Representation Formula along tori of the form $\torus_{\alpha+\beta I}=\Omega_{\{(\alpha,\beta)\}}=\alpha+\beta\,\torus$ with $\alpha\in\rr_{0,t_0},\beta\in\rr^\tau,I\in\torus$ (see Definition~\ref{def:Tsymmetric}). In connection to this formula, we also recover from $f$ the inducing $T$-stem function $F$. 

\begin{theorem}[Representation Formula for $T$-functions, associative case]\label{thm:representationformula}
Assume $A$ to be associative. If $f\in\slice(\Omega_D,A)$, then $f=\I(F)$ where $F=\sum_{K\in\mathscr{P}(\tau)}E_KF_K$ is a $T$-stem function whose $K$-component is
\[F_K(\alpha,\beta)=2^{-\tau}I_K^{-1}\,\sum_{H\in\mathscr{P}(\tau)}(-1)^{|K \cap H|}\,f(\alpha+\overline{\beta}^H\,I)\,.\]
 As a consequence: for all $(\alpha,\beta)\in D$ and all $I,J\in\torus$,
\begin{align}\label{eq:representationformula}
f(\alpha+\beta\,J)&=2^{-\tau}\sum_{K,H\in\mathscr{P}(\tau)}(-1)^{|K \cap H|}\,J_K\,I_K^{-1}\,f(\alpha+\overline{\beta}^H\,I)\\
&=\sum_{H\in\mathscr{P}(\tau)}\gamma_H\,f(\alpha+\overline{\beta}^H\,I)\,,\notag
\end{align}
where
\[\gamma_H:=2^{-\tau}\sum_{K\in\mathscr{P}(\tau)}(-1)^{|K \cap H|}\,J_K\,I_K^{-1}\,.\]
\end{theorem}

Theorem~\ref{thm:representationformula} is a special case of the forthcoming Theorem~\ref{thm:representationformulana}, whose hypotheses are weaker. In particular: if we fix $I\in\torus$, then every $f\in\slice(\Omega_D,A)$ is completely determined by its restriction $f_I$. 

We now draw two useful consequences. The first one concerns the norm $\Vert f\Vert$ of a continuous $T$-function $f$. We recall that $\omega=\omega_{\B,\B'}\geq1$ is a constant such that $\Vert xa\Vert\leq\omega\,\Vert x\Vert \,\Vert a\Vert$ for all $x\in V,a\in A$ (see Remark~\ref{rmk:norminequality}). Moreover, by Proposition~\ref{prop:norm}: if $A$ is associative and $\B'$ is a fitted distinguished basis of $A$, then $\omega=1$, whence $(1+\omega^2)^\tau=2^\tau$.

\begin{proposition}\label{prop:normTfunction}
Assume $A$ to be associative. Fix $f\in\slice(\Omega_D,A)$ and $I\in\torus$. For every nonempty subset $D'$ of $D$ and for $\Omega':=\Omega_{D'}\subseteq\Omega_D$,
\[\sup_{\Omega'}\Vert f\Vert\leq(1+\omega^2)^\tau\,\sup_{\Omega'_I}\Vert f\Vert\,.\]
\end{proposition}

A second useful consequence of the Representation Formula for $T$-functions is, that strongly $T$-regular functions are real analytic.

\begin{proposition}\label{prop:analyticTfunction}
Assume $A$ to be associative. If $f$ is strongly $T$-regular, then $f$ is real analytic.
\end{proposition}

Our current associativity assumption does not really play a role in Proposition~\ref{prop:normTfunction} and Proposition~\ref{prop:analyticTfunction}. We will therefore restate and prove these results as the forthcoming Proposition~\ref{prop:normTfunctionna} and Proposition~\ref{prop:analyticTfunctionna}, after defining $T$-functions and $T$-regular functions over general alternative $*$-algebras.


\subsection{Mirror $T$-stem functions}\label{subsec:mirrorTstem}

This subsection is devoted to some technical material, which will be useful in the forthcoming Proposition~\ref{prop:TkTfunction} to prove that the polynomial functions $\T_\k$ (defined in Subsection~\ref{subsec:T_k}) are $T$-functions, whence strongly $T$-regular. The proof of Proposition~\ref{prop:TkTfunction} will be based on the induction hypothesis that, if $|\k'|=|\k|-1$, then $\T_{\k'}=\I(F^{\k'})$ where $F^{\k'}$ is a $T$-stem function of a special kind, studied in the current subsection.

We recall that $A\otimes\rr^{2^\tau}$ is a bilateral $A$-module and give the next definition.

\begin{definition}\label{def:mirrorTstem}
Assume $A$ to be associative and let $A'$ denote the real subalgebra of $A$ generated by the mirror $\rr_{0,t_0}$. A $T$-stem function $F:D\to A\otimes\rr^{2^\tau}$ is called a \emph{mirror $T$-stem function} if it takes values in $A'\otimes\rr^{2^\tau}$.
\end{definition}

We endow the real subalgebra $A'$ of $A$ with a convenient system of generators.

\begin{remark}
Assume $A$ to be associative. We define $(v_H)_{H\in\mathscr{P}(t_0)}\subset A'$, as follows: $v_H:=1$ if $H=\emptyset$ and $v_H:=v_{h_1}\cdots v_{h_p}$ if $H=\{h_1,\ldots,h_p\}$ with $1\leq h_1<\ldots<h_p\leq t_0$.
Since the real vector space $\rr_{0,t_0}$ is the span of the anti-commuting imaginary units $v_0,v_1\ldots,v_{t_0}$, the finite sequence $(v_H)_{H\in\mathscr{P}(t_0)}$ is a system of generators for the real vector space $A'$. Moreover, there exists a subset $\mathscr{Q}\subseteq\mathscr{P}(t_0)$ such that $(v_H)_{H\in\mathscr{Q}}$ is a real vector basis of $A'$.
\end{remark}

\begin{example}
Let $A=C\ell(0,3)$ and $t_0=3$. If $v_1=e_1,v_2=e_2$ and $v_3=e_3$, then $(v_H)_{H\in\mathscr{P}(3)}$ is the standard basis of $A'=A$. If, instead, $v_1=e_1,v_2=e_2$ and $v_3=e_{12}$, then $(v_H)_{H\in\mathscr{P}(3)}$ is a system of $8$ generators for the $4$-dimensional subspace $A'=C\ell(0,2)$. In the latter case, a basis of $A'$ is $(v_H)_{H\in\mathscr{Q}}$ with $\mathscr{Q}=\mathscr{P}(2)\subsetneq\mathscr{P}(3)$.
\end{example}

It will be convenient to consider, for each mirror $T$-stem function, not only its components with respect to the basis $(v_H)_{H\in\mathscr{Q}}$ of $A'$ but also its (non unique) components with respect to the system $(v_H)_{H\in\mathscr{P}(t_0)}$ of generators for $A'$.

\begin{lemma}\label{lem:nonuniquedecomposition}
Given $(G_H)_{H\in\mathscr{P}(t_0)}$, where $G_H=\sum_{K\in\mathscr{P}(\tau)}E_KG_{H,K}:D\to A\otimes\rr^{2^\tau}$ is a $T$-stem function with real-valued $E_K$-components $G_{H,K}:D\to\rr$, the function $F:=\sum_{H\in\mathscr{P}(t_0)}G_Hv_H$ is a mirror $T$-stem function. Conversely, every mirror $T$-stem function $F$ on $D$ can be expressed in the form $F=\sum_{H\in\mathscr{P}(t_0)}G_Hv_H$, where $G_H=\sum_{K\in\mathscr{P}(\tau)}E_KG_{H,K}$ are $T$-stem functions $D\to\rr\otimes\rr^{2^\tau}$.
\end{lemma}

\begin{proof}
Let us prove the first implication: under the assumptions made, the function $F:=\sum_{H\in\mathscr{P}(t_0)}G_Hv_H$ is a mirror $T$-stem function on $D$. This fact follows from the equality $F=\sum_{K\in\mathscr{P}(\tau)}E_KF_K:D\to A'\otimes\rr^{2^\tau}$, where $F_K:=\sum_{H\in\mathscr{P}(t_0)}G_{H,K}v_H:D\to A'$, and from the following argument: for $(\alpha,\beta)\in D$,
\[F_K(\alpha,\overline{\beta}^h)=\sum_{H\in\mathscr{P}(t_0)}G_{H,K}(\alpha,\overline{\beta}^h)\,v_H\]
equals $-F_K(\alpha,\beta)$ when $h\in K$; it equals $F_K(\alpha,\beta)$ when $h\in\{1,\ldots,\tau\}\setminus K$.

Conversely, consider any mirror $T$-stem function $F=\sum_{K\in\mathscr{P}(\tau)}E_KF_K:D\to A'\otimes\rr^{2^\tau}$. Of course, with respect to the basis $(v_H)_{H\in\mathscr{Q}}$ of $A'$, the $K$-component $F_K:D\to A'$ decomposes as $F_K=\sum_{H\in\mathscr{Q}}G_{H,K}v_H$ for some unique functions $G_{H,K}:D\to\rr$. Take any $(\alpha,\beta)\in D$: for $h\in K$, the equality
\[0=F_K(\alpha,\overline{\beta}^h)+F_K(\alpha,\beta)=\sum_{H\in\mathscr{Q}}(G_{H,K}(\alpha,\overline{\beta}^h)+G_{H,K}(\alpha,\beta))\,v_H\]
yields that $G_{H,K}(\alpha,\overline{\beta}^h)=-G_{H,K}(\alpha,\beta)$ for all $H\in\mathscr{Q}$; for $h\in\{1,\ldots,\tau\}\setminus K$, the equality
\[0=F_K(\alpha,\overline{\beta}^h)-F_K(\alpha,\beta)=\sum_{H\in\mathscr{Q}}(G_{H,K}(\alpha,\overline{\beta}^h)-G_{H,K}(\alpha,\beta))\,v_H\]
yields that $G_{H,K}(\alpha,\overline{\beta}^h)=G_{H,K}(\alpha,\beta)$ for all $H\in\mathscr{Q}$. Therefore, for every $H\in\mathscr{Q}$, setting $G_H:=\sum_{K\in\mathscr{P}(\tau)}E_KG_{H,K}:D\to\rr\otimes\rr^{2^\tau}$ defines a $T$-stem function. Additionally: for every $H\in\mathscr{P}(t_0)\setminus\mathscr{Q}$, we define a $T$-stem function $G_H:D\to\rr\otimes\rr^{2^\tau}$ as $G_H:\equiv0$. By construction, $F=\sum_{H\in\mathscr{Q}}G_Hv_H=\sum_{H\in\mathscr{P}(t_0)}G_Hv_H$. The proof is now complete.
\end{proof}

It is useful to make the next remark, where we adopt the notation $X\bigtriangleup Y:=(X\setminus Y)\cup(Y\setminus X)$ for the symmetric difference of two sets $X$ and $Y$, as well as the notation $J_K$ set up in Definition~\ref{def:Jk}.

\begin{remark}\label{rmk:systemofgeneratorsformirror}
For all $H\in\mathscr{P}(t_0)$ and all $h\in\{1,\ldots,t_0\}$,
\[v_Hv_h=(-1)^{\sigma(H,h)}v_{H\bigtriangleup\{h\}}\,,\] 
where we set $\sigma(H,h)$ to be $0$ or $1$ depending on whether there is an even or odd number of elements of $H$ larger than, or equal to, $h$. Additionally, for a fixed $h$, the map $H\mapsto H\bigtriangleup\{h\}$ is an involutive bijection of $\mathscr{P}(t_0)$ onto itself.

Now fix $J=(J_1,\ldots,J_\tau)\in\torus$ and $u\in\{1,\ldots,\tau\}$. For all $K\in\mathscr{P}(\tau)$, we similarly have that
\[J_KJ_u=(-1)^{\sigma(K,u)}J_{K\bigtriangleup\{u\}}\]
and that the map $K\mapsto K\bigtriangleup\{u\}$ is an involutive bijection from $\mathscr{P}(\tau)$ onto itself. Moreover, for all $H\in\mathscr{P}(t_0)$,
\[v_HJ_u=(-1)^{|H|}J_uv_H\]
because $J_u$ anti-commutes with $v_h$ for all $h\in\{1,\ldots,t_0\}$.
\end{remark}

Thanks to the previous remark, we make the following observations.

\begin{remark}\label{rmk:M}
Assume $A$ to be associative. The set $\mathfrak{M}$ of mirror $T$-stem functions on $D$ is a real vector space. We have $\mathfrak{M}\,v_h=\mathfrak{M}$ for all $h\in\{1,\ldots,t_0\}$ and $\mathfrak{M}\,\phi=\mathfrak{M}$ for all functions $\phi:D\to\rr_{0,t_0}$ with $\phi(\alpha,\beta)$ constant in $\beta$.
\end{remark}

To prove the forthcoming Proposition~\ref{prop:TkTfunction}, we will also need the following technical lemma. We recall that $\I$ denotes the map constructed in Definition~\ref{def:Tfunction} and proven a right $A$-module isomorphism in Proposition~\ref{prop:isomorphismI}.

\begin{lemma}\label{lem:M}
Assume $A$ to be associative. Let $F$ be a mirror $T$-stem function on $D$. For any $u\in\{1,\ldots,\tau\}$, there exists a mirror $T$-stem function ${\,\!^uF}$ on $D$ such that: for all $(\alpha,\beta)\in D,J\in\torus$,
\[\I(F)(\alpha+\beta\,J)\,\beta_u\,J_u=\I({\,\!^uF})(\alpha+\beta\,J)\,.\]
\end{lemma}

\begin{proof}
According to Lemma~\ref{lem:nonuniquedecomposition}, $F$ can be expressed as $F=\sum_{H\in\mathscr{P}(t_0)}G_Hv_H$, where $G_H=\sum_{K\in\mathscr{P}(\tau)}E_KG_{H,K}$ are $T$-stem functions $D\to\rr\otimes\rr^{2^\tau}$. In particular, the $K$-component of $F$ is $F_K=\sum_{H\in\mathscr{P}(t_0)}G_{H,K}v_H$. For all $K\in\mathscr{P}(\tau)$, let us set $\widetilde{F}_K:=\sum_{H\in\mathscr{P}(t_0)}(-1)^{|H|}G_{H,K}v_H$, so that $F_K\,J_u=J_u\,\widetilde{F}_K$ by Remark~\ref{rmk:systemofgeneratorsformirror}. Thus,
\begin{align*}
\I(F)(\alpha+\beta\,J)\,\beta_u\,J_u&=\sum_{K\in\mathscr{P}(\tau)}J_K\,F_K(\alpha,\beta)\,\beta_u\,J_u=\sum_{K\in\mathscr{P}(\tau)}J_K\,J_u\,\beta_u\,\widetilde{F}_K(\alpha,\beta)\\
&=\sum_{K\in\mathscr{P}(\tau)}(-1)^{\sigma(K,u)}\,J_{K\bigtriangleup\{u\}}\,\beta_u\,\widetilde{F}_K(\alpha,\beta)\\
&=\sum_{K'\in\mathscr{P}(\tau)}J_{K'}\,(-1)^{\sigma(K'\bigtriangleup\{u\},u)}\,\beta_u\,\widetilde{F}_{K'\bigtriangleup\{u\}}(\alpha,\beta)\,.
\end{align*}
For the third and fourth equalities, we used Remark~\ref{rmk:systemofgeneratorsformirror} again.
Let us define ${\,\!^uF}:D\to A\otimes\rr^{2^\tau}$ by setting
\[{\,\!^uF}:=\sum_{K\in\mathscr{P}(\tau)}E_{K}\,(-1)^{\sigma(K\bigtriangleup\{u\},u)}\,\beta_u\,\widetilde{F}_{K\bigtriangleup\{u\}}\,.\]
If we prove that ${\,\!^uF}$ is a mirror $T$-stem function, then the desired equality $\I(F)(\alpha+\beta\,J)\,\beta_u\,J_u=\I({\,\!^uF})(\alpha+\beta\,J)$ will follow immediately. Since $\widetilde{F}_{K\bigtriangleup\{u\}}=\sum_{H\in\mathscr{P}(t_0)}(-1)^{|H|}G_{H,{K\bigtriangleup\{u\}}}v_H$, we have ${\,\!^uF}=\sum_{H\in\mathscr{P}(t_0)}{\,\!^uG}_Hv_H$, where
\begin{align*}
&{\,\!^uG}_H:=\sum_{K\in\mathscr{P}(\tau)}E_K\,{\,\!^uG}_{H,K}:D\to \rr\otimes\rr^{2^\tau}\,,\\
&{\,\!^uG}_{H,K}:=(-1)^{|H|+\sigma(K\bigtriangleup\{u\},u)}\,\beta_u\,G_{H,K\bigtriangleup\{u\}}:D\to\rr\,.
\end{align*}
Thanks to Lemma~\ref{lem:nonuniquedecomposition}, we are left with proving that each ${\,\!^uG}_H$ is a $T$-stem function. In other words, it suffices to verify the following symmetries:
\[{\,\!^uG}_{H,K}(\alpha,\overline{\beta}^h)=\left\{
\begin{array}{ll}
{\,\!^uG}_{H,K}(\alpha,\beta)&\mathrm{if\ }h\in\{1,\dots,\tau\}\setminus K\\
-{\,\!^uG}_{H,K}(\alpha,\beta)&\mathrm{if\ }h\in K
\end{array}\,.
\right.\]
We first assume $h\not\in K$, whence $h\in K\bigtriangleup\{h\}$ and $h\not\in K\bigtriangleup\{u\}$ for $u\neq h$: then
\begin{align*}
\left(\beta_u\,G_{H,K\bigtriangleup\{u\}}\right)_{|_{(\alpha,\overline{\beta}^h)}}&=-\beta_u\,(-G_{H,K\bigtriangleup\{u\}}(\alpha,\beta))=\beta_u\,G_{H,K\bigtriangleup\{u\}}(\alpha,\beta)&\mathrm{if\ }u=h\,,\\
\left(\beta_u\,G_{H,K\bigtriangleup\{u\}}\right)_{|_{(\alpha,\overline{\beta}^h)}}&=\beta_u\,G_{H,K\bigtriangleup\{u\}}(\alpha,\beta)&\mathrm{if\ }u\neq h\,,
\end{align*}
as desired. Assume, instead, $h\in K$, whence $h\not\in K\bigtriangleup\{h\}$ and $h\in K\bigtriangleup\{u\}$ for $u\neq h$: then
\begin{align*}
\left(\beta_u\,G_{H,K\bigtriangleup\{u\}}\right)_{|_{(\alpha,\overline{\beta}^h)}}&=-\beta_u\,G_{H,K\bigtriangleup\{u\}}(\alpha,\beta)&\mathrm{if\ }u=h\,,\\
\left(\beta_u\,G_{H,K\bigtriangleup\{u\}}\right)_{|_{(\alpha,\overline{\beta}^h)}}&=\beta_u\,(-G_{H,K\bigtriangleup\{u\}}(\alpha,\beta))=-\beta_u\,G_{H,K\bigtriangleup\{u\}}(\alpha,\beta)&\mathrm{if\ }u\neq h\,,
\end{align*}
as desired. The proof is now complete.
\end{proof}


\section{Series expansion and representation formula}\label{sec:seriesexpansion}

The aim of this section is studying $T$-regular functions more in depth, under suitable hypotheses about their domains. Throughout this section, we assume $A$ to be associative.

Subsection~\ref{subsec:seriesexpansion} provides $T$-regular functions with a series expansion on each ball centered at a point of the mirror. An Identity Principle and a Maximum Modulus Principle valid on $T$-slice domains follow. Subsection~\ref{subsec:representationformula} proves that $T$-regular functions on $T$-symmetric $T$-slice domains are automatically strongly $T$-regular, whence real analytic.


\subsection{Series expansion}\label{subsec:seriesexpansion}

In this subsection, we will expand $T$-regular functions into series, using the polynomial functions $\T_\k$ constructed in Section~\ref{sec:polynomials}. As a preparation for these series expansions, we prove that the $\T_\k$'s are strongly $T$-regular. 

\begin{proposition}\label{prop:TkTfunction}
Assume $A$ to be associative. For any $\k\in\zz^{t_0+\tau}$, the function $\T_\k:V\to A$ is strongly $T$-regular. Moreover, $\T_\k$ is induced by a mirror $T$-stem function.
\end{proposition}

\begin{proof}
We already established that the $\T_\k$'s are $T$-regular. To prove that they are strongly $T$-regular, we need to prove that they are $T$-functions, i.e., that for all $\k\in\nn^{t_0+\tau}$ there exists a $T$-stem function $F^\k=\sum_{K\in\mathscr{P}(\tau)}E_KF^\k_K:\rr_{0,t_0}\times\rr^\tau\to A\otimes\rr^{2^\tau}$ such that $\T_\k=\I(F^\k)$. We are actually going to prove this fact for a \emph{mirror} $T$-stem function $F^\k$.

For $\k\in\zz^{t_0+\tau}\setminus\nn^{t_0+\tau}$, it suffices to set $F^\k:\equiv0$. For $\k=(0,0,\ldots,0)$, we set $F^\k:\equiv E_{\emptyset}$. Now let us prove the thesis for $\k\in\nn^{t_0+\tau}$, assuming it true for $\k-\epsilon_s$ for all $s\in\{1,\ldots,t_0+\tau\}$. Using the induction hypothesis and Remark~\ref{rmk:M}, we make the following computation (where we omit the variable $x=\alpha+\beta\,J$ for the sake of readability). For $a:=\sum_{w=t_0+1}^{t_0+\tau}k_w, a_s:=a-k_s$ and $b_s:=\sum_{w=s+1}^{t_0+\tau}k_w$ (whence $a_s+b_s=\sum_{t_0<w<s}k_w+2b_s$), we have
\begin{align*}
&|\k|\,\T_\k=\sum_{s=1}^{t_0}k_s\,\T_{\k-\epsilon_s}\cdot\left(x_s-(-1)^a\,x_0v_s\right)+\sum_{s=t_0+1}^{t_0+\tau}(-1)^{b_s}\,k_s\,\T_{\k-\epsilon_s}\cdot\left(x_0+(-1)^{a_s}\,\beta_{s-t_0}J_{s-t_0}\right)\\
&=\sum_{s=1}^{t_0}k_s\,\I(F^{\k-\epsilon_s})\cdot\left(x_s-(-1)^a\,x_0v_s\right)+\sum_{s=t_0+1}^{t_0+\tau}(-1)^{b_s}\,k_s\,\I(F^{\k-\epsilon_s})\cdot\left(x_0+(-1)^{a_s}\,\beta_{s-t_0}J_{s-t_0}\right)\\
&=\sum_{s=1}^{t_0}\I(F^{\k-\epsilon_s}\phi_s)+\sum_{s=t_0+1}^{t_0+\tau}\I(F^{\k-\epsilon_s}\psi_s)+\sum_{s=t_0+1}^{t_0+\tau}(-1)^{a_s+b_s}\,k_s\,\I(F^{\k-\epsilon_s})\beta_{s-t_0}J_{s-t_0}\\
&=\I\left(\sum_{s=1}^{t_0}F^{\k-\epsilon_s}\phi_s+\sum_{s=t_0+1}^{t_0+\tau}F^{\k-\epsilon_s}\psi_s\right)+\sum_{u=1}^{\tau}(-1)^{d_u}k_{t_0+u}\,\I(F^{\k-\epsilon_{t_0+u}})\beta_uJ_u\,,
\end{align*}
where
\[\phi_s(\alpha,\beta):=k_s\,\left(x_s-(-1)^a\,x_0v_s\right),\quad\psi_s(\alpha,\beta):=(-1)^{b_s}\,k_s\,x_0,\quad d_u:=\sum_{t_0<w<t_0+u}k_w\,.\]
Now, Lemma~\ref{lem:M} defines, for $F=F^{\k-\epsilon_{t_0+u}}$ and for every $u\in\{1,\ldots,\tau\}$, a mirror $T$-stem function $\!^uF$ such that $\I(F)\beta_uJ_u=\I({\,\!^uF})$. We define $F^\k$ by means of the equality
\[|\k|\,F^\k=\sum_{s=1}^{t_0}F^{\k-\epsilon_s}\phi_s+\sum_{s=t_0+1}^{t_0+\tau}F^{\k-\epsilon_s}\psi_s+\sum_{u=1}^{\tau}(-1)^{d_u}k_{t_0+u}\,\!^uF^{\k-\epsilon_{t_0+u}}\,.\]
Using Remark~\ref{rmk:M}, we see that $F^\k$ is a mirror $T$-stem function. Moreover, $\T_\k=\I(F^\k)$.
\end{proof}

We are finally ready for the announced series expansion. We recall that $\omega=\omega_{\B,\B'}\geq1$ is a constant such that $\Vert xa\Vert\leq\omega\,\Vert x\Vert \,\Vert a\Vert$ for all $x\in V,a\in A$ (see Remark~\ref{rmk:norminequality}). Moreover, by Proposition~\ref{prop:norm}: if $A$ is associative and $\B'$ is a fitted distinguished basis of $A$, then $\omega=1$, whence $(1+\omega^2)^\tau=2^\tau$.

\begin{theorem}[Series expansion]\label{thm:seriesexpansion}
Assume $A$ to be associative. Let $\Omega$ be a domain in $V$ and $f\in\reg_T(\Omega,A)$. If $\Omega$ contains an open ball $B=B(p,R)$ of radius $R>0$ centered at a point $p$ in the mirror $\rr_{0,t_0}$, then the following series expansion is valid for $x\in B$:
\[f(x)=\sum_{k\in\nn}\sum_{|\k|=k}\T_\k(x-p)\frac{1}{\k!}\delta^{(0,\k)}f(p)\,.\]
Here, the series converges normally in $B$ because
\[\max_{\overline{B}(p,r_1)}\Big\Vert\sum_{|\k|=k}\T_\k(x-p)\frac{1}{\k!}\delta^{(0,\k)}f(p)\Big\Vert\leq(1+\omega^2)^\tau\,\omega^2\,\sqrt{2}\;\binom{k+m}{m}\,\left(\frac{r_1}{r_2}\right)^k\,\max_{\partial B_I(p,r_2)}\Vert f_I\Vert\]
whenever $0<r_1<r_2<R$ and $I\in\torus$.
\end{theorem}

\begin{proof}
Set $c_\k:=\frac{1}{\k!}\delta^{(0,\k)}f(p)$ for all $\k\in\nn^{t_0+\tau}$. For any $J=(J_1,\ldots,J_\tau)\in\torus$, we have $p\in\rr_{0,t_0}\subseteq\rr^{t_0+\tau+1}_J$ and $B_J$ is an open ball centered at $p$ in $\rr^{t_0+\tau+1}_J$, contained in $\Omega_J$. According to Definition~\ref{def:delta}, $c_\k=\frac{1}{\k!}\delta_J^{(0,\k)}f_J(p)=\frac{1}{\k!}\delta_{\B_J}^{(0,\k)}f_J(p)$ for all $\k\in\nn^{t_0+\tau}$. By Remark~\ref{rmk:alternatetaylor} and Lemma~\ref{lem:PTnabladelta},
\[f_J(x)=\sum_{k\in\nn}\sum_{|\k|=k}\P_\k^{\B_J}(x-p)J_\tau^{k_{t_0+\tau}}\cdots J_1^{k_{t_0+1}}\frac{1}{\k!}\delta_{\B_J}^{(0,\k)}f_J(p)=\sum_{k\in\nn}\sum_{|\k|=k}(\T_\k)_J(x-p)\,c_\k\,,\]
where the series converges normally in $B_J$. Therefore, the thesis will be proven if we can prove normal convergence of the series $\sum_{k\in\nn}\sum_{|\k|=k}\T_\k(x-p)\,c_\k$ in $B$.

Let us fix $r_1$ with $0<r_1<R$, set $C:=\overline{B}(p,r_1)$ as well as $p_k(x):=\sum_{|\k|=k}\T_\k(x-p)\,c_\k$, and prove that the number series $\sum_{k\in\nn}\max_C\Vert p_k\Vert$ converges. As a first step, we prove that each $p_k$ belongs to the set $\sr(V,A)$ of strongly $T$-regular functions on $V$. Indeed: Proposition~\ref{prop:TkTfunction} guarantees that $\T_\k\in\sr(V,A)$ for all $\k\in\nn^{t_0+\tau}$; Remark~\ref{rmk:mirrortranslationofslice} guarantees, since $p\in\rr_{0,t_0}$ and $V-p=V$, that $\sr(V,A)$ is invariant under composition with the translation $x\mapsto x-p$; and Proposition~\ref{prop:isomorphismI} guarantees that $\sr(V,A)$ is a right $A$-module. We are now ready to estimate $\max_C\Vert p_k\Vert$. For $I\in\torus$ fixed, Proposition~\ref{prop:normTfunction} yields the inequality $\max_C\Vert p_k\Vert\leq(1+\omega^2)^\tau\,\max_{C_I}\Vert p_k\Vert$. By applying Remark~\ref{rmk:alternatetaylor} to the $I$-monogenic function $\phi=f_I$ and taking into account again Lemma~\ref{lem:PTnabladelta}, we find that
\[\max_C\Vert p_k\Vert\leq(1+\omega^2)^\tau\,\max_{C_I}\Vert p_k\Vert\leq(1+\omega^2)^\tau\,\omega^2\,\sqrt{2}\;\binom{k+m}{m}\,\left(\frac{r_1}{r_2}\right)^k\,\max_{\partial B_I(p,r_2)}\Vert f_I\Vert\
\]
for any $r_2$ such that $r_1<r_2<R$. Since $\lim_{k\to+\infty}\frac{k+m}{k}\frac{r_1}{r_2}=\frac{r_1}{r_2}<1$, the
ratio test shows that number series $\sum_{k\in\nn}\max_C\Vert p_k\Vert$ converges, as desired. The proof is now complete.
\end{proof}

Besides its independent interest, Theorem~\ref{thm:seriesexpansion} allows to prove an Identity Principle over $T$-slice domains (see Definition~\ref{def:Tslicedomain}).

\begin{theorem}[Identity Principle]\label{thm:identityprinciple}
Assume $A$ to be associative. Let $\Omega\subseteq V$ be a $T$-slice domain and $f,g\in\reg_T(\Omega,A)$. If there exists $J\in\torus$ such that the $J$-slice $\Omega_J$ (whose dimension is $t_0+\tau+1$) contains a set of Hausdorff dimension $s\geq t_0+\tau$ where $f_J$ and $g_J$ coincide, then $f=g$ throughout $\Omega$.
\end{theorem}

\begin{proof}
We remark that $\Omega_J$ is a domain in $\rr^{t_0+\tau+1}_J$ because $\Omega$ is a $T$-slice domain. Since the difference $f_J-g_J$ vanishes in a subset of $\Omega_J$ having Hausdorff dimension $s\geq t_0+\tau$, Theorem~\ref{thm:identitymonogenic} guarantees that $f_J-g_J\equiv0$ throughout $\Omega_J$. The $T$-slice domain $\Omega$ certainly includes a point $p\in\rr_{0,t_0}$, whence an open ball $B=B(p,R)$ with $R>0$. By Definition~\ref{def:delta},
\[\delta^{(0,\k)}(f-g)(p)=\delta_J^{(0,\k)}(f_J-g_J)(p)=0\]
for all $\k\in\nn^{t_0+\tau}$. Let us apply Theorem~\ref{thm:seriesexpansion} to $f-g$ in $B$: since $\delta^{(0,\k)}(f-g)(p)=0$ for all $\k\in\nn^{t_0+\tau}$, it follows that $f-g\equiv0$ in $B$. For every $J'\in\torus$, we conclude that $f_{J'}-g_{J'}\equiv0$ in $B_{J'}$, which is an open subset of the domain $\Omega_{J'}$. Theorem~\ref{thm:identitymonogenic} guarantees that $f_{J'}-g_{J'}\equiv0$ throughout $\Omega_{J'}$. Thus, $f-g\equiv0$ in $\Omega$, as desired.
\end{proof}

For slice-regular quaternionic functions, the Identity Principle~\ref{thm:identityprinciple} was proven in~\cite{advances} over open balls centered at the origin, in~\cite{poli} over slice domains and in~\cite{altavillawithoutreal} over symmetric domains not intersecting the real line. For slice-regular octonionic functions on open balls centered at the origin, the Identity Principle~\ref{thm:identityprinciple}
was established in~\cite{rocky}. For an Identity Principle for slice-regular functions over a general alternative $*$-algebra, covering all symmetric domains, see~\cite{gpsalgebra}. The Identity Principle for slice-monogenic functions was proven in~\cite{israel}. We are now in a position to establish the following property of $T$-regular functions on a $T$-slice domain.

\begin{proposition}[Maximum Modulus Principle]
Assume $A$ to be associative. Let $\Omega$ be a $T$-slice domain in $V$ and $f\in\reg_T(\Omega,A)$. If the function $\Vert f\Vert:\Omega\to\rr$ has a global maximum point in $\Omega$, then $f$ is constant in $\Omega$.
\end{proposition}

\begin{proof}
Let $p$ be the global maximum point of $\Vert f\Vert:G\to\rr$ and let $J\in\torus$ be such that $p\in\Omega_J$. In particular, $p$ is a global maximum point for $\Vert f_J\Vert:\Omega_J\to\rr$. By applying Theorem~\ref{thm:monogenicmaximummodulus} to the $J$-monogenic function $f_J$, we conclude that $f_J\equiv f_J(p)=f(p)$ in $\Omega_J$. The Identity Principle~\ref{thm:identityprinciple} now yields that $f\equiv f(p)$ throughout $\Omega$.
\end{proof}

For slice-regular quaternionic functions, the Maximum Modulus Principle had been proven in~\cite{advances} over open balls centered at the origin, in~\cite{zerosopen} over slice domains and in~\cite{altavillawithoutreal} over symmetric domains not intersecting the real line. For slice-regular octonionic functions, it was established in~\cite{rocky} over open balls centered at the origin and in~\cite{wang} over slice domains. For slice-monogenic functions, it was proven in~\cite{renxu}.


\subsection{Representation formula on $T$-symmetric $T$-slice domains}\label{subsec:representationformula}

This subsection proves that $T$-regular functions on $T$-symmetric $T$-slice domains are automatically strongly $T$-regular (see Definition~\ref{def:Tfunction}), whence real analytic. This property subsumes a renowned property of quaternionic slice-regular functions, proven in~\cite[Theorem 3.1]{advancesrevised} (see also~\cite{ghilonislicebyslice}). It also subsumes the analogous property of Clifford slice-monogenic functions (see~\cite[Theorem 2.2.18]{librodaniele2} and references therein).

We recall that, in Definition~\ref{def:Tsymmetric}, we have defined $T$-symmetric sets as sets of the form $\Omega_D:=\{\alpha+\beta\,J:(\alpha,\beta)\in D\}$ for some $D\subseteq\rr_{0,t_0}\times\rr^\tau$. Throughout the present subsection, we assume $D$ to be a nonempty open subset of $\rr_{0,t_0}\times\rr^\tau$, invariant under the reflection $(\alpha,\beta)\mapsto(\alpha,\overline{\beta}^h)$ for every $h\in\{1,\ldots,\tau\}$. We point out that the $T$-symmetric open set $\Omega_D$ is a $T$-slice domain if, and only if, $D$ is connected and intersects $\rr_{0,t_0}\times\{0\}$. If this is the case, we are going to prove that $\reg_T(\Omega_D,A)=\sr(\Omega_D,A)$. As a preparation for the proof of this equality, we make a remark and establish a technical lemma.

\begin{remark}\label{rmk:technical}
Fix $H,K\in\mathscr{P}(\tau)$. Then
\[(-1)^{|(K\bigtriangleup\{u\})\cap H|}=\left\{
\begin{array}{ll}
(-1)^{|K \cap H|}&\mathrm{if\ }u\in\{1,\ldots,\tau\}\setminus H\\
-(-1)^{|K \cap H|}&\mathrm{if\ }u\in H
\end{array}
\right.\,.\]
This is because: if $u\not\in H$, then $(K\bigtriangleup\{u\})\cap H=K\cap H$; if $u\in H\setminus K$, then $|(K\bigtriangleup\{u\})\cap H|=|K\cap H|+1$; and if $u\in H\cap K$, then $|(K\bigtriangleup\{u\})\cap H|=|K\cap H|-1$. 
\end{remark}

Our technical lemma concerns the coefficient of $f(\alpha+\overline{\beta}^H\,I)$ in the Representation Formula~\eqref{eq:representationformula}.

\begin{lemma}\label{lem:technical}
Assume $A$ to be associative. Fix $H\in\mathscr{P}(\tau),I,J\in\torus,s\in\{1,\ldots,t_0\}$ and $u\in\{1,\ldots,\tau\}$. If we set
\[\gamma_H:=2^{-\tau}\sum_{K\in\mathscr{P}(\tau)}(-1)^{|K \cap H|}\,J_K\,I_K^{-1}\,,\]
then $v_s\,\gamma_H=\gamma_H\,v_s$ and
\[J_u\,\gamma_H=\left\{
\begin{array}{ll}
\gamma_H\,I_u&\mathrm{if\ }u\in\{1,\ldots,\tau\}\setminus H\\
-\gamma_H\,I_u&\mathrm{if\ }u\in H
\end{array}
\right.\,.\]
\end{lemma}

\begin{proof}
For $K\in\mathscr{P}(\tau)$, the element $v_s$ anti-commutes with $J_k$ and with $I_k$ for all $k\in K$, whence
\[v_s\,J_K\,I_K^{-1}=(-1)^{|K|}\,J_K\,v_s\,I_K^{-1}=J_K\,I_K^{-1}\,v_s\,.\]
It follows at once that $v_s\,\gamma_H=\gamma_H\,v_s$.

Proving the formula relating $J_u\,\gamma_H$ to $\gamma_H\,I_u$ requires several steps. For $K\in\mathscr{P}(\tau)$, we remark that
\[J_u\,J_K=(-1)^{\sigma(u,K)}\,J_{K\bigtriangleup\{u\}}\]
where $\sigma(u,K)$ is $0$ or $1$ according to whether the number of elements in $K$ less than, or equal to, $u$ is even or odd. For all $K'\in\mathscr{P}(\tau)$, we remark that $(-1)^{\sigma(u,K'\bigtriangleup\{u\})}=-(-1)^{\sigma(u,K')}$ and that
\[I_{K'}^{-1}\,I_u=-(I_u\,I_{K'})^{-1}=-((-1)^{\sigma(u,K')}\,I_{K'\bigtriangleup\{u\}})^{-1}=(-1)^{\sigma(u,K'\bigtriangleup\{u\})}I_{K'\bigtriangleup\{u\}}^{-1}\,.\]
We are now ready to begin the computation of $J_u\,\gamma_H$, as follows:
\begin{align*}
J_u\,\gamma_H&=2^{-\tau}\sum_{K\in\mathscr{P}(\tau)}(-1)^{|K \cap H|}\,J_u\,J_K\,I_K^{-1}\\
&=2^{-\tau}\sum_{K\in\mathscr{P}(\tau)}(-1)^{|K \cap H|}\,(-1)^{\sigma(u,K)}\,J_{K\bigtriangleup\{u\}}\,I_K^{-1}\\
&=2^{-\tau}\sum_{K'\in\mathscr{P}(\tau)}(-1)^{|(K'\bigtriangleup\{u\})\cap H|}\,(-1)^{\sigma(u,K'\bigtriangleup\{u\})}\,J_{K'}\,I_{K'\bigtriangleup\{u\}}^{-1}\\
&=2^{-\tau}\sum_{K'\in\mathscr{P}(\tau)}(-1)^{|(K'\bigtriangleup\{u\})\cap H|}\,J_{K'}\,I_{K'}^{-1}\,I_u\,.
\end{align*}
By Remark~\ref{rmk:technical}, if $u\not\in H$, then
\[J_u\,\gamma_H=2^{-\tau}\sum_{K'\in\mathscr{P}(\tau)}(-1)^{|K'\cap H|}\,J_{K'}\,I_{K'}^{-1}\,I_u=\gamma_H\,I_u\,,\]
as stated. By the same remark, if $u\in H$, then
\[J_u\,\gamma_H=-2^{-\tau}\sum_{K'\in\mathscr{P}(\tau)}(-1)^{|K'\cap H|}\,J_{K'}\,I_{K'}^{-1}\,I_u=-\gamma_H\,I_u\,,\]
as desired. The proof is now complete.
\end{proof}

We are now ready to prove that every $T$-regular function $f$ on a $T$-symmetric $T$-slice domain $\Omega_D$ is automatically strongly $T$-regular.

\begin{theorem}[Representation Formula]
Assume $A$ to be associative and the $T$-symmetric set $\Omega_D$ to be a $T$-slice domain. If $f\in\reg_T(\Omega_D,A)$, then $f$ is strongly $T$-regular and formula~\eqref{eq:representationformula} holds true for all $(\alpha,\beta)\in D$ and all $I,J\in\torus$. As a consequence, $f$ is real analytic.
\end{theorem}

\begin{proof}
Let us fix $I\in\torus$ and define
\[F_K(\alpha,\beta):=2^{-\tau}\,I_K^{-1}\,\sum_{H\in\mathscr{P}(\tau)}(-1)^{|K \cap H|}\,f_I(\alpha+\overline{\beta}^H\,I)\]
for every $K\in\mathscr{P}(\tau)$ and for $(\alpha,\beta)\in D$, as well as $F:=\sum_{K\in\mathscr{P}(\tau)}E_KF_K:D\to A\otimes\rr^{2^\tau}$. We claim that $F$ is a $T$-stem function and that $\widetilde{f}:=\I(F)$ is strongly $T$-regular. We now prove that $\widetilde{f}_I=f_I$. Indeed, for all $(\alpha,\beta)\in D$,
\begin{align*}
\widetilde{f}_I(\alpha+\beta I)&=\sum_{K\in \mathscr{P}(\tau)}I_K\,F_K(\alpha,\beta)=2^{-\tau}\sum_{K,H\in\mathscr{P}(\tau)}(-1)^{|K \cap H|}\,f_I(\alpha+\overline{\beta}^H\,I)\\
&=2^{-\tau}\,2^\tau\,f_I(\alpha+\overline{\beta}^\emptyset\,I)=f_I(\alpha+\beta I)
\end{align*}
because~\cite[Lemma 2.11]{gpseveral} implies that
\[\sum_{K\in\mathscr{P}(\tau)}(-1)^{|K\cap H|}=\left\{
\begin{array}{ll}
2^\tau&\mathrm{if\ }H=\emptyset\\
0&\mathrm{if\ }H\neq\emptyset
\end{array}
\right.\,.\]
Since $\widetilde{f}_I=f_I$, the Identity Principle~\ref{thm:identityprinciple} implies that $\widetilde{f}=f$ throughout the $T$-slice domain $\Omega_D$. The first part of the statement immediately follows. The last part of the statement now follows from Proposition~\ref{prop:analyticTfunction}.

We are left with proving our claim that $F$ is a $T$-stem function inducing a strongly $T$-regular function $\widetilde{f}$. For all $h\in\{1,\ldots,\tau\}$, we compute
\begin{align*}
F_K(\alpha,\overline{\beta}^h)&=2^{-\tau}\,I_K^{-1}\,\sum_{H\in\mathscr{P}(\tau)}(-1)^{|K \cap H|}\,f_I(\alpha+\overline{\beta}^{H\bigtriangleup\{h\}}\,I)\\
&=2^{-\tau}\,I_K^{-1}\,\sum_{H'\in\mathscr{P}(\tau)}(-1)^{|K\cap(H'\bigtriangleup\{h\})|}\,f_I(\alpha+\overline{\beta}^{H'}\,I)\,.
\end{align*}
By Remark~\ref{rmk:technical}: if $h\not\in K$, then
\[F_K(\alpha,\overline{\beta}^h)=2^{-\tau}\,I_K^{-1}\,\sum_{H'\in\mathscr{P}(\tau)}(-1)^{|K\cap H'|}\,f_I(\alpha+\overline{\beta}^{H'}\,I)=F_K(\alpha,\beta)\,;\]
if $h\in K$, then
\[F_K(\alpha,\overline{\beta}^h)=-2^{-\tau}\,I_K^{-1}\,\sum_{H'\in\mathscr{P}(\tau)}(-1)^{|K\cap H'|}\,f_I(\alpha+\overline{\beta}^{H'}\,I)=-F_K(\alpha,\beta)\,.\]
This completes the proof of the fact that $F$ is a $T$-stem function.

Let us now prove that $\widetilde{f}=\I(F)$ is $T$-regular (whence strongly $T$-regular) by fixing $J\in\torus$ and showing that $\debar_J\widetilde{f}_J\equiv0$. By Definition~\ref{def:Jmonogenic} and by Remark~\ref{rmk:incrementalratio},
\[\debar_J:=\debar_{\B_J}=\sum_{s=0}^{t_0}v_s\,D_{v_s}+\sum_{u=1}^\tau J_u\,D_{J_u}\,,\]
where for each $v$ in the basis $\B_J$ we use the temporary notation $D_v:\mathscr{C}^1(\Omega_J,A)\to\mathscr{C}^0(\Omega_J,A)$ with
\[D_v\phi(x):=\lim_{\rr\ni \varepsilon\to0}\varepsilon^{-1}\left(\phi(x+\varepsilon v)-\phi(x)\right)\,.\]
We want to apply the $\debar_J$ operator to $\widetilde{f}_J$. Let $(\alpha,\beta)\in D$: by formula~\eqref{eq:representationformula},
\[\widetilde{f}(\alpha+\beta\,J)=\sum_{H\in\mathscr{P}(\tau)}\gamma_H\,\widetilde{f}(\alpha+\overline{\beta}^H\,I)=\sum_{H\in\mathscr{P}(\tau)}\gamma_H\,f_I(\alpha+\overline{\beta}^H\,I)\,.\]
By Lemma~\ref{lem:technical},
\[(v_s\,D_{v_s}\widetilde{f}_J)(\alpha+\beta\,J)=\sum_{H\in\mathscr{P}(\tau)}\gamma_H\,(v_s\,D_{v_s}f_I)(\alpha+\overline{\beta}^H\,I)\]
for all $s\in\{0,\ldots,t_0\}$. Now let $u\in\{1,\ldots,\tau\}$ and let us compute $(J_u\,D_{J_u}\widetilde{f}_J)(\alpha+\beta\,J)$. We begin by defining, for $\varepsilon\in\rr$, the element $\beta_{u,\epsilon}\in\rr^\tau$ by means of the equality $\beta_{u,\epsilon}\,J=\beta\,J+\varepsilon J_u$. Thus,
\begin{align*}
(J_u\,D_{J_u}\widetilde{f}_J)(\alpha+\beta\,J)&=J_u\,\lim_{\rr\ni \varepsilon\to0}\varepsilon^{-1}\left(\widetilde{f}_J(\alpha+\beta_{u,\epsilon}\,J)-\widetilde{f}_J(\alpha+\beta\,J)\right)\\
&=\sum_{H\in\mathscr{P}(\tau)}J_u\,\gamma_H\,\lim_{\rr\ni \varepsilon\to0}\varepsilon^{-1}\left(f_I(\alpha+\overline{\beta_{u,\epsilon}}^H\,I)-f_I(\alpha+\overline{\beta}^H\,I)\right)\,.
\end{align*}
For $H\not\ni u$, using Lemma~\ref{lem:technical}, we find that
\begin{align*}
&J_u\,\gamma_H\,\lim_{\rr\ni \varepsilon\to0}\varepsilon^{-1}\left(f_I(\alpha+\overline{\beta_{u,\epsilon}}^H\,I)-f_I(\alpha+\overline{\beta}^H\,I)\right)\\
&=\gamma_H\,I_u\,\lim_{\rr\ni \varepsilon\to0}\varepsilon^{-1}\left(f_I(\alpha+\overline{\beta}^H\,I+\varepsilon I_u)-f_I(\alpha+\overline{\beta}^H\,I)\right)\\
&=\gamma_H\,I_u\,(D_{I_u}f_I)(\alpha+\overline{\beta}^H\,I)\,.
\end{align*}
For $H\ni u$, Lemma~\ref{lem:technical} yields
\begin{align*}
&J_u\,\gamma_H\,\lim_{\rr\ni \varepsilon\to0}\varepsilon^{-1}\left(f_I(\alpha+\overline{\beta_{u,\epsilon}}^H\,I)-f_I(\alpha+\overline{\beta}^H\,I)\right)\\
&=-\gamma_H\,I_u\,\lim_{\rr\ni \varepsilon\to0}\varepsilon^{-1}\left(f_I(\alpha+\overline{\beta}^H\,I-\varepsilon I_u)-f_I(\alpha+\overline{\beta}^H\,I)\right)\\
&=-\gamma_H\,I_u\,(-D_{I_u}f_I)(\alpha+\overline{\beta}^H\,I)=\gamma_H\,I_u\,(D_{I_u}f_I)(\alpha+\overline{\beta}^H\,I)\,.
\end{align*}
This proves that
\[(J_u\,D_{J_u}\widetilde{f}_J)(\alpha+\beta\,J)=\sum_{H\in\mathscr{P}(\tau)}\gamma_H\,I_u\,(D_{I_u}f_I)(\alpha+\overline{\beta}^H\,I)\,.\]
Using the equality $\debar_I=\sum_{s=0}^{t_0}v_s\,D_{v_s}+\sum_{u=1}^\tau I_u\,D_{I_u}$, we conclude that
\[(\debar_J\widetilde{f}_J)(\alpha+\beta\,J)=\sum_{H\in\mathscr{P}(\tau)}\gamma_H\,(\debar_If_I)(\alpha+\overline{\beta}^H\,I)\,.\]
Since $f$ is $T$-regular, $f_I$ is $I$-monogenic, i.e., $\debar_If_I\equiv0$. Overall, we conclude that $\debar_J\widetilde{f}_J\equiv0$, as desired. This completes the proof of our claim and the proof of the theorem.
\end{proof}


\section{Foundations for the nonassociative theory}\label{sec:nonassociative}

Let us go back to the general case when our $*$-algebra $A$ is alternative, but not necessarily associative. We are going to construct and study $T$-functions and strongly $T$-regular functions under this weaker hypothesis. Some preliminaries are in order.

\begin{definition}
For $J\in\torus,a\in A,K\in\mathscr{P}(\tau)$, we define $[J,a]_K$ and $]J,a[_K$ as follows. We define $[J,a]_\emptyset:=a=:\,]J,a[_\emptyset$. For $K\neq\emptyset$, say $K=\{k_1,\ldots,k_p\}$ with $1\leq k_1<\ldots<k_p\leq\tau$, we define
\begin{align*}
[J,a]_K&:=J_{k_1}(J_{k_2}(\ldots(J_{k_{p-1}}(J_{k_p}\,a))\ldots))\,,\\
]J,a[_K&:=J_{k_p}^{-1}(J_{k_{p-1}}^{-1}(\ldots(J_{k_2}^{-1}(J_{k_1}^{-1}\,a))\ldots))\,.
\end{align*}
\end{definition}

\begin{remark}\label{rmk:openclosedbrackets}
For $J\in\torus$ and $K\in\mathscr{P}(\tau)$ fixed, the map $a\mapsto[J,a]_K$ is a real vector space isomorphism from $A$ to itself, whose inverse is $a\mapsto\,]J,a[_K\,$ thanks to Artin's Theorem, see~\cite[Theorem 3.1]{schafer}.
\end{remark}

Throughout the section, in addition to Assumption~\ref{ass:alternative}, we assume $D$ to be a subset of $\rr_{0,t_0}\times\rr^\tau$, invariant under the reflection $(\alpha,\beta)\mapsto(\alpha,\overline{\beta}^h)$ for every $h\in\{1,\ldots,\tau\}$. We recall that we have defined: in Definition~\ref{def:reflections}, the symbols $\beta\,J$ and $\overline{\beta}^h$ for all $\beta\in\rr^\tau,J\in\torus,h\in\{1,\ldots,\tau\}$; in Definition~\ref{def:Tsymmetric}, the symbol $\Omega_D:=\{\alpha+\beta\,J:(\alpha,\beta)\in D\}$ for all $D\subseteq\rr_{0,t_0}\times\rr^\tau$. Using the notion of $T$-stem function from Subsection~\ref{subsec:Tstem}, we now generalize Definition~\ref{def:Tfunction} to the current nonassociative setting. This generalized definition, which subsumes the notion of slice function of~\cite[Definition 5]{perotti} and follows the lines of its multivariate generalization~\cite[Definition 2.5]{gpseveral}, has been announced in~\cite{unifiednotion}.

\begin{definition}\label{def:Tfunctionna}
Let $F=\sum_{K\in\mathscr{P}(\tau)}E_KF_K:D\to A\otimes\rr^{2^\tau}$ be a $T$-stem function. The \emph{induced} function $f=\I(F):\Omega_D\to A$, is defined at $x=\alpha+\beta\,J\in\Omega_D$ by the formula
\[f(x):=\sum_{K\in\mathscr{P}(\tau)}\left[J,F_K(\alpha,\beta)\right]_K\,.\]
A function induced by a $T$-stem function is called a \emph{$T$-function}. We denote the class of $T$-functions $\Omega_D\to A$ by the symbol $\slice(\Omega_D,A)$. If $\Omega_D$ is a domain in $V$, then the elements of the intersection $\sr(\Omega_D,A):=\slice(\Omega_D,A)\cap\reg_T(\Omega_D,A)$ are called \emph{strongly $T$-regular} functions.
\end{definition}

When $A$ is associative, Definition~\ref{def:Tfunctionna} is consistent with Definition~\ref{def:Tfunction} because in such a case the equality $[J,a]_K=J_Ka$ holds true for all $K\in\mathscr{P}(\tau)$ and all $a\in A$.

Once again, the notions of $T$-function and strongly $T$-regular function are interesting when $\tau\geq1$. In the special case $\tau=0$, every subset of $V$ is $T$-symmetric, every domain $\Omega$ in $V$ is a $T$-symmetric domain and every function $f:\Omega\to A$ is a $T$-function, induced by a $T$-stem function $F=F_\emptyset$, which coincides with $f$ up to identifying $A\otimes\rr^0$ with $A$.

We now provide a first study of the map $\I$ introduced in Definition~\ref{def:Tfunctionna}.

\begin{lemma}\label{lem:isomorphismIna}
The map $\I$ from the class of $T$-stem functions on $D$ to $\slice(\Omega_D,A)$ is well-defined and surjective. Moreover, the set $\slice(\Omega_D,A)$ is a real vector space and $\I$ is real linear map. Finally, $\sr(\Omega_D,A)$ is a real vector space.
\end{lemma}

\begin{proof}
Let us show that $\I(F)$ is well-defined for each $T$-stem function $F:D\to A\otimes\rr^{2^\tau}$. We begin by proving two properties valid for any $(\alpha,\beta)\in D,J\in\torus,K\in\mathscr{P}(\tau)$.
\begin{enumerate}
\item Assume $\beta_k=0$ for some $k\in\{1,\ldots,\tau\}$. If $k\in K$, the symmetry $F_K(\alpha,\overline{\beta}^k)=-F_K(\alpha,\beta)$ implies $F_K(\alpha,\beta)=0$, whence $[J,F_K(\alpha,\beta)]_K=0$. If $k\not\in K$, the expression $[J,F_K(\alpha,\beta)]_K$ still does not depend on the choice of $J_k$.
\item Let us apply the reflection $(\alpha,\beta)\mapsto(\alpha,\overline{\beta}^h)$ and the reflection $J=(J_1,\ldots,J_h,\ldots,J_\tau)\mapsto\widetilde{J}=(J_1,\ldots,-J_h,\ldots,J_\tau)$. If $h\in K$, then
\[[\widetilde{J},F_K(\alpha,\overline{\beta}^h)]_K=-[J,-F_K(\alpha,\beta)]_K=[J,F_K(\alpha,\beta)]_K\,.\]
If $h\not\in K$, then
\[[\widetilde{J},F_K(\alpha,\overline{\beta}^h)]_K=[J,F_K(\alpha,\beta)]_K\,.\]
\end{enumerate}
Suppose that, for some $(\alpha,\beta),(\alpha,\beta')\in D,J,J'\in\torus$, the equality $\alpha+\beta J=\alpha+\beta' J'$ holds: thanks to Remark~\ref{rmk:decomposedvariable} it is possible to prove, by finitely many applications of property {1} and property {2}, that $[J,F_K(\alpha,\beta)]_K=[J',F_K(\alpha,\beta')]_K$ for all $K\in\mathscr{P}(\tau)$. It follows that $\I$ is well-defined.

The map $\I$ is surjective by the very definition of $\slice(\Omega_D,A)$.

Additionally, Definition~\ref{def:Tfunctionna} and Remark~\ref{rmk:openclosedbrackets} immediately imply that, for all $\lambda,\mu\in\rr$ and all $T$-stem functions $F,G:D\to A\otimes\rr^{2^\tau}$, the equality $\I(\lambda F+\mu G)=\lambda \I(F)+\mu \I(G)$ holds true. It follows at once that $\slice(\Omega_D,A)$ is a real vector space and that $\I$ is a real linear map.

Finally, Remark~\ref{rmk:mirrortranslation} guarantees that $\reg_T(\Omega_D,A)$ is a real vector space, whence $\sr(\Omega_D,A)=\slice(\Omega_D,A)\cap\reg_T(\Omega_D,A)$ is a real vector space, too.
\end{proof}

As in the associative case, the following property can be established using the properties of $T$-stem functions on $D$, along with Remark~\ref{rmk:mirrortranslation}.

\begin{remark}\label{rmk:mirrortranslationofslicena}
Fix $p\in\rr_{0,t_0}$. If $f\in\slice(\Omega_D,A)$ (or $f\in\sr(\Omega_D,A)$), then setting $g(x):=f(x+p)$ defines a $g\in\slice(\Omega_D-p,A)$ (a $g\in\sr(\Omega_D-p,A)$, respectively).
\end{remark}

We now state and prove the generalization to the present nonassociative setting of Theorem~\ref{thm:representationformula}. In particular, for each $T$-function we recover a unique inducing $T$-stem function, thus proving that the real linear map $\I$ is an isomorphism. 

\begin{theorem}[Representation Formula for $T$-functions]\label{thm:representationformulana}
If $f=\I(F)\in\slice(\Omega_D,A)$, then the $K$-component of $F$ is
\[F_K(\alpha,\beta)=2^{-\tau}\sum_{H\in\mathscr{P}(\tau)}(-1)^{|K \cap H|}\;]I, f(\alpha+\overline{\beta}^H\,I)[_K\,.\]
As a consequence, the $T$-stem function $F$ inducing $f$ is unique and $\I$ is a real vector space isomorphism. Moreover,
\begin{equation}\label{eq:representationformulna}
f(\alpha+\beta\,J)=2^{-\tau}\sum_{K,H\in\mathscr{P}(\tau)}(-1)^{|K \cap H|}\,\Big[\,J\,,\,]I, f(\alpha+\overline{\beta}^H\,I)[_K\,\Big]_K
\end{equation}
for all $(\alpha,\beta)\in D$ and all $I,J\in\torus$.
\end{theorem}

\begin{proof}
Let us prove the first statement. We begin by computing, under the hypothesis $f=\I(F)$,
\begin{align*}
\sum_{H\in\mathscr{P}(\tau)}(-1)^{|K \cap H|}\,f(\alpha+\overline{\beta}^H\,I)&=\sum_{H\in\mathscr{P}(\tau)}(-1)^{|K \cap H|}\,\sum_{K'\in\mathscr{P}(\tau)}[I,F_{K'}(\alpha,\overline{\beta}^H)]_{K'}\\
&=\sum_{H\in\mathscr{P}(\tau)}(-1)^{|K \cap H|}\,\sum_{K'\in\mathscr{P}(\tau)}(-1)^{|H\cap K'|}\,[I,F_{K'}(\alpha,\beta)]_{K'}\\
&=\sum_{K'\in\mathscr{P}(\tau)}[I,F_{K'}(\alpha,\beta)]_{K'}\,\sum_{H\in\mathscr{P}(\tau)}(-1)^{|K \cap H|+|H\cap K'|}\\
&=2^\tau\,[I,F_K(\alpha,\beta)]_K
\end{align*}
for all $(\alpha,\beta)\in D,I\in\torus$. Here, we used the fact, proven in~\cite[Lemma 2.11]{gpseveral}, that
\[\sum_{H\in\mathscr{P}(\tau)}(-1)^{|K\cap H|+|H\cap K'|}=\left\{
\begin{array}{ll}
2^\tau&\mathrm{if\ }K'=K\\
0&\mathrm{if\ }K'\neq K
\end{array}
\right.\,.\]
Using Remark~\ref{rmk:openclosedbrackets}, we derive
\begin{align*}
F_K(\alpha,\beta)&=2^{-\tau}\;\Big]I,\sum_{H\in\mathscr{P}(\tau)}(-1)^{|K \cap H|}\,f(\alpha+\overline{\beta}^H\,I)\Big[_K\\
&=2^{-\tau}\sum_{H\in\mathscr{P}(\tau)}(-1)^{|K \cap H|}\;]I, f(\alpha+\overline{\beta}^H\,I)[_K\,,
\end{align*}
which is the first statement. As a consequence, $\I$ is injective and a real vector space isomorphism. Formula~\eqref{eq:representationformulna} follows immediately, if we take into account that $f(\alpha+\beta\,J)=\sum_{K\in\mathscr{P}(\tau)}[J,F_K(\alpha,\beta)]_K$ for all $(\alpha,\beta)\in D,J\in\torus$.
\end{proof}

As a consequence of the last theorem, if we fix $I\in\torus$, then every $f\in\slice(\Omega_D,A)$ is completely determined by its restriction $f_I$.

We now prove, in our present nonassociative setting, Proposition~\ref{prop:normTfunction}, which we restate for the reader's convenience. We recall that $\omega=\omega_{\B,\B'}\geq1$ is a constant such that $\Vert xa\Vert\leq\omega\,\Vert x\Vert \,\Vert a\Vert$ for all $x\in V,a\in A$ (see Remark~\ref{rmk:norminequality}).

\begin{proposition}\label{prop:normTfunctionna}
Fix $f\in\slice(\Omega_D,A)$ and $I\in\torus$. For every nonempty subset $D'$ of $D$ and for $\Omega':=\Omega_{D'}\subseteq\Omega_D$,
\[\sup_{\Omega'}\Vert f\Vert\leq(1+\omega^2)^\tau\,\sup_{\Omega'_I}\Vert f\Vert\,.\]
\end{proposition}

\begin{proof}
The Representation Formula~\eqref{eq:representationformulna} applies to $f$: for all $J\in\torus,(\alpha,\beta)\in D$ we have
\begin{align*}
f(\alpha+\beta\,J)=2^{-\tau}\sum_{K,H\in\mathscr{P}(\tau)}(-1)^{|K \cap H|}\,\Big[\,J\,,\,]I, f(\alpha+\overline{\beta}^H\,I)[_K\,\Big]_K\,.
\end{align*}
Thus,
\[\sup_{\Omega'}\Vert f\Vert\leq2^{-\tau}\sum_{K,H\in\mathscr{P}(\tau)}\sup_{(\alpha,\beta)\in D',J\in\torus}\,\left\Vert\Big[\,J\,,\,]I, f(\alpha+\overline{\beta}^H\,I)[_K\,\Big]_K\right\Vert\,.\]
If $K=\{k_1,\ldots,k_p\}$ with $1\leq k_1<\ldots<k_p\leq\tau$ and if we set $a:=f(\alpha+\overline{\beta}^H\,I),b:=\,]I,a[_K$, then
\begin{align*}
\left\Vert[J,b]_K\right\Vert&=\left\Vert J_{k_1}(J_{k_2}(\ldots(J_{k_{p-1}}(J_{k_p}\,b))\ldots))\right\Vert\leq\omega\left\Vert J_{k_2}(\ldots(J_{k_{p-1}}(J_{k_p}\,b))\ldots)\right\Vert\leq\ldots\\
&\leq\omega^{|K|}\Vert b\Vert=\omega^{|K|}\big\Vert\,]I, a[_K\big\Vert=\omega^{|K|}\left\Vert I_{k_p}^{-1}(I_{k_{p-1}}^{-1}(\ldots(I_{k_2}^{-1}(I_{k_1}^{-1}\,a))\ldots))\right\Vert\leq\ldots\leq\omega^{2|K|}\Vert a\Vert\,.
\end{align*}
We conclude that
\begin{align*}
\sup_{\Omega'}\Vert f\Vert&\leq2^{-\tau}\,\sum_{K,H\in\mathscr{P}(\tau)}\omega^{2|K|}\,\sup_{(\alpha,\beta)\in D'}\left\Vert f(\alpha+\overline{\beta}^H\,I)\right\Vert\\
&=\Big(2^{-\tau}\,\sum_{H\in\mathscr{P}(\tau)}\sup_{(\alpha,\beta)\in D'}\left\Vert f(\alpha+\overline{\beta}^H\,I)\right\Vert\Big)\,\Big(\sum_{K\in\mathscr{P}(\tau)}\omega^{2|K|}\Big)\\
&\leq\sup_{\Omega'_I}\Vert f\Vert\;\sum_{s=0}^\tau\sum_{|K|=s}\omega^{2s}=\sup_{\Omega'_I}\Vert f\Vert\;\sum_{s=0}^\tau\binom{\tau}{s}\omega^{2s}\\
&=(1+\omega^2)^\tau\,\sup_{\Omega'_I}\Vert f\Vert\,.
\end{align*}
Here, we took into account that $|\mathscr{P}(\tau)|=2^\tau$ and that there are exactly $\binom{\tau}{s}$ elements $K\in\mathscr{P}(\tau)$ such that $|K|=s$.
\end{proof}

The multiplicative constant $(1+\omega^2)^\tau$ appearing in the estimate of Proposition~\ref{prop:normTfunctionna} reduces to $2^\tau$ in the cases described in the next remark. These include the case of real octonions $\oo$ with the standard basis $\B'$.

\begin{remark}
Assume the trace function $t:A\to A$ to take values in the associative nucleus of $A$. If $\B'$ is a fitted distinguished basis of $A$, then Proposition~\ref{prop:norm} guarantees that $\omega=1$, whence $(1+\omega^2)^\tau=2^\tau$.
\end{remark}

We now prove the generalization to the present nonassociative setting of Proposition~\ref{prop:analyticTfunction}.

\begin{proposition}\label{prop:analyticTfunctionna}
If $f$ is strongly $T$-regular, then $f$ is real analytic.
\end{proposition}

\begin{proof}
Assume $f=\I(F)$ to be strongly $T$-regular. For $I\in\torus$ fixed, Theorem~\ref{thm:representationformulana} guarantees that the $K$-component $F_K$ of $F$ is
\[F_K(\alpha,\beta)=2^{-\tau}\sum_{H\in\mathscr{P}(\tau)}(-1)^{|K \cap H|}\;]I, f_I(\alpha+\overline{\beta}^H\,I)[_K\,.\]
Now, since $f$ is $T$-regular, its restriction $f_I$ is $I$-monogenic, whence real analytic by Remark~\ref{rmk:J-harmonic}. It follows at once that each $K$-component $F_K:D\to A$ of $F$ is real analytic. In this situation, Remark~\ref{rmk:whitney} guarantees that
\[F_K(\alpha,\beta)=\beta_{k_1}\cdots\beta_{k_p}\,G_K(\alpha,\beta^2)\]
for appropriate real analytic functions $\{G_K\}_{K\in\mathscr{P}(\tau)}$. The equality $f=\I(F)$ now implies that
\begin{align*}
&f(\alpha+\beta\,J)=\sum_{K\in\mathscr{P}(\tau)}\left[J,F_K(\alpha,\beta)\right]_K\\
&=F_{\emptyset}(\alpha,\beta)+\sum_{1\leq p\leq\tau}\sum_{1\leq k_1<\ldots<k_p\leq\tau} J_{k_1}(J_{k_2}(\ldots(J_{k_{p-1}}(J_{k_p}F_{k_1\ldots k_p}(\alpha,\beta)))\ldots))\\
&=G_{\emptyset}(\alpha,\beta^2)+\sum_{1\leq p\leq\tau}\sum_{1\leq k_1<\ldots<k_p\leq\tau} \beta_{k_1}J_{k_1}(\beta_{k_2}J_{k_2}(\ldots(\beta_{k_{p-1}}J_{k_{p-1}}(\beta_{k_p}J_{k_p}G_{k_1\ldots k_p}(\alpha,\beta^2)))\ldots))
\end{align*}
for all $(\alpha,\beta)\in D,J\in\torus$. Referring to the decomposition of the variable $x\in\Omega_D$ performed in Remark~\ref{rmk:decomposedvariable}, we conclude that
\begin{align*}
f(x)&=G_{\emptyset}(x^0,\Vert x^1\Vert^2,\ldots,\Vert x^\tau\Vert^2)\\
&\quad+\sum_{1\leq p\leq\tau}\sum_{1\leq k_1<\ldots<k_p\leq\tau} x^{k_1}(x^{k_2}(\ldots(x^{k_{p-1}}(x^{k_p}G_{k_1\ldots k_p}(x^0,\Vert x^1\Vert^2,\ldots,\Vert x^\tau\Vert^2)))\ldots))\,.
\end{align*}
Since $V\to\rr_{0,t_0}\times\rr^\tau,\ x\mapsto(x^0,\Vert x^1\Vert^2,\ldots,\Vert x^\tau\Vert^2)$ is a real polynomial map, it is clear that each map
\[\Omega_D\to A,\quad x\mapsto G_K(x^0,\Vert x^1\Vert^2,\ldots,\Vert x^\tau\Vert^2)\]
is real analytic. Since $V\to V\subseteq A,\ x=x^0+x^1+\ldots+x^\tau\mapsto x^{k}$ is a real linear map for all $k\in\{1,\ldots,\tau\}$, we conclude that $f$ is real analytic, as desired.
\end{proof}

To keep the main focus of the present work on the associative case, we postpone to a future paper any further study of the properties of $T$-regular functions over nonassociative $*$-algebras.


\section*{Appendix}\label{appendix}
\setcounter{section}{3}

This appendix provides proofs for the results presented in Section~\ref{sec:monogenicproperties}. For the reader's convenience, the statements are repeated here.


\subsection*{Proofs of the results in Subsection~\ref{subsec:monogenicpolynomial}}

\setcounter{theorem}{2}
\setcounter{equation}{3}

\begin{proposition}
Assume $A$ to be associative. For $\k=(k_1,\ldots,k_m)\in\zz^m$, the following properties hold true.
\begin{enumerate}
\item There exists a map $p_\k=(p_\k^0,p_\k^1,\ldots,p_\k^m):\rr^{m+1}\to\rr^{m+1}$ such that $\P_\k^\B=L_\B\circ p_\k\circ L_\B^{-1}$ for any hypercomplex basis $\B$ of a hypercomplex subspace $M$ of $A$.
\item The equality $(k_u+1)\,p^s_\k=k_s\,p^u_{\k+\epsilon_u-\epsilon_s}$ holds true for all distinct $s,u\in\{1,\ldots,m\}$.
\item For all $x\in M$,
\begin{align}
&\sum_{s=1}^mk_s\,v_s\,\P_{\k-\epsilon_s}^\B(x)=\sum_{s=1}^mk_s\,\P_{\k-\epsilon_s}^\B(x)\,v_s\,,
\\
&|\k|\,\P_\k^\B(x)=\sum_{s=1}^mk_s\,\zeta_s^\B\,\P_{\k-\epsilon_s}^\B(x)
\end{align}
\item The equality $\partial_s\P_\k^\B=k_s\,\P_{\k-\epsilon_s}^\B$ holds true for all $s\in\{1,\ldots,m\}$.
\end{enumerate}
\end{proposition}

\begin{proof}[Proof of Proposition~\ref{prop:fueterpolynomialsindependofbasis}]
We prove each property separately.
\begin{enumerate}
\item Let us set $p_\k:\equiv(0,0,\ldots,0)$ when $\k\in\zz^m\setminus\nn^m$, $p_\k:\equiv(1,0,\ldots,0)$ when $\k=(0,\ldots,0)$ and
\begin{align*}
|\k|\,p_\k^0&:=\sum_{w=1}^mk_w\left(p_{\k-\epsilon_w}^0x_w+p_{\k-\epsilon_w}^wx_0\right)\\
|\k|\,p_\k^u&:=-k_u\,p_{\k-\epsilon_u}^0x_0+\sum_{w=1}^mk_w\,p_{\k-\epsilon_w}^ux_w
\end{align*}
for all $u\in\{1,\ldots,m\}$ when $\k\in\nn^m\setminus\{(0,\ldots,0)\}$. With this definition, we claim that
\[k_u\,p_{\k-\epsilon_u}^s=k_s\,p_{\k-\epsilon_s}^u\]
for all $\k\in\zz^m,s,u\in\{1,\ldots,m\}$, which is equivalent to property {\it 2}. We now proceed by induction on $\k$.

For $\k\in\zz^m\setminus\nn^m$, we find that $L_\B\circ p_\k\circ L_\B^{-1}\equiv0v_0+0v_1+\ldots+0v_m=0\equiv\P_\k^\B$.

For $\k=(0,\ldots,0)$, we find that $L_\B\circ p_\k\circ L_\B^{-1}\equiv1v_0+0v_1+\ldots+0v_m=1\equiv\P_\k^\B$.

Let us now prove the thesis for $\k\in\nn^m\setminus\{(0,\ldots,0)\}$, assuming it true for $\k-\epsilon_s$ for all $s\in\{1,\ldots,m\}$. Using the definition of $\P_\k^\B$, the induction hypothesis, the definition of $p_\k$ and the claim, we find that
\begin{align*}
&|\k|\,\P_\k^\B\circ L_\B=\sum_{s=1}^mk_s(\P_{\k-\epsilon_s}^\B\circ L_\B)\cdot(x_s-x_0v_s)=\sum_{s=1}^m(L_\B\circ p_{\k-\epsilon_s})\cdot k_s\,(x_s-x_0v_s)\\
&=\sum_{s=1}^m\left(\sum_{u=0}^mv_u\,p_{\k-\epsilon_s}^u\right) k_s\,(x_s-x_0v_s)\\
&=\sum_{s=1}^mk_s\,p_{\k-\epsilon_s}^0x_s-\sum_{s=1}^mv_s\,k_s\,p_{\k-\epsilon_s}^0x_0+\sum_{u=1}^mv_u\left(\sum_{s=1}^mk_s\,p_{\k-\epsilon_s}^ux_s\right)-x_0\sum_{s,u=1}^mv_u\,v_s\,k_s\,p_{\k-\epsilon_s}^u\\
&=\sum_{s=1}^m\left(k_s\,p_{\k-\epsilon_s}^0x_s+k_s\,p_{\k-\epsilon_s}^sx_0\right)+\sum_{u=1}^mv_u\left(-k_u\,p_{\k-\epsilon_u}^0x_0+\sum_{s=1}^mk_s\,p_{\k-\epsilon_s}^ux_s\right)\\
&-x_0\sum_{1\leq s<u\leq m}\left(v_u\,v_s\,k_s\,p_{\k-\epsilon_s}^u+v_s\,v_u\,k_u\,p_{\k-\epsilon_u}^s\right)\\
&=\sum_{s=1}^mk_s\left(p_{\k-\epsilon_s}^0x_s+p_{\k-\epsilon_s}^sx_0\right)+\sum_{u=1}^mv_u\left(-k_u\,p_{\k-\epsilon_u}^0x_0+\sum_{s=1}^mk_s\,p_{\k-\epsilon_s}^ux_s\right)\\
&\quad+x_0\,\sum_{1\leq s<u\leq m}v_s\,v_u\,\left(k_s\,p_{\k-\epsilon_s}^u-k_u\,p_{\k-\epsilon_u}^s\right)\\
&=|\k|\,p_\k^0+\sum_{u=1}^mv_u\,|\k|\,p_\k^u+x_0\sum_{1\leq s<u\leq m}v_s\,v_u\,0\\
&=|\k|\,\sum_{u=0}^mv_u\,p_\k^u=|\k|\,L_\B\circ p_\k\,,
\end{align*}
as desired.
\item We prove this property, which also settles our previous claim, by induction on $\k$. The property is clearly true when $\k\in\zz^m\setminus\nn^m$ or $\k=(0,\ldots,0)$, which implies that (for all distinct $s,u\in\{1,\ldots,m\}$) $\k+\epsilon_u-\epsilon_s\in\zz^m\setminus\nn^m$ and $p^s_\k\equiv0\equiv p^u_{\k+\epsilon_u-\epsilon_s}$. We now prove the property for $\k\in\nn^m\setminus\{(0,\ldots,0)\}$, assuming it true for $\k+\epsilon_u-\epsilon_s-\epsilon_w$ for all $s,u,w\in\{1,\ldots,m\}$. Using first the definition of $p_\k^s$, then three separate instances of the induction hypothesis and finally the definition of $p^u_{\k+\epsilon_u-\epsilon_s}$, we find that
\begin{align*}
|\k|\,(k_u+1)\,p_\k^s&=(k_u+1)\,\left(-k_s\,p_{\k-\epsilon_s}^0x_0+\sum_{w=1}^mk_w\,p_{\k-\epsilon_w}^sx_w\right)\\
&=k_s\,\left(-(k_u+1)\,p_{\k-\epsilon_s}^0x_0+(k_u+1)\,p_{\k-\epsilon_s}^sx_s\right)\\
&\quad+(k_u+1)\,k_u\,p_{\k-\epsilon_u}^sx_u+\sum_{w\in\{1,\ldots,m\}\setminus\{u,s\}}k_w\,(k_u+1)\,p_{\k-\epsilon_w}^sx_w\\
&=k_s\,\left(-(k_u+1)\,p_{\k-\epsilon_s}^0x_0+(k_s-1)\,p_{\k+\varepsilon_u-2\epsilon_s}^ux_s\right)\\
&\quad+(k_u+1)\,k_s\,p_{\k-\epsilon_s}^ux_u+\sum_{w\in\{1,\ldots,m\}\setminus\{u,s\}}k_w\,k_s\,p_{\k+\epsilon_u-\epsilon_s-\epsilon_w}^ux_w\\
&=k_s\,\Big(-(k_u+1)\,p_{\k-\epsilon_s}^0x_0+(k_s-1)\,p_{\k+\varepsilon_u-2\epsilon_s}^ux_s\\
&\quad+(k_u+1)\,p_{\k-\epsilon_s}^ux_u+\sum_{w\in\{1,\ldots,m\}\setminus\{u,s\}}k_w\,p_{\k+\epsilon_u-\epsilon_s-\epsilon_w}^ux_w\Big)\\
&=k_s\,|\k+\epsilon_u-\epsilon_s|\,p^u_{\k+\epsilon_u-\epsilon_s}=|\k|\,k_s\,p^u_{\k+\epsilon_u-\epsilon_s}\,,
\end{align*}
whence the announced property follows immediately. This completes the induction step.
\item Using
\[\P_\k^\B\circ L_\B=L_\B\circ p_\k=\sum_{s=0}^mv_sp_\k^s=\sum_{s=0}^mp_\k^sv_s\,,\]
we prove formula~\eqref{eq:technical1} by the next computation:
\begin{align*}
\sum_{u=1}^mk_u\,v_u\,\P_{\k-\epsilon_u}^\B\circ L_\B&=\sum_{u=1}^mk_u\,v_u\,\sum_{s=0}^mv_s\,p_{\k-\epsilon_u}^s\\
&=\sum_{u=1}^mk_u\,p_{\k-\epsilon_u}^0\,v_u+\sum_{s,u=1}^mk_u\,p_{\k-\epsilon_u}^s\,v_u\,v_s\\
&=\sum_{s=1}^mk_s\,p_{\k-\epsilon_s}^0\,v_s+\sum_{s,u=1}^mk_s\,p_{\k-\epsilon_s}^u\,v_u\,v_s\\
&=\sum_{s=1}^mk_s\,\left(\sum_{u=0}^mp_{\k-\epsilon_s}^u\,v_u\right)\,v_s\\
&=\sum_{s=1}^mk_s\,\left(\P_{\k-\epsilon_s}^\B\circ L_\B\right)\,v_s\,.
\end{align*}
Here, the third equality follows from property {\it 2} (more precisely, from our previous claim). Taking into account the definition $\zeta_s^\B:=x_s-x_0\,v_s$ and the definition of $\P_\k^\B$, we prove formula~\eqref{eq:technical2} by the next computation:
\begin{align*}
\sum_{s=1}^mk_s\,\zeta_s^\B\,\P_{\k-\epsilon_s}^\B(x)&=\sum_{s=1}^mk_s\,x_s\,\P_{\k-\epsilon_s}^\B(x)-x_0\,\sum_{s=1}^mk_s\,v_s\,\P_{\k-\epsilon_u}^\B(x)\\
&=\sum_{s=1}^mk_s\,\P_{\k-\epsilon_s}^\B(x)\,x_s-x_0\,\sum_{s=1}^mk_s\,\P_{\k-\epsilon_u}^\B(x)\,v_s\\
&=\sum_{s=1}^mk_s\,\P_{\k-\epsilon_s}^\B(x)\,\zeta_s^\B=|\k|\,\P_\k^\B(x)\,.
\end{align*}
For the second equality, we used formula~\eqref{eq:technical1}.
\item We prove this property by induction on $\k$. The equality is obviously true when $\k\in\zz^m\setminus\nn^m$ or $\k=(0,0,\ldots,0)$, which implies $\partial_s\P_\k^\B\equiv0\equiv\P_{\k-\epsilon_s}^\B$. We prove it for $\k\in\nn^m\setminus\{(0,0,\ldots,0)\}$, assuming it true for $\P_{\k-\epsilon_u}^\B$ for all $u\in\{1,\ldots,m\}$, by means of the equalities
\begin{align*}
|\k|\,\partial_s\P_\k^\B(x)&=\sum_{u=1}^mk_u\,\partial_s\left(\P_{\k-\epsilon_u}^\B(x)\,\zeta_u^\B\right)\\
&=\sum_{u=1}^mk_u\left(\partial_s\P_{\k-\epsilon_u}^\B(x)\right)\,\zeta_u^\B+k_s\,\P_{\k-\epsilon_s}^\B(x)\\
&=\sum_{u\neq s}k_u\,k_s\P_{\k-\epsilon_s-\epsilon_u}^\B(x)\,\zeta_u^\B+k_s\,(k_s-1)\,\P_{\k-2\epsilon_s}^\B(x)+k_s\,\P_{\k-\epsilon_s}^\B(x)\\
&=k_s\,\left(|\k-\epsilon_s|\,\P_{\k-\epsilon_s}^\B(x)+\P_{\k-\epsilon_s}^\B(x)\right)\\
&=k_s\,|\k|\,\P_{\k-\epsilon_s}^\B(x)\,.
\end{align*}
The proof is now complete.\qedhere
\end{enumerate}
\end{proof}

\setcounter{theorem}{5}

\begin{proposition}
Assume $A$ to be associative and fix $k\in\nn$. Then $\left\{\P_\k^\B\right\}_{|\k|=k}$ is a right $A$-basis for $U^\B_k$. Namely, for all $P\in U^\B_k$, the equality
\begin{equation}
P(x)=\sum_{|\k|=k}\P_\k^\B(x)\,\frac{1}{\k!}\nabla_\B^{(0,\k)}P(0)
\end{equation}
holds true at all $x\in M$. 
\end{proposition}

\begin{proof}[Proof of Proposition~\ref{prop:basismonogenicpolynomials}]
For all $\k\in\nn^m$, the function $\P_\k^\B$ is a $k$-homogenous polynomial map by construction. We now check, by induction on $\k$, that $\P_\k^\B$ is left-monogenic with respect to $\B$. The Fueter polynomial $\P_{(0,0,\ldots,0)}^\B\equiv1$ is obviously left-monogenic with respect to $\B$. We can prove the same property for $\P_\k^\B$, assuming it true for $\P_{\k-\epsilon_s}^\B$ for all $s\in\{1,\ldots,m\}$, by means of the equalities
\begin{align*}
|\k|\,\debar_\B\P_\k^\B&=\sum_{s=1}^mk_s\,\debar_\B\left(\P_{\k-\epsilon_s}^\B(x)\zeta_s^\B\right)\\
&=\sum_{s=1}^mk_s\sum_{u=0}^mv_u\,\partial_u\left(\P_{\k-\epsilon_s}^\B(x)\zeta_s^\B\right)\\
&=\sum_{s=1}^mk_s\sum_{u=0}^mv_u\left(\left(\partial_u\P_{\k-\epsilon_s}^\B(x)\right)\zeta_s^\B+\P_{\k-\epsilon_s}^\B(x)\,\partial_u\zeta_s^\B\right)\\
&=\sum_{s=1}^mk_s\left(\debar_\B\P_{\k-\epsilon_s}^\B(x)\right)\zeta_s^\B+\sum_{s=1}^mk_s\left(\P_{\k-\epsilon_s}^\B(x)\,\partial_0\zeta_s^\B+v_s\,\P_{\k-\epsilon_s}^\B(x)\,\partial_s\zeta_s^\B\right)\\
&=0+\sum_{s=1}^mk_s\left(\P_{\k-\epsilon_s}^\B(x)(-v_s)+v_s\P_{\k-\epsilon_s}^\B(x)\right)\\
&\equiv0\,.
\end{align*}
For the last equality, we have applied formula~\eqref{eq:technical1}. We have therefore proven that $\left\{\P_\k^\B\right\}_{|\k|=k}\subset U^\B_k$.

We now prove that formula~\eqref{eq:monogenicpolynomialexpansion} holds true for all $P\in U^\B_k$. As a preparation, we make a remark. If $P:M\to A$ is a $k$-homogeneous polynomial map, i.e., $P(tx)=t^kP(x)$ for all $t>0$, then differentiating with respect to $t$ and evaluating at $t=1$ proves Euler's formula
\[\sum_{s=0}^mx_s\,\partial_sP(x)=k\,P(x)\,.\]
Combining this property with the equality $0\equiv\debar_\B P=\sum_{s=0}^mv_s\partial_sP$, we find that
\begin{align*}
k\,P(x)&=x_0\,\partial_0P(x)+x_s\,\sum_{s=1}^m\partial_sP(x)=-x_0\,\sum_{s=1}^mv_s\,\partial_sP(x)+\sum_{s=1}^mx_s\,\partial_sP(x)\\
&=\sum_{s=1}^m(x_s-x_0v_s)\,\partial_sP(x)=\sum_{s=1}^m\zeta_s^\B\,\partial_sP(x)\,.
\end{align*}
We are now ready to prove~\eqref{eq:monogenicpolynomialexpansion}, by induction on $k$. If $k=0$, equality~\eqref{eq:monogenicpolynomialexpansion} is true because $\P_{(0,0,\ldots,0)}\equiv1$ and $\nabla_\B^{(0,0,\ldots,0)}P=P$. If~\eqref{eq:monogenicpolynomialexpansion} is true for all $P\in U^\B_k$, we can prove it for any $P\in U^\B_{k+1}$ by means of the following chain of equalities:
\begin{align*}
(k+1)P(x)&=\sum_{s=1}^m\zeta_s^\B\,\partial_sP(x)\\
&=\sum_{s=1}^m\zeta_s^\B\,\sum_{|\k|=k}\P_\k^\B(x)\,\frac{1}{\k!}\nabla_\B^{(0,\k)}(\partial_sP)(0)\\
&=\sum_{s=1}^m\zeta_s^\B\,\sum_{|\k|=k}\P_\k^\B(x)\,\frac{1}{\k!}\nabla_\B^{(0,\k+\epsilon_s)}P(0)\\
&=\sum_{s=1}^m\zeta_s^\B\,\sum_{|\k'|=k+1,\,k'_s\neq0}\P_{\k'-\epsilon_s}^\B(x)\,\frac{1}{(\k'-\epsilon_s)!}\nabla_\B^{(0,\k')}P(0)\\
&=\sum_{|\k'|=k+1}\sum_{s=1}^mk'_s\,\zeta_s^\B\,\P_{\k'-\epsilon_s}^\B(x)\,\frac{1}{\k'!}\nabla_\B^{(0,\k')}P(0)\\
&=(k+1)\sum_{|\k'|=k+1}\P_{\k'}^\B(x)\,\frac{1}{\k'!}\nabla_\B^{(0,\k')}P(0)\,.
\end{align*}
For the last equality, we have applied formula~\eqref{eq:technical2} to $\P_{\k'}^\B$. This completes the proof of formula~\eqref{eq:monogenicpolynomialexpansion}.

We are left with proving that $\left\{\P_\k^\B\right\}_{|\k|=k}$ is a right $A$-basis for $U^\B_k$. Clearly, formula~\eqref{eq:monogenicpolynomialexpansion} implies that $\left\{\P_\k^\B\right\}_{|\k|=k}$ is a set of generators for the right $A$-module $U^\B_k$. We now prove that the elements of $\left\{\P_\k^\B\right\}_{|\k|=k}$ are linearly independent. Assume $\{a_\k\}_{|\k|=k}\subset A$ to be such that $\sum_{|\k|=k}\P_\k^\B(x)\,a_\k\equiv0$. If $\k'\in\nn^m$ has $|\k'|=k$, then
\[0\equiv\sum_{|\k|=k}\nabla_\B^{(0,\k')}\P_\k^\B(x)\,a_\k\equiv\k'!\,a_{\k'}\]
because of Remark~\ref{rmk:propertiesfueterpolynomials} and because $\k'\neq\k$ (with $|\k'|=k=|\k|$) implies $k'_s>k_s$ for at least one $s\in\{1,\ldots,m\}$. It follows that $a_{\k'}=0$, as desired. This completes the proof of the fact that $\left\{\P_\k^\B\right\}_{|\k|=k}$ is a right $A$-basis for $U^\B_k$.
\end{proof}


\subsection*{Proofs of the results in Subsection~\ref{subsec:monogenicintegral}}

\setcounter{theorem}{7}
\setcounter{equation}{12}

\begin{proposition}
Assume $A$ to be associative and fix a domain $G$ in the hypercomplex subspace $M$ of $A$. The following properties hold true for all integrable $\phi,\psi:G\to A$, all $a,b\in A$ and all disjoint domains $G_1,G_2$ in $M$:
\begin{enumerate}
\item $G=G_1\cup G_2\Rightarrow\int_{G}\phi\,d\sigma=\int_{G_1}\phi\,d\sigma+\int_{G_2}\phi\,d\sigma$.
\item $\int_{G}(a\phi+b\psi)\,d\sigma=a\int_{G}\phi\,d\sigma+b\int_{G}\psi\,d\sigma$.
\item $\int_{G}(\phi a+\psi b)\,d\sigma=\left(\int_{G}\phi\,d\sigma\right)a+\left(\int_{G}\psi\,d\sigma\right)b$.
\item $\left(\int_{G}\phi\,d\sigma\right)^c=\int_{G}\phi^c\,d\sigma$.
\item $\Vert\int_{G}\phi\,d\sigma\Vert\leq\int_{G}\Vert \phi\Vert\,d\sigma$.
\end{enumerate}
\end{proposition}

\begin{proof}[Proof of Proposition~\ref{prop:propertiesintegral}]
We first establish the following facts, valid for all integrable functions $\phi_0,\ldots,\phi_d:G\to\rr$ and for all real linear endomorphisms $\mathcal{F}:\rr^{d+1}\to\rr^{d+1}$:
\begin{align}
&\int_{G_1\cup G_2}\left(\phi_0,\ldots,\phi_d\right)\,d\sigma=\int_{G_1}\left(\phi_0,\ldots,\phi_d\right)\,d\sigma+\int_{G_2}\left(\phi_0,\ldots,\phi_d\right)\,d\sigma\label{eq:integraloverunion}\\
&\mathcal{F}\left(\int_{G}\left(\phi_0,\ldots,\phi_d\right)\,d\sigma\right)=\mathcal{F}\left(\int_{G}\phi_0\,d\sigma,\ldots,\int_{G}\phi_d\,d\sigma\right)=\int_{G}\mathcal{F}\left(\phi_0,\ldots,\phi_d\right)\,d\sigma\,,\label{eq:integralandendomorphismcommute}\\
&\left\Vert\int_{G}\left(\phi_0,\ldots,\phi_d\right)\,d\sigma\right\Vert_{\rr^{d+1}}\leq\int_{G}\left\Vert\left(\phi_0,\ldots,\phi_d\right)\right\Vert_{\rr^{d+1}}\,d\sigma\,.\label{eq:normofintegral}
\end{align}
Formula~\eqref{eq:integraloverunion} follows immediately from the fact that $\int_{G_1\cup G_2}\phi_s\,d\sigma=\int_{G_1}\phi_s\,d\sigma+\int_{G_2}\phi_s\,d\sigma$ for all $s\in\{0,\ldots,d\}$. To establish~\eqref{eq:integralandendomorphismcommute}, it suffices to find real coefficients $(\alpha_{su})_{s,u\in\{0,\ldots,d\}}$ such that the $s$-th component of $\mathcal{F}\left(\phi_0,\ldots,\phi_d\right)$ is $\alpha_{s0}\phi_0+\ldots+\alpha_{sd}\phi_d$
and to remark that
\[\alpha_{s0}\int_{G}\phi_0\,d\sigma+\ldots+\alpha_{sd}\int_{G}\phi_d\,d\sigma=\int_{G}(\alpha_{s0}\phi_0+\ldots+\alpha_{sd}\phi_d)\,d\sigma\]
for each $s\in\{0,\ldots,d\}$. Moreover,~\eqref{eq:normofintegral} can be proven as follows: if $p=\int_{G}\left(\phi_0,\ldots,\phi_d\right)\,d\sigma$ has $\left\Vert p\right\Vert_{\rr^{d+1}}=0$, then the inequality is obvious; else, the inequality follows from the next chain of inequalities:
\begin{align*}
\left\Vert p\right\Vert_{\rr^{d+1}}^2&=\langle p,p\rangle_{\rr^{d+1}}=\left\langle p,\left(\int_{G}\phi_0\,d\sigma,\ldots,\int_{G}\phi_d\,d\sigma\right)\right\rangle_{\rr^{d+1}}\\
&=p_0\int_{G}\phi_0\,d\sigma+\ldots+p_d\int_{G}\phi_d\,d\sigma=\int_{G}(p_0\phi_0+\ldots+p_d\phi_d)\,d\sigma\\
&=\int_{G}\left\langle p,\left(\phi_0,\ldots,\phi_d\right)\right\rangle_{\rr^{d+1}}\,d\sigma\leq \left\Vert p\right\Vert_{\rr^{d+1}}\,\int_{G}\left\Vert\left(\phi_0,\ldots,\phi_d\right)\right\Vert_{\rr^{d+1}}\,d\sigma\,.
\end{align*}
We now prove each of the properties listed in Proposition~\ref{prop:propertiesintegral} separately, using the notations $\mathcal{CON},\L_a,\R_a$ set in Remark~\ref{rmk:endomorphisms}.
\begin{enumerate}
\item Formula~\eqref{eq:integraloverunion}, along with Definition~\ref{def:volumeintegral}, immediately yields the desired equality.
\item By construction, $\phi\mapsto\int_{G}\phi\,d\sigma$ is real linear. It is also left $A$-linear because
\begin{align*}
a\int_{G}\phi\,d\sigma&=a\,L_{\B'}\left(\int_{G}L_{\B'}^{-1}\phi\,d\sigma\right)=(L_{\B'}\circ\L_a)\left(\int_{G}L_{\B'}^{-1}\phi\,d\sigma\right)\\
&=L_{\B'}\left(\int_{G}(\L_a\circ L_{\B'}^{-1})\phi\,d\sigma\right)=L_{\B'}\left(\int_{G}L_{\B'}^{-1}(a\phi)\,d\sigma\right)\\
&=\int_{G}a\,\phi\,d\sigma\,.
\end{align*}
Here, we have used~\eqref{eq:integralandendomorphismcommute} with $\mathcal{F}=\L_a$.
\item The map $\phi\mapsto\int_{G}\phi\,d\sigma$ is also right $A$-linear because
\begin{align*}
\left(\int_{G}\phi\,d\sigma\right)a&=L_{\B'}\left(\int_{G}L_{\B'}^{-1}\phi\,d\sigma\right)\,a=(L_{\B'}\circ\R_a)\left(\int_{G}L_{\B'}^{-1}\phi\,d\sigma\right)\\
&=L_{\B'}\left(\int_{G}(\R_a\circ L_{\B'}^{-1})\phi\,d\sigma\right)=L_{\B'}\left(\int_{G}L_{\B'}^{-1}(\phi a)\,d\sigma\right)\\
&=\int_{G}\phi\,a\,d\sigma\,.
\end{align*}
Here, we have used~\eqref{eq:integralandendomorphismcommute} with $\mathcal{F}=\R_a$.
\item Using~\eqref{eq:integralandendomorphismcommute} with $\mathcal{F}=\mathcal{CON}$, we find that
\begin{align*}
\left(\int_{G}\phi\,d\sigma\right)^c&=\left(L_{\B'}\left(\int_{G}L_{\B'}^{-1}\phi\,d\sigma\right)\right)^c=(L_{\B'}\circ\mathcal{CON})\left(\int_{G}L_{\B'}^{-1}\phi\,d\sigma\right)\\
&=L_{\B'}\left(\int_{G}(\mathcal{CON}\circ L_{\B'}^{-1})\phi\,d\sigma\right)=L_{\B'}\left(\int_{G}L_{\B'}^{-1}(\phi^c)\,d\sigma\right)\\
&=\int_{G}\phi^c\,d\sigma\,.
\end{align*}
\item Using the fact that $L_{\B'}:\rr^{d+1}\to M$ is an isometry, as well as inequality~\eqref{eq:normofintegral}, we get
\[\left\Vert\int_{G}\phi\,d\sigma\right\Vert=\left\Vert L_{\B'}\left(\int_{G}L_{\B'}^{-1}\phi\,d\sigma\right)\right\Vert=\left\Vert\int_{G}L_{\B'}^{-1}\phi\,d\sigma\right\Vert_{\rr^{d+1}}\leq\int_{G}\Vert L_{\B'}^{-1}\phi\Vert_{\rr^{d+1}}\,d\sigma=\int_{G}\Vert \phi\Vert\,d\sigma\,,\]
as desired.\qedhere
\end{enumerate}
\end{proof}

\setcounter{theorem}{9}

\begin{lemma}
$\left\Vert\int_{\partial B^{m+1}(0,1)}\phi(w)\,|do_w|\right\Vert\leq\int_{\partial B^{m+1}(0,1)}\Vert\phi(w)\Vert\,|do_w|$.
\end{lemma}

\begin{proof}[Proof of Lemma~\ref{lem:surfaceintegral}]
If we set, for all integrable functions $\phi_0,\ldots,\phi_d:\partial B^{m+1}(0,1)\to\rr$,
\[\int_{\partial B^{m+1}(0,1)}\left(\phi_0,\ldots,\phi_d\right)\,|do_w|:=\left(\int_{\partial B^{m+1}(0,1)}\phi_0\,|do_w|,\ldots,\int_{\partial B^{m+1}(0,1)}\phi_d\,|do_w|\right)\in\rr^{d+1}\,,\]
then $\int_{\partial B^{m+1}(0,1)}\phi(w)\,|do_w|=L_{\B'}\left(\int_{G}L_{\B'}^{-1}\circ\phi\,\,|do_w|\right)$. The thesis can be proven with the same technique used in the proof of Proposition~\ref{prop:propertiesintegral}.
\end{proof}

\begin{theorem}[Gauss]
Assume $A$ to be associative and fix a bounded domain $G$ in the hypercomplex subspace $M$ of $A$, with a $\mathscr{C}^1$ boundary $\partial G$. Then
\[\int_{\partial G}\psi\,dx^*\phi = \int_{G}\left((\psi\debar_\B)\,\phi+\psi\,(\debar_\B \phi)\right)\,d\sigma\]
for any $\phi,\psi\in\mathscr{C}^1(\overline{G},A)$.
\end{theorem}

\begin{proof}[Proof of Theorem~\ref{thm:gauss}]
Using~\cite[Proposition A.1.12]{librogurlebeck2}, we compute
\[d(\psi\,dx^*\phi)=d\psi\wedge(dx^*\,\phi)+\psi\,d(dx^*\,\phi)=d\psi\wedge dx^*\phi+(-1)^m\psi\,dx^*\wedge d\phi\,.\]
Now,
\begin{align*}
d\psi\wedge dx^*&=\left(\sum_{s=0}^m\partial_s\psi\,dx_s\right)\wedge\left(\sum_{u=0}^mv_u\,dx_u^*\right)\\
&=\sum_{s,u=0}^m\partial_s\psi\,v_u\,dx_s\wedge dx_u^*=\sum_{s=0}^m\partial_s\psi\,v_s\,dx_s\wedge dx_s^*\\
&=\sum_{s=0}^m\partial_s\psi\,v_s\,d\sigma = (\psi\debar_\B)\,d\sigma\,,\\
dx^*\wedge d\phi&=\left(\sum_{u=0}^mv_u\,dx_u^*\right)\wedge\left(\sum_{s=0}^m\partial_s\phi\,dx_s\right)\\
&=\sum_{u,s=0}^mv_u\,\partial_s\phi\,dx_u^*\wedge\,dx_s=\sum_{s=0}^mv_s\,\partial_s\phi\,dx_s^*\wedge\,dx_s\\
&=\sum_{s=0}^mv_s\,\partial_s\phi\,(-1)^md\sigma = (-1)^m(\debar_\B \phi)\,d\sigma\,,
\end{align*}
whence
\[d(\psi\,dx^*\phi)=(\psi\debar_\B)\,d\sigma\,\phi+\psi\,(\debar_\B \phi)\,d\sigma = \left((\psi\debar_\B)\,\phi+\psi\,(\debar_\B \phi)\right)\,d\sigma\,.\]
The thesis now immediately follows from Stokes' theorem,~\cite[Theorem A.2.18]{librogurlebeck2}.
\end{proof}

\setcounter{theorem}{12}

\begin{lemma}
If we fix $x\in M$, then the function
\[M\setminus\{x\} \to A\,,\quad y\mapsto E_m(y-x)\]
is both left- and right-monogenic with respect to $\B$.
\end{lemma}

\begin{proof}[Proof of Lemma~\ref{lem:cauchykernel}]
Since the kernel of $\debar_\B$ is invariant under composition with translations, it suffices to prove that $\debar_\B E_m\equiv0\equiv E_m\debar_\B$. Since $\Vert x\Vert^{m+1}=(\sum_{s=0}^mx_s^2)^{\frac{m+1}2}$, we find that
\[\partial_s(\Vert x\Vert^{-m-1}x_u)=-\Vert x\Vert^{-m-3}(m+1)x_sx_u+\Vert x\Vert^{-m-1}\delta_{su}=\Vert x\Vert^{-m-3}(-(m+1)x_sx_u+\delta_{su}\Vert x\Vert^2)\,.\]
Therefore,
\begin{align*}
\sigma_m\Vert x\Vert^{m+3}\debar_\B E_m(x)&=\Vert x\Vert^{m+3}\sum_{s=0}^mv_s\,\partial_s(\Vert x\Vert^{-m-1}x^c)\\
&=\sum_{s=0}^mv_s\,\Vert x\Vert^{m+3}\partial_s(\Vert x\Vert^{-m-1}x_0)-\sum_{s=0}^m\sum_{u=1}^mv_sv_u\,\Vert x\Vert^{m+3}\partial_s(\Vert x\Vert^{-m-1}x_u)\\
&=-(m+1)\sum_{s=0}^mv_sx_sx_0+\Vert x\Vert^2+(m+1)\sum_{s=0}^m\sum_{u=1}^mv_sv_ux_sx_u-\sum_{u=1}^mv_u^2\Vert x\Vert^2\\
&=-(m+1)xx_0+(m+1)x\sum_{u=1}^mv_ux_u+(m+1)\Vert x\Vert^2\\
&=-(m+1)xx^c+(m+1)\Vert x\Vert^2\\
&\equiv0\,,\\
\sigma_m\Vert x\Vert^{m+3} (E_m\debar_\B)(x)&=\Vert x\Vert^{m+3}\sum_{s=0}^m\partial_s(\Vert x\Vert^{-m-1}x^c)\,v_s\\
&=\sum_{s=0}^m\Vert x\Vert^{m+3}\partial_s(\Vert x\Vert^{-m-1}x_0)\,v_s-\sum_{s=0}^m\sum_{u=1}^m\Vert x\Vert^{m+3}\partial_s(\Vert x\Vert^{-m-1}x_u)\,v_u\,v_s\\
&=-(m+1)\sum_{s=0}^mx_sx_0v_s+\Vert x\Vert^2+(m+1)\sum_{s=0}^m\sum_{u=1}^mx_sx_uv_uv_s-\sum_{u=1}^m\Vert x\Vert^2v_u^2\\
&=-(m+1)x_0x+(m+1)\sum_{u=1}^mx_uv_ux+(m+1)\Vert x\Vert^2\\
&=-(m+1)x^cx+(m+1)\Vert x\Vert^2\\
&\equiv0\,,
\end{align*}
as desired.
\end{proof}

\begin{theorem}[Borel-Pompeiu]
Assume $A$ to be associative and fix a bounded domain $G$ in the hypercomplex subspace $M$ of $A$, with a $\mathscr{C}^1$ boundary $\partial G$. If $\phi\in\mathscr{C}^1(\overline{G},A)$, then
\[\int_{\partial G}E_m(y-x)\,dy^*\,\phi(y) - \int_{G}E_m(y-x)\,\debar_\B \phi(y)\,d\sigma_y
=\left\{\begin{array}{ll}
\phi(x)&\mathrm{if\ } x\in G\\
0&\mathrm{if\ } x\in M\setminus\overline{G}\,.
\end{array}\right.\]
\end{theorem}

\begin{proof}[Proof of Theorem~\ref{thm:borelpompeiu}]
If $x\in M\setminus\overline{G}$, the thesis follows directly from Theorem~\ref{thm:gauss} and Lemma~\ref{lem:cauchykernel}. We assume henceforth $x\in G$. For any $\varepsilon>0$ such that $\overline{B}^{m+1}(x,\varepsilon)\subset G$, set $G_\varepsilon:=G\,\setminus\,\overline{B}^{m+1}(x,\varepsilon)$. By Theorem~\ref{thm:gauss} and by Lemma~\ref{lem:cauchykernel}, we find that
\[\int_{\partial G_\varepsilon}E_m(y-x)\,dy^*\,\phi(y) = \int_{G_\varepsilon} E_m(y-x)\,\debar_\B \phi(y)\,d\sigma_y\,,\]
whence
\begin{align*}
&\int_{\partial B^{m+1}(x,\varepsilon)}E_m(y-x)\,dy^*\,\phi(y)-\int_{B^{m+1}(x,\varepsilon)}E_m(y-x)\,\debar_\B \phi(y)\,d\sigma_y\\
&= \int_{\partial G}E_m(y-x)\,dy^*\,\phi(y) - \int_{G} E_m(y-x)\,\debar_\B \phi(y)\,d\sigma_y\,.
\end{align*}
Now, for all $w\in\partial B^{m+1}(0,1)$ we have $E_m(\varepsilon w)=\sigma_m^{-1}\varepsilon^{-m+1}(\varepsilon w)^c=\varepsilon^{-m}
E_m(w)$. Moreover, Remark~\ref{rmk:sphere} guarantees that $dy^*=\varepsilon^m\,w\,|do_w|$, where $|do_w|$ denotes the surface element of the sphere $\partial B^{m+1}(0,1)$. Thus,
\begin{align*}
\int_{\partial B^{m+1}(x,\varepsilon)}E_m(y-x)\,dy^*\,\phi(y)
&=\int_{\partial B^{m+1}(0,1)}E_m(w)\,w\,\phi(x+\varepsilon w)\,|do_w|\\
&=\sigma_m^{-1} \int_{\partial B^{m+1}(0,1)}\phi(x+\varepsilon w)\,|do_w|\,.
\end{align*}
Since $\phi$ is continuous, we conclude that
\[\lim_{\varepsilon\to0}\int_{\partial B^{m+1}(x,\varepsilon)}E_m(y-x)\,dy^*\,\phi(y)=\sigma_m^{-1} \int_{\partial B^{m+1}(0,1)}|do_w|\,\phi(x)=\phi(x)\,.\]
Another application of Remark~\ref{rmk:sphere} yields
\begin{align*}
&\lim_{\varepsilon\to0}\int_{B^{m+1}(x,\varepsilon)} E_m(y-x)\,\debar_\B \phi(y)\,d\sigma_y\\
&=\lim_{\varepsilon\to0}\int_0^\varepsilon \int_{\partial B^{m+1}(0,1)} r^{-m}\,E_m(w)\,\debar_\B \phi(x+r w)\,|do_w|\,r^m\, dr\\
&=\sigma_m^{-1}\,\lim_{\varepsilon\to0}\int_0^\varepsilon \int_{\partial B^{m+1}(0,1)}w^c\,\debar_\B \phi(x+r w)\,|do_w|\, dr\\
&=0
\end{align*}
The thesis immediately follows.
\end{proof}

\setcounter{theorem}{15}

\begin{proposition}[Mean value property]
Assume $A$ to be associative and fix an open ball $B^{m+1}=B^{m+1}(x,R)$ in the hypercomplex subspace $M$ of $A$. If $\phi\in\mathscr{C}^1(\overline{B}^{m+1},A)$ is left-monogenic with respect to $\B$, then
\[\phi(x)=\frac1{\sigma_m}\int_{\partial B^{m+1}(0,1)}\phi(x+R w)\,|do_w|\,.\]
\end{proposition}

\begin{proof}[Proof of Proposition~\ref{prop:monogenicmeanvalue}]
Corollary~\ref{cor:monogeniccauchy} tells us that $\phi(x)=\int_{\partial B^{m+1}(x,\varepsilon)}E_m(y-x)\,dy^*\,\phi(y)$.
Moreover,  we already established, as a byproduct of the proof of Theorem~\ref{thm:borelpompeiu}, the equality $\int_{\partial B^{m+1}(x,\varepsilon)}E_m(y-x)\,dy^*\,\phi(y)=\sigma_m^{-1}\int_{\partial B^{m+1}(0,1)}\phi(x+\varepsilon w)\,|do_w|$ for all $\varepsilon>0$ such that $\overline{B}^{m+1}(x,\varepsilon)\subset G$.
\end{proof}


\subsection*{Proofs of the results in Subsection~\ref{subsec:reproducingkernel}}

To prove the results of Subsection~\ref{subsec:reproducingkernel}, several preliminary steps are needed.

\setcounter{section}{0}
\setcounter{theorem}{0}

The next remark recalls the definition of the Gegenbauer polynomials, see~\cite[Definition 9.22]{librogurlebeck2}, and some of their properties, see~\cite[Proposition 9.23 and Proof of Theorem 9.24]{librogurlebeck2}.

\begin{remark}\label{rmk:gegenbauer}
Fix $\mu\in\rr$ with $\mu>0$ and consider the sequence $\left\{C_h^\mu\right\}_{h\in\nn}$ of polynomial functions $[-1,1]\to\rr$ defined as
\[C_h^\mu(t_1):=\sum_{n=\lfloor\frac h2\rfloor}^h\binom{-\mu}{n}\binom{n}{2n-h}(-2t_1)^{2n-h}\,.\]
For $t_1\in[-1,1]$ fixed, the real power series $\sum_{h\in\nn}C_h^{\mu}(t_1)\,t_2^h$ centered at $0$ in the variable $t_2$ has radius of convergence $1$. Its sum is the function $(-1,1)\to\rr\,,\ t_2\mapsto\left(1-2t_1t_2 +t_2^2\right)^{-\mu}$.
\end{remark}

We recall the following properties from~\cite[\S 8.930, 8.933, 8.935, 8.937]{grashteynryzhik} and~\cite[Theorem 7.33.1]{libroszego}, valid for all $\mu>0$ and all $h\in\nn$.
\begin{itemize}
\item $C_0^\mu(t_1)\equiv1,C_1^\mu(t_1)=2\mu t_1$ and $h\,C_h^\mu(t_1)=2\,(h+\mu-1)\,t_1\,C_{h-1}^\mu(t_1)-(h+2\mu-2)\,C_{h-2}^\mu(t_1)$ for all $t_1\in[-1,1]$.
\item $\frac{d}{dt_1}C_h^\mu=2\,\mu\,C_{h-1}^{\mu+1}$.
\item $\max_{[-1,1]}|C_h^\mu|=C_h^\mu(1)=\binom{h+2\mu-1}{h}$.
\end{itemize}
We will only be interested in the cases when $\mu=\frac{m-1}2$ or $\mu=\frac{m+1}2$. For the latter case, we make a useful remark.

\begin{remark}\label{rmk:roottest}
Since $m\in\nn$, for $h\in\nn$ the expression $\binom{h+m}{h}=\binom{h+m}{m}=\frac{(h+m)(h+m-1)\cdots(h+1)}{m!}$ is polynomial of degree $m$ in the variable $h$, with rational coefficients. It follows at once that $\lim_{h\to+\infty}\binom{h+m}{m}^{\frac1h}=1$.
\end{remark}

We are now ready to prove an important technical lemma. We will use the temporary notations $\partial_{x_s}$ and $\partial_\B^x:=\sum_{s=0}^mv_s^c\,\partial_{x_s}$, instead of the usual $\partial_s$ and $\partial_\B:=\sum_{s=0}^mv_s^c\,\partial_s$, because two variables $x,y\in M$ are considered. 

\begin{lemma}\label{lem:technicalforexpansion}
Let us define $u:M\to M$ by setting $u(x):=\frac{x}{\Vert x\Vert}=\sigma_1E_1(x^c)$ when $x\neq0$, as well as $u(0):=1$. Assume $m\geq2$, set $\Lambda:=\{(x,y)\in M\times M : \Vert x\Vert<\Vert y\Vert\}$ and define $A_h:\Lambda\to\rr$ as 
\[A_h(x,y):=C_h^{\frac{m-1}2}(\langle u(x),u(y)\rangle)\,\Vert x\Vert^h\,.\]
Then, for all $(x,y)\in\Lambda$ and all $s\in\{1,\ldots,m\}$,
\begin{align*}
\Vert y-x\Vert^{-m+1}&=\sum_{h\in\nn}A_h(x,y)\,\Vert y\Vert^{-m-h+1}\,,\\
\partial_\B^x\Vert y-x\Vert^{-m+1}&=\sum_{k\in\nn}\partial_\B^x A_{k+1}(x,y)\,\Vert y\Vert^{-m-k}\,,
\end{align*}
where both series converge normally in $\Lambda$. Moreover, for any $k\in\nn$:
\begin{enumerate}
\item $\Lambda\to M,\ (x,y)\mapsto\partial_\B^x A_{k+1}(x,y)\,\Vert y\Vert^{k+1}$ is a polynomial function, $k$-homogeneous in the real variables $x_0,x_1\ldots,x_m$ and $(k+1)$-homogeneous in the real variables $y_0,y_1\ldots,y_m$;
\item $\Vert\partial_\B^x A_{k+1}(x,y)\Vert\leq \sqrt{2}\,(m-1)\,\binom{k+m}{m}\,\Vert x\Vert^k$ for all $(x,y)\in\Lambda$.
\end{enumerate}
\end{lemma}

\begin{proof}
We remark that $x=\Vert x\Vert\,u(x)$ and that $\Vert u(x)\Vert=1$ for all $x\in M$. For $(x,y)\in \Lambda$, by Theorem~\ref{thm:hypercomplexbasis},
\[\Vert y-x\Vert^2=\Vert y\Vert^2-2\langle x,y\rangle+\Vert x\Vert^2=\Vert y\Vert^2\left(1-2t_1t_2 +t_2^2\right)\,,\]
where
\[t_1:=\langle u(x),u(y)\rangle\in[-1,1],\quad t_2:=\frac{\Vert x\Vert}{\Vert y\Vert}\in[0,1)\,.\]
Thus,
\[\Vert y-x\Vert^{-m+1}=\Vert y\Vert^{-m+1}\left(1-2 t_1 t_2 +t_2^2\right)^{\frac{-m+1}2}\,.\]
Remark~\ref{rmk:gegenbauer} guarantees that, for $(x,y)\in\Lambda$,
\begin{align*}
\Vert y-x\Vert^{-m+1}&=\Vert y\Vert^{-m+1}\left(1-2t_1t_2 +t_2^2\right)^{\frac{-m+1}2}\\
&=\Vert y\Vert^{-m+1}\,\sum_{h\in\nn}C_h^{\frac{m-1}2}(t_1)\,t_2^h\\
&=\sum_{h\in\nn}C_h^{\frac{m-1}2}(\langle u(x),u(y)\rangle)\,\Vert x\Vert^h\,\Vert y\Vert^{-m-h+1}\\
&=\sum_{h\in\nn}A_h(x,y)\,\Vert y\Vert^{-m-h+1}\,,
\end{align*}
where the series converges normally in $\Lambda$, for the following reason. Let $T$ be a compact subset of $\Lambda$ and set $\varepsilon:=\min_{(x,y)\in T}\Vert y\Vert>0, R:=\max_{(x,y)\in T}\frac{\Vert x\Vert}{\Vert y\Vert}<1$: then
\[\left|A_h(x,y)\,\Vert y\Vert^{-m-h+1}\right|\leq\binom{h+m-2}{h}\,\Vert x\Vert^h\,\Vert y\Vert^{-m-h+1}\leq\binom{h+m-2}{m-2}\,\varepsilon^{-m+1}\,R^h\]
for all $(x,y)\in T$. Using Remark~\ref{rmk:roottest}, the root test shows that the number series
\[\sum_{h\in\nn}\max_{(x,y)\in T}\left|A_h(x,y)\,\Vert y\Vert^{-m-h+1}\right|\]
converges. Since the choice of $T$ was arbitrary, we have proven normal convergence in $\Lambda$.

We now remark that $A_0\equiv1$, that $A_1(x,y)\,\Vert y\Vert=2\frac{m-1}{2}\,\langle u(x),u(y)\rangle\,\Vert x\Vert\,\Vert y\Vert=(m-1)\,\langle x,y\rangle$ and that, for any $h\in\nn$,
\begin{align*}
A_h(x,y)\,\Vert y\Vert^h&:=C_h^{\frac{m-1}2}(\langle u(x),u(y)\rangle)\,\Vert x\Vert^h\,\Vert y\Vert^h\\
&=\sum_{j=0}^{\lfloor\frac{h}2\rfloor}\binom{(-m+1)/2}{h-j}\binom{h-j}{h-2j}\left(-2\langle u(x),u(y)\rangle\right)^{h-2j}\,\Vert x\Vert^h\,\Vert y\Vert^h\\
&=\sum_{j=0}^{\lfloor\frac{h}2\rfloor}\binom{(-m+1)/2}{h-j}\binom{h-j}{h-2j}\left(-2\langle x,y\rangle\right)^{h-2j}\,\Vert x\Vert^{2j}\,\Vert y\Vert^{2j}
\end{align*}
is an $h$-homogeneous polynomial map in the real variables $x_0,x_1\ldots,x_m$ and an $h$-homogeneous polynomial map in the real variables $y_0,y_1\ldots,y_m$.

For every $s\in\{1,\ldots,m\}$, we remark that $\partial_{x_s}A_0\equiv0$, that $\partial_{x_s}A_1(x,y)\,\Vert y\Vert=(m-1)\,\langle v_s,y\rangle=(m-1)\,y_s$ and that, for any $k\in\nn$, the function $\partial_{x_s}A_{k+1}(x,y)\,\Vert y\Vert^{k+1}$ is a polynomial function $\Lambda\to\rr$, which is $k$-homogeneous in the real variables $x_0,x_1\ldots,x_m$ and $(k+1)$-homogeneous in the real variables $y_0,y_1\ldots,y_m$. For the operator $\partial_\B^x:=\sum_{s=0}^mv_s^c\,\partial_{x_s}$, it follows that $\partial_\B^x A_0\equiv0$, that $\partial_\B^x A_1(x,y)\,\Vert y\Vert=(m-1)\,y^c$ and that, for any $k\in\nn$, the function $\partial_\B^xA_{k+1}(x,y)\,\Vert y\Vert^{k+1}$ is a polynomial function $\Lambda\to M$, which is $k$-homogeneous in the real variables $x_0,x_1\ldots,x_m$ and $(k+1)$-homogeneous in the real variables $y_0,y_1\ldots,y_m$.

Let us now prove that, for $s\in\{1,\ldots,m\}$, the series $\sum_{k\in\nn}\partial_{x_s}A_{k+1}(x,y)\,\Vert y\Vert^{-m-k}$ and $\sum_{k\in\nn}\partial_\B^x A_{k+1}(x,y)\,\Vert y\Vert^{-m-k}$ converge normally in $\Lambda$: their sums will then automatically equal $\partial_{x_s}\Vert y-x\Vert^{-m+1}$ and $\partial_\B^x\Vert y-x\Vert^{-m+1}$, respectively. We first establish the equalities
\begin{align*}
\partial_{x_s}\Vert x\Vert^{k+1}&=\frac{k+1}2\Vert x\Vert^{k-1}\partial_{x_s}x_s^2=(k+1)\,x_s\,\Vert x\Vert^{k-1}\,,\\
\partial_{x_s}\langle u(x),u(y)\rangle&=\partial_{x_s}(\Vert x\Vert^{-1}\langle x,u(y)\rangle)=-x_s\,\Vert x\Vert^{-3}\,\langle x,u(y)\rangle+\Vert x\Vert^{-1}\,\langle v_s,u(y)\rangle\\
&=(y_s\Vert y\Vert^{-1}-x_s\,\Vert x\Vert^{-1}\,\langle u(x),u(y)\rangle)\,\Vert x\Vert^{-1}\,,
\end{align*}
valid for all $k\in\nn$ and all $(x,y)\in\Lambda$ with $x\neq0$. Then, for the same choices of $k$ and $(x,y)$, we compute
\begin{align*}
\partial_{x_s}A_{k+1}(x,y)&=\partial_{x_s}\left(C_{k+1}^{\frac{m-1}2}(\langle u(x),u(y)\rangle)\,\Vert x\Vert^{k+1}\right)\\
&=\left(\frac{d}{dt_1}C_{k+1}^{\frac{m-1}2}\right)(\langle u(x),u(y)\rangle)\,(y_s\Vert y\Vert^{-1}-x_s\,\Vert x\Vert^{-1}\,\langle u(x),u(y)\rangle)\,\Vert x\Vert^k\\
&\quad+C_{k+1}^{\frac{m-1}2}(\langle u(x),u(y)\rangle)\,(k+1)\,x_s\,\Vert x\Vert^{k-1}\\
&=\Big((m-1)\,C_k^{\frac{m+1}2}(\langle u(x),u(y)\rangle)\,(y_s\Vert y\Vert^{-1}-x_s\,\Vert x\Vert^{-1}\,\langle u(x),u(y)\rangle)\\
&\quad+(k+1)\,C_{k+1}^{\frac{m-1}2}(\langle u(x),u(y)\rangle)\,x_s\,\Vert x\Vert^{-1}\Big)\,\Vert x\Vert^k\in\rr\,,\\
\partial_\B^x A_{k+1}(x,y)&=\sum_{s=0}^mv_s^c\,\partial_{x_s}A_{k+1}(x,y)\\
&=\Big((m-1)\,C_k^{\frac{m+1}2}(\langle u(x),u(y)\rangle)\,\big(u(y)-u(x)\,\langle u(x),u(y)\rangle\big)\\
&\quad+(k+1)\,C_{k+1}^{\frac{m-1}2}(\langle u(x),u(y)\rangle)\,u(x)\Big)^c\,\Vert x\Vert^k\in M\,.
\end{align*}
If $k\geq1$, we also have $\partial_{x_s}A_{k+1}(0,y)=0=\partial_\B^x A_{k+1}(0,y)$ for all $y\in M\setminus\{0\}$. We point out that $u(x)$ is a unitary element of $M$ and that $u(y)-u(x)\,\langle u(x),u(y)\rangle$ is the component orthogonal to $u(x)$ of the unitary vector $u(y)$. Now, fix $s\in\{1,\ldots,m\}$. For all $k\geq1$ and all $(x,y)\in\Lambda$, we obtain the estimate
\begin{align*}
\left|\partial_{x_s}A_{k+1}(x,y)\right|&\leq\left\Vert\partial_\B^x A_{k+1}(x,y)\right\Vert\\
&\leq\sqrt{(m-1)^2\,\binom{k+m}{m}^2+(k+1)^2\,\binom{k+m-1}{m-2}^2}\;\Vert x\Vert^k\\
&=\sqrt{(m-1)^2+(k+1)^2\frac{m^2(m-1)^2}{(k+m)^2(k+1)^2}\,}\,\binom{k+m}{m}\,\Vert x\Vert^k\\
&=(m-1)\,\sqrt{1+\frac{m^2}{(k+m)^2}\,}\,\binom{k+m}{m}\,\Vert x\Vert^k\\
&\leq \sqrt{2}\,(m-1)\,\binom{k+m}{m}\,\Vert x\Vert^k\,.
\end{align*}
This estimate is consistent with~\cite[Formula (11.12)]{librosommen}, because $\sqrt{2}(m-1)\frac{(k+m)(k+1)}{m(m-1)}=\sqrt{2}(\frac1mk^2+(\frac1m+1)k+1)\leq2\sqrt{2}(k^2+1)$ (as a consequence of the inequality $k\leq k^2$ and of our hypothesis $2\leq m$). Additionally, we remark that $\partial_\B^x A_1(x,y)=(m-1)\,u(y)^c$ has
\[\left|\partial_{x_s}A_{k+1}(x,y)\right|\leq\left\Vert\partial_\B^x A_1(x,y)\right\Vert=m-1\leq\sqrt{2}\,(m-1)\,\binom{m}{m}\,\Vert x\Vert^0\]
for all $(x,y)\in\Lambda$. Recalling our previous choices of $T$, $\varepsilon>0$ and $R<1$, we conclude that the inequalities
\begin{align*}
\left|\partial_{x_s}A_{k+1}(x,y)\,\Vert y\Vert^{-m-k}\right|&\leq\left\Vert\partial_\B^x A_{k+1}(x,y)\,\Vert y\Vert^{-m-k}\right\Vert\\
&\leq \sqrt{2}\,(m-1)\,\binom{k+m}{m}\,\Vert x\Vert^k\,\Vert y\Vert^{-m-k}\\
&\leq \sqrt{2}\,(m-1)\,\binom{k+m}{m}\,\varepsilon^{-m}\,R^k
\end{align*}
hold true for all $k\in\nn$ and all $(x,y)\in T$. Using Remark~\ref{rmk:roottest}, the root test immediately shows that the number series
\[\sum_{k\in\nn}\max_{(x,y)\in T}\left|\partial_{x_s}A_{k+1}(x,y)\,\Vert y\Vert^{-m-k}\right|,\qquad\sum_{k\in\nn}\max_{(x,y)\in T}\left\Vert\partial_\B^x A_{k+1}(x,y)\,\Vert y\Vert^{-m-k}\right\Vert\]
both converge. Since the choice of the compact subset $T$ of $\Lambda$ was arbitrary, we conclude that: the real-valued series $\sum_{k\in\nn}\partial_{x_s}A_{k+1}(x,y)\,\Vert y\Vert^{-m-k}$ converges normally in $\Lambda$ to the function $\Lambda\to\rr,\ (x,y)\mapsto\partial_{x_s}\Vert y-x\Vert^{-m+1}$; the $M$-valued series $\sum_{k\in\nn}\partial_\B^x A_{k+1}(x,y)\,\Vert y\Vert^{-m-k}$ converges normally in $\Lambda$ to the function $\Lambda\to M,\ (x,y)\mapsto\partial_\B^x \Vert y-x\Vert^{-m+1}$.
\end{proof}

Our technical preparation allows us to finally prove Theorem~\ref{thm:taylorkernel}.

\setcounter{section}{3}
\setcounter{theorem}{17}

\begin{theorem}
There exists a family $\big\{q_\k\big\}_{\k\in\nn^m}$, where, for $|\k|=k$, $q_{\k}:M\setminus\{0\}\to M$ is a $(k+1)$-homogeneous polynomial function, such that
\[E_m(y-x)=\frac1{\sigma_m}\sum_{k\in\nn}\sum_{|\k|=k}\P_\k^\B(x)\frac{q_\k(y)}{\Vert y\Vert^{m+2k+1}}\]
for all $(x,y)\in\Lambda:=\{(x,y)\in M\times M : \Vert x\Vert<\Vert y\Vert\}$. Here, the series converges normally in $\Lambda$ because
\[\Big\Vert\sum_{|\k|=k}\P_\k^\B(x)\frac{q_\k(y)}{\Vert y\Vert^{m+2k+1}}\Big\Vert\leq\sqrt{2}\,\binom{k+m}{m}\,\,\Vert x\Vert^k\,\Vert y\Vert^{-m-k}\,.\]
In particular, $E_m(y-x)$ is a real analytic function in the real variables $x_0,x_1\ldots,x_m,y_0,y_1\ldots,y_m$.
\end{theorem}

\begin{proof}[Proof of Theorem~\ref{thm:taylorkernel}]
We first assume $m\geq2$. For fixed $y\in M$, we may apply the operator $\partial_\B^x:=\sum_{s=0}^mv_s^c\,\partial_{x_s}$ to the function $M\setminus\{y\}\to\rr,\ x\mapsto\Vert y-x\Vert^{-m+1}$, to obtain
\begin{align*}
\partial_\B^x\Vert y-x\Vert^{-m+1}&=\sum_{s=0}^mv_s^c\,\partial_{x_s}\Vert y-x\Vert^{-m+1}\\
&=(m-1)\,\Vert y-x\Vert^{-m-1}\,\sum_{s=0}^mv_s^c\,(y_s-x_s)\\
&=(m-1)\,\Vert y-x\Vert^{-m-1}\,(y-x)^c\\
&=\sigma_m\,(m-1)\,E_m(y-x)\,.
\end{align*}
By Lemma~\ref{lem:technicalforexpansion}, if we define $P_k:\Lambda\to M$ as
\[P_k(x,y):=(m-1)^{-1}\,\left(\partial_\B^x A_{k+1}(x,y)\right)\Vert y\Vert^{-m-k}\,,\]
then
\begin{align*}
E_m(y-x)&=\sigma_m^{-1}\,(m-1)^{-1}\,\partial_\B^x\Vert y-x\Vert^{-m+1}\\
&=\sigma_m^{-1}\,\sum_{k\in\nn}P_k(x,y)
\end{align*}
for all $(x,y)\in\Lambda$. The convergence of the series is normal in $\Lambda$ because
\[\Vert P_k(x,y)\Vert\leq\sqrt{2}\,\binom{k+m}{m}\,\Vert x\Vert^k\,\Vert y\Vert^{-m-k}\]
for all $(x,y)\in\Lambda$. Moreover, Lemma~\ref{lem:technicalforexpansion} guarantees that $P_k(x,y)\,\Vert y\Vert^{m+2k+1}$ is a polynomial function, $k$-homogenous in $x_0,x_1,\ldots,x_m$ and $(k+1)$-homogeneous in $y_0,y_1,\ldots,y_m$, whence $P_k$ is real analytic.

If instead $m=1$, then $M$ is a plane and a $*$-subalgebra of $A$ that is $*$-isomorphic to $\cc$. We remark that
\begin{align*}
E_1(y-x)&=\frac{1}{\sigma_1}\frac{y^c-x^c}{\Vert y-x\Vert^2}=\sigma_1^{-1}(y-x)^{-1}=\sigma_1^{-1}y^{-1}\left(1-xy^{-1}\right)^{-1}=\frac{1}{\sigma_1}\sum_{k\in\nn}P_k(x,y)
\end{align*}
where $P_k(x,y):=y^{-1}(xy^{-1})^k=x^ky^{-k-1}$ and where the series $\sum_{k\in\nn}P_k$ converges normally in $\Lambda$ because $\Vert P_k(x,y)\Vert=\Vert x\Vert^k\,\Vert y\Vert^{-k-1}$. Moreover, $P_k(x,y)\,\Vert y\Vert^{2k+2}=x^k(y^c)^{k+1}$ is a polynomial function, $k$-homogenous in $x_0,x_1$ and $(k+1)$-homogeneous in $y_0,y_1$.

Now take any $m\in\nn\setminus\{0\}$. We have proven, in particular that $\Lambda\to M,\ (x,y)\mapsto E_m(y-x)=\sum_{k\in\nn}P_k(x,y)$ is a real analytic function in the real variables $x_0,x_1\ldots,x_m,y_0,y_1\ldots,y_m$. Now fix $k\in\nn$ and $y\in M\setminus\{0\}$. By the uniqueness of the Taylor expansion of real analytic functions, $\sigma_m^{-1}\,P_k(x,y)$ is the $k$-homogenous component of the Taylor expansion of the function $B^{m+1}(0,\Vert y\Vert)\to M,\ x\mapsto E_m(y-x)$, which is left-monogenic with respect to $\B$ by Lemma~\ref{lem:cauchykernel}. Let us prove that the function $B^{m+1}(0,\Vert y\Vert)\to M,\ x\mapsto P_k(x,y)$ is left-monogenic with respect to $\B$. Following~\cite[Lemma 11.3.3]{librosommen}, for all $\h\in\nn^{m+1}$ with $|\h|=k-1$, we remark that
\[0\equiv\nabla_\B^\h0=\nabla_\B^\h\debar_\B^x E_m(y-x)=\debar_\B^x \nabla_\B^\h E_m(y-x)\]
whence, comparing constant terms,
\[0=\debar_\B^x \nabla_\B^\h P_k(x,y)=\nabla_\B^\h\debar_\B^x P_k(x,y)\,.\]
Since $x\mapsto\debar_\B^x P_k(x,y)$ is a $(k-1)$-homogeneous polynomial function and the last chain of equalities is true for arbitrary $\h$ with $|\h|=k-1$, we conclude that $\debar_\B^x P_k(x,y)\equiv0$, as desired.

By Proposition~\ref{prop:basismonogenicpolynomials}, for any $y\in M\setminus\{0\}$ there exists a finite sequence $\{a_\k(y)\}_{|\k|=k}\subset A$ such that
\[P_k(x,y)=\sum_{|\k|=k}\P_\k^\B(x)\,a_\k(y)\]
for all $x\in B^{m+1}(0,\Vert y\Vert)$. Now, set $q_\k:M\setminus\{0\}\to A,\  y\mapsto a_\k(y)\,\Vert y\Vert^{m+2k+1}$ for all $\k\in\nn^m$ with $|\k|=k$, so that
\[P_k(x,y)\,\Vert y\Vert^{m+2k+1}=\sum_{|\k|=k}\P_\k^\B(x)\,a_\k(y)\,\Vert y\Vert^{m+2k+1}=\sum_{|\k|=k}\P_\k^\B(x)\,q_\k(y)\,.\]
We recall that $P_k(x,y)\,\Vert y\Vert^{m+2k+1}$ is a polynomial function $\Lambda\to M$, which is $k$-homogenous in $x_0,x_1,\ldots,x_m$ and $(k+1)$-homogeneous in $y_0,y_1,\ldots,y_m$. For any $\k'\in\nn^m$ with $|\k'|=k$ we can apply the differential operator $\nabla_\B^{(0,\k')}$ to $x\mapsto P_k(x,y)\,\Vert y\Vert^{m+2k+1}$ and remark that the expression
\[\nabla_\B^{(0,\k')}(P_k(x,y)\,\Vert y\Vert^{m+2k+1})=\sum_{|\k|=k}\nabla_\B^{(0,\k')}\P_\k^\B(x)\,q_\k(y)=\k'!\,q_{\k'}(y)\]
still defines a polynomial function $\Lambda\to M$, now $0$-homogenous in $x_0,x_1,\ldots,x_m$ but still $(k+1)$-homogeneous in $y_0,y_1,\ldots,y_m$. For the last equality, we applied Remark~\ref{rmk:propertiesfueterpolynomials}. We conclude, as desired, that $q_{\k'}$ is a polynomial function $M\setminus\{0\}\to M$, which is $(k+1)$-homogeneous in $y_0,y_1,\ldots,y_m$.
\end{proof}

\begin{theorem}[Integral formula for $\nabla_\B^\h \phi$]
Assume $A$ to be associative.  Fix a domain $G$ in the hypercomplex subspace $M$ of $A$ and a function $\phi:G\to A$ that is left-monogenic with respect to $\B$. Then $\phi$ is harmonic with respect to $\B$ and real analytic. For every $\h\in\nn^{m+1}$: the function $\nabla_\B^\h \phi$ is still left-monogenic with respect to $\B$ and real analytic; given any open ball $B^{m+1}=B^{m+1}(p,R)$ whose closure $\overline{B}^{m+1}$ is contained in $G$,
\[\nabla_\B^\h \phi(x)=(-1)^{|\h|}\int_{\partial B^{m+1}}\left(\nabla_\B^{\h}E_m\right)(y-x)\,dy^*\,\phi(y)\]
for all $x\in B^{m+1}$; and, at the center $p$ of the ball $B^{m+1}$,
\[\Vert\nabla_\B^\h \phi(p)\Vert\leq\frac{C_m}{R^{|\h|}}\,\max_{\partial B^{m+1}}\Vert \phi\Vert\,,\qquad C_m:=\sigma_m\,\omega^2\,\max_{\partial B^{m+1}(0,1)}\Vert\nabla_\B^{\h}E_m\Vert\,.\]
\end{theorem}

\begin{proof}[Proof of Theorem~\ref{thm:integralformulaforderivatives}]
Let us list some properties of the reproducing kernel $\Lambda\to M,\ (x,y)\mapsto E_m(y-x)$:
\begin{itemize}
\item it is left- and right-monogenic with respect to $x$ and with respect to $y$ by Lemma~\ref{lem:cauchykernel};
\item it is real analytic by Theorem~\ref{thm:taylorkernel};
\item for any $\h\in\nn^{m+1}$, its $\h$-partial derivative with respect to the variable $x$, which is the map  $\Lambda\to M,\ (x,y)\mapsto(-1)^{|\h|}(\nabla_\B^\h E_m)(y-x)$, is still real analytic, as well as left- and right-monogenic with respect to $x$ and with respect to $y$ (because $\nabla_\B^\h$ commutes with $\debar_\B$).
\end{itemize}
For later use, let us also prove by induction on $\h\in\nn^{m+1}$ the property that
\[\psi_\h(x):=\Vert x\Vert^{m+2|\h|+1}\,\nabla_\B^{\h}E_m(x)\]
is an $(|\h|+1)$-homogeneous polynomial function. This property is clearly true for $\nabla_\B^{(0,0,\ldots,0)}E_m=E_m$, since $\Vert x\Vert^{m+1}\,E_m(x)=\frac1{\sigma_m}x^c$. If the property is true for $\nabla_\B^\h E_m$, then for $\h'=\h+\epsilon_s$ we compute
\[\nabla_\B^{\h'}E_m(x)=\partial_s\nabla_\B^\h E_m(x)=\partial_s\frac{\psi_\h(x)}{\Vert x\Vert^{m+2|\h|+1}}
=\frac{\Vert x\Vert^2\partial_s\psi_\h(x)-(m+2|\h|+1)x_s\psi_\h(x)}{\Vert x\Vert^{m+2|\h|+3}}\,.\]
Since $m+2|\h|+3=m+2|\h'|+1$, the property is also true for $\nabla_\B^{\h'}E_m$. The induction step is therefore complete.

Now fix $\h\in\nn^{m+1}$ and an open ball $B^{m+1}=B^{m+1}(p,R)$ whose closure $\overline{B}^{m+1}$ is contained in $G$. Recall that Corollary~\ref{cor:monogeniccauchy} provides the integral formula
\[\phi(x)=\int_{\partial B^{m+1}}E_m(y-x)\,dy^*\,\phi(y)\,,\]
valid for all $x\in B^{m+1}$. From the listed properties, we see that $\nabla_\B^\h \phi$ exists in $B^{m+1}$, that
\[\nabla_\B^\h \phi(x)=\int_{\partial B^{m+1}}(-1)^{|\h|}(\nabla_\B^\h E_m)(y-x)\,dy^*\,\phi(y)\]
for all $x\in B^{m+1}$ and that $\nabla_\B^\h \phi$ is left-monogenic in $B^{m+1}$. Since the choice of $\h$ and $B^{m+1}$ was arbitrary, in particular $\phi\in\mathscr{C}^2(G,A)$. An application of Remark~\ref{rmk:harmonic} now proves that $\phi:G\to A$ is harmonic with respect to $\B$ and real analytic.

Let us now prove the inequality appearing in the statement. For $w\in\partial B^{m+1}(0,1)$ and for $y=p+R w$, it follows from the first part of the proof that $(\nabla_\B^\h E_m)(y-p)=\nabla_\B^{\h}E_m(R w)=R^{-m-|\h|}\nabla_\B^{\h}E_m(w)$. Using Remark~\ref{rmk:sphere}, we compute $\nabla_\B^\h E_m$ at the center $p$ of $B^{m+1}$ as
\begin{align*}
\nabla_\B^\h \phi(p)&=(-1)^{|\h|}\int_{\partial B^{m+1}}(\nabla_\B^\h E_m)(y-p)\,dy^*\,\phi(y)\\
&=(-1)^{|\h|}\int_{\partial B^{m+1}(0,1)}R^{-m-|\h|}(\nabla_\B^\h E_m)(w)\,R^m\,w\,\phi(p+R w)\,|do_w|\\
&=(-R)^{-|\h|}\int_{\partial B^{m+1}(0,1)}(\nabla_\B^\h E_m)(w)\,w\,\phi(p+R w)\,|do_w|\,.
\end{align*}
Using Lemma~\ref{lem:surfaceintegral}, we find that
\begin{align*}
\Vert\nabla_\B^\h \phi(p)\Vert&\leq R^{-|\h|}\int_{\partial B^{m+1}(0,1)}\Vert\nabla_\B^{\h}E_m(w)\,w\,\phi(p+R w)\Vert\,|do_w|\\
&\leq R^{-|\h|}\cdot\max_{w\in\partial B^{m+1}(0,1)}\Vert\nabla_\B^{\h}E_m(w)\,w\,\phi(p+R w)\Vert\cdot\int_{\partial B^{m+1}(0,1)}|do_w|\\
&\leq R^{-|\h|}\,\omega\cdot\max_{w\in\partial B^{m+1}(0,1)}\Vert\nabla_\B^{\h}E_m(w)\Vert\cdot\max_{w\in\partial B^{m+1}}\Vert w\,\phi(w)\Vert\cdot\sigma_m\\
&= R^{-|\h|}\,\sigma_m\,\omega^2\cdot\max_{\partial B^{m+1}(0,1)}\Vert\nabla_\B^{\h}E_m\Vert\cdot\max_{\partial B^{m+1}}\Vert \phi\Vert\,,
\end{align*}
as desired. For the third and fourth inequalities, we applied Remark~\ref{rmk:norminequality} along with the fact that $\nabla_\B^{\h}E_m$ takes values in $M$.
\end{proof}

\setcounter{theorem}{20}

\begin{theorem}[Maximum Modulus Principle]
Assume $A$ to be associative. Fix a domain $G$ in the hypercomplex subspace $M$ of $A$ and a function $\phi:G\to A$, left-monogenic with respect to $\B$. If the function $\Vert\phi\Vert:G\to\rr$ has a global maximum point in $G$, then $\phi$ is constant in $G$.
\end{theorem}

\begin{proof}[Proof of Theorem~\ref{thm:monogenicmaximummodulus}]
Let $\mu:=\sup_G\Vert\phi\Vert$. Our hypothesis is that $\mu$ is finite and that $\mu=\max_G\Vert\phi\Vert$. We will first prove that $\Vert\phi\Vert\equiv\mu$ in $G$, then prove that $\phi$ itself is constant.

Using the continuity of $\Vert\phi\Vert$ and our hypothesis, we see that the level set $\L:=\{x\in G:\Vert\phi(x)\Vert=\mu\}$ is a nonempty closed subset of $G$. Moreover, this level set $\L$ must be open: for every $x$ such that $\Vert\phi(x)\Vert=\mu$ and every $R>0$ such that $B^{m+1}:=B^{m+1}(x,R)\subseteq G$, we can prove that $B^{m+1}$ is contained in $\L$. Indeed, for any $r$ with $0<r<R$, if there existed $w\in\partial B^{m+1}(0,1)$ such that $\Vert\phi(x+rw)\Vert<\mu$ (whence a spherical cap in $\partial B^{m+1}(0,1)$ where the same inequality holds true), then the Mean Value Property~\ref{prop:monogenicmeanvalue} would yield $\Vert\phi(x)\Vert<\mu$. Since $\L$ is a nonempty closed and open subset of the connected set $G$, we conclude that $\L=G$. In other words, $\Vert\phi\Vert\equiv\mu$ in $G$.

Let us express $\phi$ as $\phi=\sum_{u=0}^d\phi_uv_u$ with respect to the basis $\B'=\{v_0,v_1,\ldots,v_d\}$ of $A$. We know from Theorem~\ref{thm:integralformulaforderivatives} that $\Delta_\B\phi\equiv0$ and conclude that $\Delta_\B\phi_u\equiv0$ for each $u\in\{0,\ldots,d\}$. Moreover, the equality $\Vert\phi\Vert^2\equiv\mu^2$ reads as $\sum_{u=0}^d\phi_u^2\equiv\mu^2$. For each $s\in\{0,\ldots,m\}$, by applying $\partial_s$ to both hands of the equality, we find $2\sum_{u=0}^d\phi_u\partial_s\phi_u\equiv0$. Repeating the operation, we find that $2\sum_{u=0}^d\left((\partial_s\phi_u)^2+\phi_u\partial_s^2\phi_u\right)\equiv0$. Thus,
\[0\equiv\sum_{u=0}^d\left(\sum_{s=0}^m(\partial_s\phi_u)^2+\phi_u\Delta_\B\phi_u\right)=\sum_{u=0}^d\sum_{s=0}^m(\partial_s\phi_u)^2\,.\]
It follows that $\partial_s\phi_u\equiv0$ for all $s\in\{0,\ldots,m\}$ and all $u\in\{0,\ldots,d\}$, whence $\phi$ is constant.
\end{proof}


\subsection*{Proofs of the results in Subsection~\ref{subsec:seriesexpansionmonogenic}}

\begin{theorem}[Series expansion]
Assume $A$ to be associative. Fix a domain $G$ in the hypercomplex subspace $M$ of $A$ and a function $\phi:G\to A$ that is left-monogenic with respect to $\B$. In every open ball $B^{m+1}(p,R)$ contained in $G$, the following series expansion is valid:
\[\phi(x)=\sum_{k\in\nn}\sum_{|\k|=k}\P_\k^\B(x-p)\,a_\k\,,\quad a_\k=\frac{1}{\k!}\nabla_\B^{(0,\k)}\phi(p)\,.\]
Here, the series converges normally in $B^{m+1}(p,R)$ because
\[\max_{\Vert x-p\Vert\leq r_1}\Big\Vert\sum_{|\k|=k}\P_\k^\B(x-p)\,a_\k\Big\Vert\leq\omega^2\,\sqrt{2}\;\binom{k+m}{m}\,\left(\frac{r_1}{r_2}\right)^k\,\max_{\Vert y-p\Vert=r_2}\Vert\phi(y)\Vert\]
whenever $0<r_1<r_2<R$.
\end{theorem}

\begin{proof}[Proof of Theorem~\ref{thm:taylormonogenic}]
Let us fix an open ball $B^{m+1}(p,R)$ contained in $G$. Since the kernel of $\debar_\B$ is invariant under composition with translations, we may assume without loss of generality $p=0$. Pick any $r$ with $0<r<R$ and set $B^{m+1}:=B^{m+1}(p,r)$, which has $\overline{B}^{m+1}\subset G$. Using Corollary~\ref{cor:monogeniccauchy} and Theorem~\ref{thm:taylorkernel}, we see that
\begin{align*}
\phi(x)&=\int_{\partial B^{m+1}}E_m(y-x)\,dy^*\,\phi(y)\\
&=\frac1{\sigma_m}\int_{\partial B^{m+1}}\Big(\sum_{k\in\nn}\sum_{|\k|=k}\P_\k^\B(x)\frac{q_\k(y)}{\Vert y\Vert^{m+2k+1}}\Big)\,dy^*\,\phi(y)\\
&=\sum_{k\in\nn}\sum_{|\k|=k}\P_\k^\B(x)\,a_\k
\end{align*}
for all $x\in B^{m+1}$, where
\[a_\k:=\frac1{\sigma_m}\int_{\partial B^{m+1}}\frac{q_\k(y)}{\Vert y\Vert^{m+2k+1}}\,dy^*\,\phi(y)\,.\]
In the previous chain of equalities: the last equality is true, and the series $\sum_{k\in\nn}\sum_{|\k|=k}\P_\k^\B\,a_\k$ converges normally in $B^{m+1}$, because of the following argument. For all $w\in\partial B^{m+1}(0,1)$, we recall that $q_\k(r w)=r^{|\k|+1}q_\k(w)$. Using Remark~\ref{rmk:sphere}, we compute
\begin{align*}
\sum_{|\k|=k}\P_\k^\B(x)\,a_\k&=\sum_{|\k|=k}\P_\k^\B(x)\,\frac1{\sigma_m}\int_{\partial B^{m+1}}\frac{q_\k(y)}{\Vert y\Vert^{m+2k+1}}\,dy^*\,\phi(y)\\
&=\frac1{\sigma_m}\int_{\partial B^{m+1}}\sum_{|\k|=k}\P_\k^\B(x)\frac{q_\k(y)}{\Vert y\Vert^{m+2k+1}}\,dy^*\,\phi(y)\\
&=\frac1{\sigma_m}\int_{\partial B^{m+1}(0,1)}\sum_{|\k|=k}\P_\k^\B(x)\,\frac{q_\k(w)}{r^k}\,w\,\phi(r w)\,|do_w|
\end{align*}
for all $x\in B^{m+1}$. Take any $r_1$ with $0<r_1<r$, set $C:=\overline{B}^{m+1}(0,r_1)$ and $T:=C\times\partial B^{m+1}\subset\Lambda=\{(x,y)\in M\times M:\Vert x\Vert<\Vert y\Vert\}$. Using Lemma~\ref{lem:surfaceintegral}, we remark:
\begin{align*}
\max_{x\in C}\Big\Vert\sum_{|\k|=k}\P_\k^\B(x)\,a_\k\Big\Vert
&\leq\frac1{\sigma_m}\,\int_{\partial B^{m+1}(0,1)}\max_{x\in C}\Big\Vert\sum_{|\k|=k}\P_\k^\B(x)\frac{q_\k(w)}{r^k}\,w\,\phi(r w)\Big\Vert\,|do_w|\\
&\leq\frac1{\sigma_m}\cdot\max_{(x,r w)\in T}\Big\Vert\sum_{|\k|=k}\P_\k^\B(x)\,\frac{q_\k(w)}{r^k}\,w\,\phi(r w)\Big\Vert\cdot\int_{\partial B^{m+1}(0,1)}|do_w|\\
&=\max_{(x,r w)\in T}\Big\Vert\sum_{|\k|=k}\P_\k^\B(x)\,\frac{q_\k(w)}{r^k}\,w\,\phi(r w)\Big\Vert\\
&\leq\omega\cdot\max_{(x,r w)\in T}\Big\Vert\sum_{|\k|=k}\P_\k^\B(x)\,\frac{q_\k(w)}{r^k}\Big\Vert\cdot\max_{rw\in\partial B^{m+1}}\Vert w\,\phi(rw)\Vert\\
&\leq\omega^2\cdot\max_{(x,y)\in T}\Big\Vert\sum_{|\k|=k}\P_\k^\B(x)\,\frac{q_\k(y)}{r^{2k+1}}\Big\Vert\cdot\max_{\partial B^{m+1}}\Vert\phi\Vert\\
&\leq\omega^2\cdot\sqrt{2}\;\binom{k+m}{m}\,\max_{(x,y)\in T}\left(\Vert x\Vert^k\,\Vert y\Vert^{-k}\right)\cdot\max_{\partial B^{m+1}}\Vert\phi\Vert\\
&=\omega^2\,\sqrt{2}\;\binom{k+m}{m}\,\left(\frac{r_1}{r}\right)^k\,\max_{\partial B^{m+1}}\Vert\phi\Vert\,.
\end{align*}
Here, we first used Remark~\ref{rmk:norminequality}, along with the fact that $\sum_{|\k|=k}\P_\k^\B(x)\,\frac{q_\k(w)}{r^k}$ takes values in $M$, then Theorem~\ref{thm:taylorkernel}. Remark~\ref{rmk:roottest} and the root test show that the number series
\[\sum_{k\in\nn}\max_{x\in C}\Big\Vert\sum_{|\k|=k}\P_\k^\B(x)\,a_\k\Big\Vert\]
converges. For any $r_1$ with $0<r_1<R$, it is possible to choose $r_2$ such that $0<r_1<r_2<R$. Thus, the series $\sum_{k\in\nn}\sum_{|\k|=k}\P_\k^\B(x)\,a_\k$ converges normally in $B^{m+1}(0,R)$, as desired.

We are left with proving that $a_\k=\frac{1}{\k!}\nabla_\B^{(0,\k)}\phi(p)$ for all $\k\in\nn^m$. Fix $\k'\in\nn^m$: Theorem~\ref{thm:integralformulaforderivatives} guarantees that $\nabla_\B^{(0,\k')}\phi$ exists and is still left-monogenic with respect to $\B$, as well as real analytic. By the first part of the proof, there exists a sequence $\{a'_{\k}\}_{\k\in\nn^m}\subset A$ such that $\nabla_\B^{(0,\k')}\phi$ expands throughout $B^{m+1}(0,R)$ into the normally convergent series $\sum_{k\in\nn}\sum_{|\k|=k}\P_\k^\B(x)\,a'_\k$. In particular, $a'_{(0,\ldots,0)}=\nabla_\B^{(0,\k')}\phi(0)$. By the uniqueness of the Taylor expansions of real analytic functions, $a'_{(0,\ldots,0)}$ can be obtained by applying $\nabla_\B^{(0,\k')}$ to the $|\k'|$-homogenous component $\sum_{|\k|=|\k'|}\P_\k^\B(x)\,a_\k$ of the expansion $\phi(x)=\sum_{k\in\nn}\sum_{|\k|=k}\P_\k^\B(x)\,a_\k$. Therefore,
\begin{align*}
a'_{(0,\ldots,0)}&\equiv\nabla_\B^{(0,\k')}\sum_{|\k|=|\k'|}\P_\k^\B(x)\,a_\k\\
&=\sum_{|\k|=|\k'|}\nabla_\B^{(0,\k')}\P_\k^\B(x)\,a_\k\\
&\equiv\k'!\,a_{\k'}\,,
\end{align*}
where the last equality follows from Remark~\ref{rmk:propertiesfueterpolynomials}. Thus, $a_{\k'}=\frac{a'_{(0,\ldots,0)}}{\k'!}=\frac{1}{\k'!}\nabla_\B^{(0,\k')}\phi(0)$, as desired.
\end{proof}

\begin{theorem}[Identity Principle]
Assume $A$ to be associative. Fix a domain $G$ in the hypercomplex subspace $M$ of $A$ and functions $\phi,\psi:G\to A$ that are left-monogenic with respect to $\B$. If $G$ contains a set of Hausdorff dimension $n\geq m$ where $\phi$ and $\psi$ coincide, then $\phi=\psi$ throughout $G$.
\end{theorem}

\begin{proof}[Proof of Theorem~\ref{thm:identitymonogenic}]
Let $S$ denote the set of all points of $G\subseteq M$ where $\phi$ and $\psi$ coincide, i.e., the zero set of $\chi:=\phi-\psi$. Seeking a contradiction, we assume that neither the equality $S=G$ nor the inequality $\dim_\rr(S)<m$ hold true. Since $\dim_\rr(M)=m+1$ and since $\chi$ is a real analytic function, it follows that $\dim_\rr(S)=m$ and that there exists an open ball $B^{m+1}\subseteq M$ such that $S\cap B^{m+1}$ is a real analytic hypersurface of $B^{m+1}$. We will prove that the zero set of $\chi$ includes an open ball, thus reaching a contradiction with our hypotheses.

Taking into account Theorem~\ref{thm:taylormonogenic}, which provides a series expansion of $\chi$ centered at any $p\in G$ with coefficients $\{\frac1{\k!}\nabla_\B^{(0,\k)}\chi(p)\}_{\k\in\nn^m}$, it suffices to prove that $\nabla_\B^{\h}\chi$ vanishes identically in $S\cap B^{m+1}$ for all $\h\in\nn^{m+1}$. We do so by induction on $|\h|$. The induction basis is the fact that $\chi$ vanishes identically in $S\cap B^{m+1}$. The induction step from $h$ to $h+1$ consists in assuming $\nabla_\B^\h \chi$ to vanish in $S\cap B^{m+1}$ for all $\h\in\nn^{m+1}$ with $|\h|=h$ and proving that $\nabla_\B^{\h'} \chi$ vanishes in $S\cap B^{m+1}$ for all $\h'\in\nn^{m+1}$ with $|\h'|=h+1$. This is the same as proving that, for any $\h\in\nn^{m+1}$ with $|\h|=h$ and any $p \in S\cap B^{m+1}$, the vector
\[w:=\begin{pmatrix}
(\partial_0\nabla_\B^\h \chi)(p)\\
(\partial_1\nabla_\B^\h \chi)(p)\\
\vdots\\
(\partial_m\nabla_\B^\h \chi)(p)
\end{pmatrix}\in A^{m+1}\]
is the zero vector. Since $\nabla_\B^\h \chi$ vanishes identically on the hypersurface $S\cap B^{m+1}$, there exists an $m\times (m+1)$ matrix $\mathcal{A}$ of rank $m$, with entries in $\rr\subset A$, such that $\mathcal{A}\,w=0\in A^m$. Thus, there exist $n\in\{0,\ldots,m\}$ and $c_0,\ldots,c_m\in\rr$ (with $c_n=1$) such that
\[w=\begin{pmatrix}
c_0\\
c_1\\
\vdots\\
c_m
\end{pmatrix}(\partial_n\nabla_\B^\h \chi)(p)\,.\]
Now, since $\nabla_\B^\h \chi$ is still left-monogenic,
\[0=(\debar_\B \nabla_\B^\h \chi)(p)=\sum_{s=0}^mv_s(\partial_s\nabla_\B^\h \chi)(p)=\left(\sum_{s=0}^mv_sc_s\right)(\partial_n\nabla_\B^\h \chi)(p)\,.\]
Now, $\sum_{s=0}^mv_sc_s$ belongs to $M$ and is not zero because $c_n=1$. We conclude that $\sum_{s=0}^mv_sc_s$ is not a left zero divisor in $A$ and that $(\partial_n\nabla_\B^\h \chi)(p)=0$. Thus, $w$ is the zero vector in $A^{m+1}$, as desired.
\end{proof}

\newpage

\section*{Acknowledgements}

Both authors are partly supported by: GNSAGA INdAM; Progetto “Teoria delle funzioni ipercomplesse e applicazioni” Università di Firenze. The second author is also partly supported by: PRIN 2022 “Real and complex manifolds: geometry and holomorphic dynamics” (2022AP8HZ9) MIUR; Finanziamento Premiale “Splines for accUrate NumeRics: adaptIve models for Simulation Environments” INdAM.\\
The authors warmly thank the anonymous reviewer for providing precious suggestions, which allowed to improve the presentation of this work.





\begin{thebibliography}{10}

\bibitem{altavillawithoutreal}
A.~Altavilla.
\newblock Some properties for quaternionic slice regular functions on domains without real points.
\newblock {\em Complex Var. Elliptic Equ.}, 60(1):59--77, 2015.

\bibitem{librosommen}
F.~Brackx, R.~Delanghe, and F.~Sommen.
\newblock {\em Clifford analysis}, volume~76 of {\em Research Notes in
  Mathematics}.
\newblock Pitman (Advanced Publishing Program), Boston, MA, 1982.

\bibitem{librocnops}
J.~Cnops and H.~Malonek.
\newblock {\em An introduction to {C}lifford analysis}.
\newblock Textos de Matem{\'a}tica. S{\'e}rie B [Texts in Mathematics. Series
  B], 7. Universidade de Coimbra Departamento de Matem{\'a}tica, Coimbra, 1995.

\bibitem{advancesrevised}
F.~Colombo, G.~Gentili, I.~Sabadini, and D.~Struppa.
\newblock Extension results for slice regular functions of a quaternionic
  variable.
\newblock {\em Adv. Math.}, 222(5):1793--1808, 2009.

\bibitem{librodaniele}
F.~Colombo, I.~Sabadini, F.~Sommen, and D.~C. Struppa.
\newblock {\em Analysis of {D}irac systems and computational algebra},
  volume~39 of {\em Progress in Mathematical Physics}.
\newblock Birkh\"auser Boston Inc., Boston, MA, 2004.

\bibitem{israel}
F.~Colombo, I.~Sabadini, and D.~C. Struppa.
\newblock Slice monogenic functions.
\newblock {\em Israel J. Math.}, 171:385--403, 2009.

\bibitem{librodaniele2}
F.~Colombo, I.~Sabadini, and D.~C. Struppa.
\newblock {\em Noncommutative functional calculus. Theory and applications of
  slice hyperholomorphic functions}, volume 289 of {\em Progress in
  Mathematics}.
\newblock Birkh{\"a}user/Springer Basel AG, Basel, 2011.

\bibitem{sce}
P.~Dentoni and M.~Sce.
\newblock Funzioni regolari nell'algebra di {C}ayley.
\newblock {\em Rend. Sem. Mat. Univ. Padova}, 50:251--267 (1973).

\bibitem{ebbinghaus}
H.-D. Ebbinghaus, H.~Hermes, F.~Hirzebruch, M.~Koecher, K.~Mainzer,
  J.~Neukirch, A.~Prestel, and R.~Remmert.
\newblock {\em Numbers}, volume 123 of {\em Graduate Texts in Mathematics}.
\newblock Springer-Verlag, New York, 1990.
\newblock With an introduction by K. Lamotke, Translated from the second German
  edition by H. L. S. Orde, Translation edited and with a preface by J. H.
  Ewing, Readings in Mathematics.

\bibitem{fueter1}
R.~Fueter.
\newblock Die {F}unktionentheorie der {D}ifferentialgleichungen {$\Delta u=0$}
  und {$\Delta\Delta u=0$} mit vier reellen {V}ariablen.
\newblock {\em Comment. Math. Helv.}, 7(1):307--330, 1934.

\bibitem{fueter2}
R.~Fueter.
\newblock \"{U}ber die analytische {D}arstellung der regul\"aren {F}unktionen
  einer {Q}uaternionenvariablen.
\newblock {\em Comment. Math. Helv.}, 8(1):371--378, 1935.

\bibitem{librospringer2}
G.~Gentili, C.~Stoppato, and D.~C. Struppa.
\newblock {\em Regular functions of a quaternionic variable}.
\newblock Springer Monographs in Mathematics. Springer, Cham, 2022.
\newblock Second edition.

\bibitem{zerosopen}
G.~Gentili and C.~Stoppato.
\newblock The zero sets of slice regular functions and the open mapping theorem
\newblock In: I.~Sabadini and F.~Sommen, editors, {\em Hypercomplex analysis and applications}, Trends in Mathematics, pages 95--107. Birkh\"auser Verlag, Basel, 2011.

\bibitem{cras}
G.~Gentili and D.~C. Struppa.
\newblock A new approach to {C}ullen-regular functions of a quaternionic
  variable.
\newblock {\em C. R. Math. Acad. Sci. Paris}, 342(10):741--744, 2006.

\bibitem{advances}
G.~Gentili and D.~C. Struppa.
\newblock A new theory of regular functions of a quaternionic variable.
\newblock {\em Adv. Math.}, 216(1):279--301, 2007.

\bibitem{rocky}
G.~Gentili and D.~C. Struppa.
\newblock Regular functions on the space of {C}ayley numbers.
\newblock {\em Rocky Mountain J. Math.}, 40(1):225--241, 2010.

\bibitem{ghilonislicebyslice}
R.~Ghiloni.
\newblock Slice-by-slice and global smoothness of slice regular and
  polyanalytic functions.
\newblock {\em Ann. Mat. Pura Appl. (4)}, 201(5):2549--2573, 2022.

\bibitem{perotti}
R.~Ghiloni and A.~Perotti.
\newblock Slice regular functions on real alternative algebras.
\newblock {\em Adv. Math.}, 226(2):1662--1691, 2011.

\bibitem{volumeintegral}
R.~Ghiloni and A.~Perotti.
\newblock Volume {C}auchy formulas for slice functions on real associative
  *-algebras.
\newblock {\em Complex Var. Elliptic Equ.}, 58(12):1701--1714, 2013.

\bibitem{gpseveral}
R.~Ghiloni and A.~Perotti.
\newblock Slice regular functions in several variables.
\newblock {\em Math. Z.}, 302(1):295--351, 2022.

\bibitem{gpsalgebra}
R.~Ghiloni, A.~Perotti, and C.~Stoppato.
\newblock The algebra of slice functions.
\newblock {\em Trans. Amer. Math. Soc.}, 369(7):4725--4762, 2017.

\bibitem{unifiednotion}
R.~Ghiloni and C.~Stoppato.
\newblock A unified notion of regularity in one hypercomplex variable.
\newblock {\em J. Geom. Phys.}, 202:Paper No. 105219, 2024.

\bibitem{grashteynryzhik}
I.~S. Gradshteyn and I.~M. Ryzhik.
\newblock {\em Table of integrals, series, and products}.
\newblock Academic Press, Inc., San Diego, CA, fifth edition, 1996.
\newblock CD-ROM version 1.0 for PC, MAC, and UNIX computers.

\bibitem{librogurlebeck2}
K.~G{{\"u}}rlebeck, K.~Habetha, and W.~Spr{{\"o}}{\ss}ig.
\newblock {\em Holomorphic functions in the plane and {$n$}-dimensional space}.
\newblock Birkh{\"a}user Verlag, Basel, 2008.
\newblock Translated from the 2006 German original, With 1 CD-ROM (Windows and
  UNIX).

\bibitem{krausshardifferentialtopological}
R.~S. Krau{\ss}har.
\newblock Differential topological aspects in octonionic monogenic function
  theory.
\newblock {\em Adv. Appl. Clifford Algebr.}, 30(4):Paper No. 51, 25, 2020.

\bibitem{malonekhabilitation}
H.~R. Malonek.
\newblock Zum holomorphiebegriff in h{\"o}heren dimensionen.
\newblock Habilitationsschrift. P{\"a}dagogische Hochschule Halle, 1987.

\bibitem{moisilteodorescu}
G.~C. Moisil and N.~Teodorescu.
\newblock Fonctions holomorphes dans l'espace.
\newblock {\em Mathematica (Cluj)}, 5:142--159, 1931.

\bibitem{perotticr}
A.~Perotti.
\newblock Cauchy-{R}iemann operators and local slice analysis over real
  alternative algebras.
\newblock {\em J. Math. Anal. Appl.}, 516(1):Paper No. 126480, 34, 2022.

\bibitem{perticiseveral}
D.~Pertici.
\newblock Regular functions of several quaternionic variables.
\newblock {\em Ann. Mat. Pura Appl. (4)}, 151:39--65, 1988.

\bibitem{renxu}
G.~Ren and Z.~Xu.
\newblock Schwarz's lemma for slice Clifford analysis.
\newblock {\em Adv. Appl. Clifford Algebr.}, 25(4):965--976, 2015.

\bibitem{schafer}
R.~D. Schafer.
\newblock {\em An introduction to nonassociative algebras}.
\newblock Pure and Applied Mathematics, Vol. 22. Academic Press, New York,
  1966.

\bibitem{poli}
C.~Stoppato.
\newblock Poles of regular quaternionic functions.
\newblock {\em Complex Var. Elliptic Equ.}, 54(11):1001--1018, 2009.   

\bibitem{sudbery}
A.~Sudbery.
\newblock Quaternionic analysis.
\newblock {\em Math. Proc. Cambridge Philos. Soc.}, 85(2):199--224, 1979.

\bibitem{libroszego}
G.~Szeg\"{o}.
\newblock {\em Orthogonal {P}olynomials}.
\newblock American Mathematical Society Colloquium Publications, Vol. 23.
  American Mathematical Society, New York, 1939.
  
\bibitem{wang}
X.~Wang.
\newblock On geometric aspects of quaternionic and octonionic slice regular
  functions.
\newblock {\em J. Geom. Anal.}, 27(4):2817--2871, 2017.

\bibitem{whitney}
H.~Whitney.
\newblock Differentiable even functions.
\newblock {\em Duke Math. J.}, 10:159--160, 1943.

\bibitem{xsannouncement}
Z.~Xu and I.~Sabadini.
\newblock Generalized partial-slice monogenic functions: a synthesis of two
  function theories.
\newblock {\em Adv. Appl. Clifford Algebr.}, 34(2):Paper No. 10, 14, 2024.

\bibitem{xsgeneralizedpartialslice}
Z.~Xu and I.~Sabadini.
\newblock Generalized partial-slice monogenic functions.
\newblock {\em Trans. Amer. Math. Soc.}, 378(2):851--883, 2025.

\bibitem{xsfuetersce}
Z.~Xu and I.~Sabadini.
\newblock On the {F}ueter-{S}ce theorem for generalized partial-slice monogenic
  functions.
\newblock {\em Ann. Mat. Pura Appl. (4)}, 204(2):835--857, 2025.

\end{thebibliography}
\end{document}